\documentclass[a4paper,dvipsnames]{article}
\usepackage[utf8]{inputenc}
\usepackage[a4paper,top=3cm,bottom=2cm,left=3cm,right=3cm,marginparwidth=1.75cm]{geometry}
\usepackage{array}
\usepackage{booktabs} 
\usepackage{setspace} 
\usepackage{cancel}   
\usepackage{enumitem} 
\usepackage{etoolbox} 
\usepackage[margin=2cm]{caption} 
\usepackage{diagbox}

\usepackage[backend=biber, hyperref=true, style=numeric, sorting=nyt, doi=true, isbn=false, url=false, maxbibnames=9, sortcites=true]{biblatex}
\renewenvironment{abstract}
  {\quotation
  {\bfseries\noindent{\abstractname:}}}
  {\endquotation}

\usepackage{authblk} 

\usepackage{bbm}
\usepackage{amssymb,amsmath,amsthm}
\allowdisplaybreaks 
\usepackage{verbatim}
\usepackage[notref, notcite, final]{showkeys} 
\usepackage{mathtools}
\usepackage{enumitem}
\usepackage[colorlinks, bookmarksnumbered, bookmarks]{hyperref} 
\usepackage{xcolor}
\usepackage{txfonts}
\usepackage{lscape}
\usepackage{arydshln}
\usepackage{mathabx}

\theoremstyle{definition} 
\newtheorem{theorem}{Theorem}[section] 
\newtheorem{definition}[theorem]{Definition}
\newtheorem{algorithm}[theorem]{Algorithm}
\newtheorem{lemma}[theorem]{Lemma}
\newtheorem{proposition}[theorem]{Proposition}
\newtheorem{corollary}[theorem]{Corollary}
\newtheorem{assumption}[theorem]{Assumption}
\newtheorem{rmk_temp}[theorem]{Remark}
\numberwithin{equation}{section} 

\newenvironment{remark}
  {\pushQED{\qed}\begin{rmk_temp}}
  {\popQED\end{rmk_temp}}

\newcommand{\N}{\mathbb{N}}
\newcommand{\Z}{\mathbb{Z}}

\newcommand{\R}{\mathbb{R}}
\newcommand{\C}{\mathbb{C}}
\newcommand{\A}{\mathbb{A}}

\newcommand{\simgrad}{\sym\nabla}
\newcommand{\eps}{{\varepsilon}}

\DeclareMathOperator{\sym}{sym}

\newcommand{\vect}[1]{\boldsymbol #1}

\DeclareMathOperator{\supp}{supp}
\DeclareMathOperator{\range}{range}
\DeclareMathOperator{\eigenspace}{Eig}

\usepackage{graphicx}
\usepackage{empheq}
\usepackage[most]{tcolorbox}

\let\oldsqrt\sqrt
\def\sqrt{\mathpalette\DHLhksqrt}
\def\DHLhksqrt#1#2{%
\setbox0=\hbox{$#1\oldsqrt{#2\,}$}\dimen0=\ht0
\advance\dimen0-0.2\ht0
\setbox2=\hbox{\vrule height\ht0 depth -\dimen0}%
{\box0\lower0.4pt\box2}}

\AtBeginDocument{

}

\title{\LARGE\MakeUppercase{\textbf{Higher-order Homogenization for Equations of Linearized Elasticity using the Operator-Asymptotic Approach}}}

\def\correspondingauthor{\footnote{Corresponding author:
josip.zubrinic@fer.hr}}

\author[1]{Yi-Sheng Lim}
\author[2]{Josip Žubrinić\correspondingauthor{}}

\affil[1]{Department of Mathematics, Texas A\&M University, College Station, TX 77843, \newline United States (Email: yishenglimysl@tamu.edu)}

\affil[2]{Faculty of Electrical Engineering and Computing, University of Zagreb, Unska 3, 10000 Zagreb, Croatia (Email: josip.zubrinic@fer.hr)}

\date{}

\begin{document}

\maketitle

\vspace{-0.8cm}

\begin{abstract}
    The operator-asymptotic approach was introduced by Lim-Žubrinić in \cite{simplified_method} for the homogenization of an $\varepsilon\mathbb{Z}^d$-periodic composite media. In this article, we consider the setting of three-dimensional linearized elasticity, and extend the approach to obtain higher-order convergence rates. In particular, we consider the so-called ``Bloch approximation" for vector-valued functions with compact Fourier support, and demonstrate that under such data, the approach provides an expansion that yields an error of order $\varepsilon^{n+1}$ in $L^2$ and $\varepsilon^n$ in $H^1$, for any $n$.

    \vskip 0.5cm
    
    \vskip 0.5cm

\end{abstract}


\onehalfspacing
\section{Introduction}\label{sect:introduction}
In this article, we consider the equations of linearized elasticity on $\R^3$:
\begin{align}\label{eqn:intro_main_equation}
    (\simgrad)^\ast \left( \A\left(\tfrac{\cdot}{\eps}\right) \simgrad \vect u_\eps \right) + \vect u_\eps = \vect f, \quad \vect f \in L^2(\R^3;\C^3), \quad \eps>0.
\end{align}
The unknown $\vect u_\eps$ denotes the displacement of the medium from some reference position, $\simgrad \vect u = \frac{1}{2} ( \nabla \vect u + \nabla \vect u^\top ) $ denotes the linearized strain tensor,  $(\simgrad)^\ast$ is the $L^2-$adjoint of $\simgrad$, and $\A$ is a $\Z^3-$periodic fourth order tensor-valued function describing the properties of the medium. We shall impose further assumptions on $\A$ (Assumption \ref{coffassumption}) so that \eqref{eqn:intro_main_equation} is well-posed. The rescaling to $\A_\eps(x) := \A\left(\tfrac{x}{\eps}\right)$ in \eqref{eqn:intro_main_equation} thus describes an $\eps\Z^3-$periodic medium, where $\eps>0$ measures the heterogeneity of the medium.

We are interested in the \textit{higher-order} homogenization problem for \eqref{eqn:intro_main_equation}. That is, to capture the behaviour of $\vect u_\eps$ as $\eps\downarrow 0$, to $\mathcal{O}(\eps^n)$ accuracy for any $n$. It is well-known (e.g.\,we mention the books \cite{cioranescu_donato,zhikov,bensoussan_lions_papanicolaou,Oleinik2012MathematicalPI}) that the leading-order behaviour is captured by $\vect u_{\hom}$, which is the unique solution to
\begin{align}\label{eqn:intro_homo_equation}
    (\simgrad)^\ast \left( \A^{\hom} \simgrad \vect u_{\hom} \right) + \vect u_{\hom} = \vect f,
\end{align}
where $\A^{\hom}$ is some tensor that is constant in space. In fact, $\vect u_\eps$ converges to $\vect u_{\hom}$ in $L^2$ at an $\mathcal{O}(\eps)$ rate. Thus, the object's displacement in the highly-heterogeneous medium \eqref{eqn:intro_main_equation} is approximated by its displacement in a homogeneous one \eqref{eqn:intro_homo_equation}. We mention a few ways that the convergence $\vect u_\eps \rightarrow \vect u_{\hom}$ can be shown (which may not have an error rate): the two-scale asymptotic expansion method \cite[Chapter II]{Oleinik2012MathematicalPI}-\cite[Chapter 3.6]{shen_book}, the oscillating test functions method \cite[Chapter 10]{cioranescu_donato}, two-scale convergence \cite[Chapter 9]{cioranescu_donato}, and G-convergence \cite[Chapter 12]{zhikov}. This list is non-exhaustive, and growing at the point of writing.

Recently, an ``operator-asymptotic" approach was introduced\footnote{First appeared in its rough form by Cherednichenko and Vel\v{c}i\'c for the study of periodic thin elastic plates \cite{kirill_igor_plates}.} by the authors in \cite{simplified_method}, where it was shown in particular that:
\begin{theorem}
    [Special case of \texorpdfstring{\cite[Theorem 2.15]{simplified_method}}{Theorem 2.15}]
    \label{thm:simplified_method}
    There exist $C>0$, independent of $\eps>0$ such that
    \begin{align}
        &\left\| \left( \mathcal{A}_\eps + I \right)^{-1} - \left( \mathcal{A}^{\hom} + I \right)^{-1} \right\|_{L^2(\R^3;\R^3) \rightarrow L^2(\R^3;\R^3)} \leq C \eps, \label{eqn:simplified_mthd_l2l2}\\
        &\left\| \left( \mathcal{A}_\eps + I \right)^{-1} - \left( \mathcal{A}^{\hom} + I \right)^{-1} - \mathcal{R}_{1,\eps} \right\|_{L^2(\R^3;\R^3) \rightarrow H^1(\R^3;\R^3)} \leq C \eps, \label{eqn:simplified_mthd_l2h1}\\
        &\left\| \left( \mathcal{A}_\eps + I \right)^{-1} - \left( \mathcal{A}^{\hom} + I \right)^{-1} - \mathcal{R}_{1,\eps} - \mathcal{R}_{2,\eps} \right\|_{L^2(\R^3;\R^3) \rightarrow L^2(\R^3;\R^3)} \leq C \eps^2, \label{eqn:simplified_mthd_l2l2higher}
    \end{align}
    where $\mathcal{R}_{1,\eps}$ and $\mathcal{R}_{2,\eps}$ are corrector operators defined through the procedure.
\end{theorem}

We have used here: $\mathcal{A}_\eps = (\simgrad)^\ast \A_\eps (\simgrad)$ and $\mathcal{A}^{\hom} = (\simgrad)^\ast \A^{\hom} (\simgrad)$. Note for instance, that \eqref{eqn:simplified_mthd_l2l2} may be written in as
\begin{align}\label{eqn:simplified_mthd_l2l2_fcts}
    \| \vect u_\eps - \vect u_{\hom} \|_{L^2} \leq C \eps \| \vect f \|_{L^2},
    \quad \text{where $C>0$ is independent of $\eps$ and $\vect f$.}
\end{align}
For this reason, \eqref{eqn:simplified_mthd_l2l2}-\eqref{eqn:simplified_mthd_l2l2higher} are typically referred to as ``norm-resolvent estimates" or ``operator/uniform estimates". Norm-resolvent estimates, as opposed to ``strong(-resolvent) estimates" $\| \vect u_\eps - \vect u_{\hom} \|_{L^2} \leq C(\vect f) \eps$, are desirable as it imply convergence of spectra $\sigma(\mathcal{A}_\eps) \rightarrow \sigma(\mathcal{A}^{\hom})$ (in the Hausdorff sense, on compact sets. See \cite[Sect.\,1.3.3]{ys_thesis}), and also $\| g(\mathcal{A}_\eps) - g(\mathcal{A}^{\hom}) \|_{L^2\rightarrow L^2} \rightarrow 0$ for $g$ continuous on $\R$ and vanishing at infinity \cite[Theorem VIII.20]{reed_simon4}.


\subsection{The operator-asymptotic method in comparison to existing spectral methods}

In the language of \cite{zhikov_pastukhova_2016_opsurvey}, the operator-asymptotic approach belongs to the class of ``spectral methods", where a (rescaled) Floquet-Bloch-Gelfand transform $\mathcal{G}_\eps$ \eqref{eqn:gelfand_transform} is used to decompose the $\eps\Z^3-$periodic operator $\mathcal{A}_\eps$ into a family of operators $\{ \frac{1}{\eps^2} \mathcal{A}_\chi \}_{\chi \in [-\pi,\pi)^3}$ on $L^2([0,1]^3;\C^3)$. The underlying theme among such methods is a focus on the behaviour of the spectra $\sigma(\mathcal{A}_\eps)$ near the bottom, and this corresponds to the bottom of the spectra $\sigma(\mathcal{A}_\chi)$, for small $|\chi|$'s. This is used to construct $\mathcal{A}^{\hom}$, and the specifics of the construction and its error estimates depends on the method. Notable spectral methods include the ``spectral germ" approach by Birman and Suslina \cite{birman_suslina_2004} and a ``Bloch approach" by Conca-Vanninathan \cite{conca_vanninathan1997}. We refer the reader to \cite[Sections 1.2-1.4]{simplified_method} for a detailed comparison of existing spectral methods.

While the result itself (with norm-resolvent estimates) is not new (see e.g.\,\cite{birman_suslina_2004, birman_suslina_2007_l2h1, birman_suslina_2006_l2l2higherorder}), Theorem \ref{thm:simplified_method} serves as a proof of concept for the operator-asymptotic approach to elliptic \textit{systems} of PDEs on $\R^d$. The fundamental difference in moving from a scalar problem to a system of $n$ PDEs, is the loss of simplicity of the lowest eigenvalue of $\mathcal{A}_{\chi=0}$ (Section \ref{sect:recap_spectral_analysis}). Since parameter $\chi$ is $d-$dimensional, this prohibits a direct use of analytic perturbation theory \cite{kato_book}. The spectral germ \cite{birman_suslina_2004} and Bloch approaches \cite{ganesh_vanninathan_2005}\footnote{Note that the main result \cite[Theorem 5.1]{ganesh_vanninathan_2005} is statement of weak $H^1$ convergence, and not of the form \eqref{eqn:simplified_mthd_l2l2}.} deal with this by introducing polar coordinates $\chi = t\theta$, $t \geq 0$ and $|\theta| = 1$, and perturbing the one-dimensional parameter $t$. In addition, one has to ensure that the estimates are uniform in $\theta$.
On the other hand, the operator-asymptotic approach differs from existing spectral approaches as it works with the \textit{sum} $P_\chi^1 + \cdots + P_\chi^n$ of the spectral projections, instead of the individual $P_\chi^i$'s. By studying the lowest $n$ eigenspaces as a collective, one is able to avoid the use of analytic perturbation theory.

\subsection{Higher-order homogenization and the need for well-prepared data}

The term ``higher-order homogenization" refers to an extension of \eqref{eqn:simplified_mthd_l2l2}-\eqref{eqn:simplified_mthd_l2l2higher} to $\mathcal{O}(\eps^n)$ error for any $n \in \N$, by adding higher-order terms. Among the approaches available, the two-scale expansion has been the most successful at achieving this for a wide class of PDEs. Indeed, details can be found in the classical monographs \cite[Chapter 4]{bakhvalov_panasenko} and \cite[Chapter II]{Oleinik2012MathematicalPI}. It should be noted however, that the error estimates therein are given in the ``strong" form, where the constant $C$ now depends on $\vect f$ (and also $n$).

The fact that we obtain only strong (as opposed to norm) resolvent estimates, is not an artefact of the proofs \cite{bakhvalov_panasenko, Oleinik2012MathematicalPI}, but rather an inherent limitation of the homogenized approximation $(\mathcal{A}^{\hom} + I)^{-1}$. More precisely, it is now well-understood (see e.g.\,\cite{birman_suslina_2004, birman_suslina_2007_l2h1}) that the $L^2\rightarrow L^2$ $\mathcal{O}(\eps^2)$ error \eqref{eqn:simplified_mthd_l2l2higher} and the $L^2\rightarrow H^1$ $\mathcal{O}(\eps)$ error \eqref{eqn:simplified_mthd_l2h1} cannot be improved by adding higher-order terms. For higher-order homogenization, it is thus necessary to impose conditions on the data $\vect f \in L^2$.

By inspecting the two-scale expansion \cite{bakhvalov_panasenko, Oleinik2012MathematicalPI}, one naturally arrives at the condition $\vect f \in H^\infty$. This smoothness conditions plays two roles: First, it ensures that for each $n$, the terms in the $n-$th order expansion are well-defined. Second, it provides the desired $\mathcal{O}(\eps^n)$ estimate. Interestingly, we note that by further imposing conditions on the \textit{growth} of the $H^k$ norm of $\vect f$ in $k$ (the so-called ``Gevrey regular" functions), Kamotski, Matthies, and Smyshlyaev demonstrated \cite{kamotski_matthies_smyshlyaev_exphomo_2007} that an \textit{exponential} error may be attained, by truncating within a range of optimal $n$'s.

On the other hand, a separate set of conditions are adopted by Lamacz-Keymling and Yosept \cite{lamacz_yousept_2021} for higher-order homogenization of scalar problems, under the Bloch approach. In particular, they restricted themselves to the so-called ``Bloch approximation/adapted" functions $\vect f_\eps$ of some $\vect f \in L^2$ whose Fourier transform is compactly supported. Since the operator-asymptotic approach is spectral method, we shall impose a condition that is similar to this, adapted to vector-valued functions (Section \ref{sect:well_prepared_data}).

On a separate note, we remark that higher-order homogenization of \eqref{eqn:intro_main_equation} also serves as an intermediate step towards understanding the homogenization of the corresponding wave equation for times beyond $\mathcal{O}(\eps^{-1})$, under the operator-asymptotic method. For further details, we refer the reader to \cite{lamacz_1D, allaire_lamacz_rauch_2022_crime_pays, allaire_briane_vanninathan_2016, benoit_gloria_2019_ballistic_transport}, and a recent survey \cite{kirill_ys_wave_survey} by Cherednichenko and the first author on the homogenization of the wave equation.

\subsection{Goal of the article}

This article aims to further develop the operator-asymptotic approach, by continuing expansion $(\mathcal{A}^{\hom} + I)^{-1} + \mathcal{R}_{1,\eps} + \mathcal{R}_{2,\eps}$ so that the errors in Theorem \ref{thm:simplified_method} can be upgraded to order $\mathcal{O}(\eps^n)$, for any $n \in \N$. As mentioned above, it is necessary to restrict the data $\vect f$ to a subset of $L^2$. We therefore introduce and motivate the so-called Bloch-approximation $\widehat{\Xi}_{\eps,\mu} \vect f$ of a vector-valued function $\vect f \in L^2$ (Definition \ref{defn:bloch_approximation}), and we shall further assume that $\vect f$ has a compactly supported Fourier transform.
Under this class of functions, we demonstrate that the operator-asymptotic approach produces an (operator) expansion that is close to $(\mathcal{A}_\eps+I)^{-1}$ in the $L^2$ and $H^1$ sense with $\mathcal{O}(\eps^n)$ error, for any $n$ (Theorem \ref{thm:main_result_intro}).
%
Among existing spectral methods, this is the first higher-order homogenization result for vector-valued problems.

\subsection{Structure of the article}

This article is structured as follows: Section \ref{sect:prelims}, we introduce some basic notation, define the (rescaled) Gelfand transform $\mathcal{G}_\eps$ (Section \ref{sect:notation}), state the main assumption (Assumption \ref{coffassumption}), and define the key operator of interest $\mathcal{A}_\eps$ (Definition \ref{defn:main_operator_fullspace}) and $\mathcal{A}_\chi, \chi \in [-\pi,\pi)^3$ (Definition \ref{defn:main_operator_fiberspace}). The family $\{ \frac{1}{\eps^2} \mathcal{A}_\chi \}_{\chi \in [-\pi,\pi)^3}$ are the fibers of $\mathcal{A}_\eps$ under $\mathcal{G}_\eps$ (Proposition \ref{prop:pass_to_unitcell}).
The main results of this paper can be found in Section \ref{sect:main_results}.

Section \ref{sect:recap_spectral_analysis} begins with a brief review of the spectral analysis of the operators $\mathcal{A}_\chi$ on $L^2([0,1]^3;\C^3)$. In Section \ref{sect:well_prepared_data}, we introduce and motivate the notion of a ``Bloch approximation" $\widehat{\Xi}_{\eps,\mu}\vect f$ of a given $\vect f \in L^2$.

Section \ref{sect:fibrewise_asmptotics} details the asymptotic procedure for $(\frac{1}{|\chi|^2} \mathcal{A}_\chi - zI)^{-1}$. We extend \cite{simplified_method} by carrying out the procedure for $n$ cycles (iterations), for any $n$. Next, we show a structural result, namely that the terms $\vect u_j^{(k)}$ defined through the procedure solve certain well-posed problems, and the error term $\vect u_\text{error}^{(n)}$ satisfies a given variational problem (Section \ref{sect:algebraic_step}). Section \ref{sect:fibrewise_estimates_working} contains the analytic step, where the difference between $(\frac{1}{|\chi|^2} \mathcal{A}_\chi - zI)^{-1}$ and the expansion is estimated. The result is summarized in Section \ref{sect:fiberwise_results}.

Section \ref{sect:fullspace_estimates} brings the fiberwise (for each $\chi$) estimates of Section \ref{sect:fibrewise_asmptotics} back to $L^2(\R^3;\C^3)$. We prove the main result (Theorem \ref{thm:main_result_intro}) in two parts, Theorem \ref{thm:l2tol2} (Section \ref{sect:proof_l2tol2}), and Theorem \ref{thm:l2toh1} (Section \ref{sect:proof_l2toh1}).

\section{Preliminaries, problem setup, and main results}\label{sect:prelims}

\subsection{Notation}\label{sect:notation}

\paragraph*{Basic notation.} Let $d \in \N \in \{ 1, 2, \cdots \}$ and $\N_0 = \N \cup \{ 0 \}$. For $\vect a, \vect b \in \C^d$, write $\langle \vect a, \vect b \rangle_{\C^d} = \vect a \cdot \overline{\vect b}$ for the inner product, with corresponding norm $|\vect a|$. Write $\Re{(\vect a)}$ and $\Im{(\vect a)}$ for the real and imaginary part of $\vect a$ respectively, taken component-wise, and $\vect a \otimes \vect b = \vect a \overline{\vect b}^\top$. For matrices $\vect A = (A_j^i), \vect B = (B_j^i) \in \C^{d\times d}$, $A_j^i$ refers to the $i$'th row and $j$'th column of $\vect A$. $\R_\text{sym}^{d\times d}$ stands for the space of real $d\times d$ symmetric matrices (similarly for $\C_\text{sym}^{d\times d}$). The symmetric part of $\vect A$ is $\sym \vect A = \tfrac{1}{2}(\vect A + \vect A^\top)$. The Frobenius inner product on $\C^{d\times d}$ is given by $\vect A : \vect B = \sum_{i,j=1}^d A_j^i \overline{B_j^i}$, and the corresponding norm is denoted by $|\vect A|$. For a fourth-order tensor $\A = (\A_{jl}^{ik}) \in \C^{d\times d\times d\times d}$, it shall be viewed as a mapping from $\C^{d\times d}$ to itself, with $\A \vect B = ((\A \vect B)_j^i) \in \C^{d\times d}$ given by $(\A \vect B)_j^i = \sum_{k,l=1}^d \A_{jl}^{ik} \vect B_l^k$.

\paragraph*{Vector-valued function spaces.} For $\Omega \subset \R^3$ open, with $\partial \Omega$ Lipschitz, we need $C^\infty(\Omega;\C^3)$, $L^p(\Omega;\C^3)$, and the Sobolev spaces $W^{k,p}(\Omega;\C^3)$, for $p \in [1,\infty]$ and $k \in \N_0$. We shall simply write, say $L^2$, when $\Omega$ and $\C^3$ is understood. $H^k = W^{k,2}$. 

\paragraph*{Periodic function and function spaces.} Let $Y = [0,1)^3$ be the unit cell, and $Y' = [-\pi,\pi)^3$ be the corresponding dual cell. A vector $\chi \in Y'$ shall be referred to as the quasimomentum. Let $\{e_1,e_2,e_3 \}$ be the standard basis of $\R^3$, and $\vect u:\R^3 \rightarrow \C^3$ be measurable. Then we say that $\vect u$ is a $(\Z^3-)$periodic function if
\begin{align}
    \vect u (x + ke_j) = \vect u(x), \qquad \text{a.e. } x \in \R^3, \text{ and all } k \in \Z.
\end{align}
The collection of smooth periodic functions is given by
\begin{align}
    C_{\#}^\infty(Y;\C^3) := \{ \vect u \,:\, \R^3 \rightarrow \C^3 \,|\, \vect u \text{ smooth and periodic} \},
\end{align}
and $\vect u \in C_{\#}^\infty$ is identified with its restriction to $\overline{Y}$. The periodic Sobolev space is given by
\begin{align}
    H_{\#}^k(Y;\C^3) := \overline{C_{\#}^\infty (Y;\C^3) }^{\| \cdot \|_{H^k}},
\end{align}
which also identified as a subspace of $L^2(Y;\C^3)$. $\dot{H}_{\#}^1(Y;\C^3)$ denotes the set of mean zero $H_{\#}^1$ functions.

\paragraph*{Operators.} The Fourier transform $\mathcal{F}$, viewed as a unitary operator from $L^2(\R^3;\C^3)$ to $L^2(\R^3;\C^3)$ is given by $\mathcal{F}(\vect f)(\theta) = \int_{\R^3} \vect f(y) e^{-2\pi i \theta \cdot y} \,dy$. We shall use $\mathcal{F}^{-1}$ for the inverse Fourier transform.

The (rescaled) Gelfand transform $\mathcal{G}_\eps$, for $\eps>0$, is given by
\begin{align}
    &\mathcal{G}_\varepsilon : L^2(\mathbb{R}^3;\mathbb{C}^3) \to L^2(Y';L^2(Y; \mathbb{C}^3)) = \int_{Y'}^\oplus L^2(Y; \mathbb{C}^3)d\chi \cong L^2(Y\times Y';\C^3), \\
    &(\mathcal{G}_\varepsilon \vect u)(y,\chi):= \left(\frac{\varepsilon}{2\pi}\right)^{3/2} \sum_{n\in \mathbb{Z}^3}e^{-i\chi \cdot (y+n)}\vect u(\varepsilon(y+n)), \quad y \in  Y, \quad \chi\in Y'. \label{eqn:gelfand_transform}
\end{align}
This is a unitary operator, with inversion formula
\begin{equation}\label{eqn:gelfand_inversion}
   \vect u(x) = (\mathcal{G}_\eps^\ast \mathcal{G}_\eps \vect u)(x) = \frac{1}{\left( 2\pi \varepsilon\right)^{3/2}} \int_{Y'} e^{i\chi\cdot x} (\mathcal{G}_\varepsilon \vect u)\left( \frac{x}{\eps}, \chi \right) d\chi, \quad x\in\mathbb{R}^3,
\end{equation}
where $\mathcal{G}_\eps\vect u(y,\chi)$ is extended in the $y$ variable by $\Z^3$-periodicity to the whole of $\R^3$. 

For each $\chi \in Y'$, we shall need the operator
\begin{align}
    &X_{\chi}: L^2( Y; \mathbb{C}^3) \rightarrow L^2(Y;\mathbb{C}^{3\times 3}) \\
    &X_{\chi}\vect u = \sym \left(\vect u \otimes \chi \right) = \sym \left(\vect u \chi^\top \right) \nonumber\\
    &\qquad = \begin{bmatrix}
    \chi_1 u_1 & \frac{1}{2}(\chi_1 u_2 + \chi_2 u_1) & \frac{1}{2}(\chi_1 u_3 + \chi_3 u_1)  \\
    \frac{1}{2}(\chi_1 u_2 + \chi_2 u_1) & \chi_2 u_2 & \frac{1}{2}(\chi_3 u_2 + \chi_2 u_3) \\
    \frac{1}{2}(\chi_3 u_1 + \chi_1 u_3) & \frac{1}{2}(\chi_3 u_2 + \chi_2 u_3) & \chi_3 u_3
    \end{bmatrix}.
\end{align}
By Lemma \ref{lem:symrk1}, $X_\chi$ has the property
\begin{equation}\label{eqn:Xchi_est}
    \frac{1}{2} |\chi| \|\vect u \|_{L^2(Y;\mathbb{C}^3)} 
    \leq \|X_{\chi}\vect u \|_{L^2(Y;\mathbb{C}^{3 \times 3})} 
    \leq |\chi|\|\vect u\|_{L^2(Y;\mathbb{C}^3)}, \qquad \text{for all } \vect u\in L^2(Y;\mathbb{C}^d).
\end{equation}

We conclude the section with a key identity relating $X_\chi$, $\mathcal{G}_\eps$, and (symmetric) gradients:
\begin{equation} \label{eqn:gelfand_symgrad_formula}
    \mathcal{G}_\varepsilon\left( \simgrad \vect u\right)(y,\chi) = \frac{1}{\varepsilon} \left(\simgrad_y\left(\mathcal{G}_\varepsilon \vect u\right) + i \sym \left(\left( \mathcal{G}_\varepsilon \vect u\right)\chi^\top  \right)\right)    = \frac{1}{\varepsilon} \left(\simgrad_y\left(\mathcal{G}_\varepsilon \vect u\right) + i X_\chi\left( \mathcal{G}_\varepsilon \vect u\right) \right).
\end{equation}

\subsection{Definition of key operators under study}
We shall now define the key operator of interest $\mathcal{A}_\eps$, and its counterpart $\mathcal{A}_\chi$ in the fiber (for each $\chi$) space $L^2(Y;\C^3)$. We begin with $\mathcal{A}_\eps$. Assume that the coefficient tensor $\mathbb{A}(y)$ satisfies:

\begin{assumption}\label{coffassumption}
    We impose the following assumptions on tensor of material coefficients $\A$:
    \begin{itemize}
        \item $\A:\R^3 \rightarrow \R^{3\times 3 \times 3 \times 3}$ is $\Z^3-$periodic.
    
        \item $\A$ is uniformly (in $y$) positive definite on symmetric matrices: There exists $\nu>0$ such that
        \begin{equation}\label{assump:elliptic}
            \nu|\vect \xi|^2 \leq \A(y) \vect \xi : \vect \xi \leq \frac{1}{\nu}|\vect \xi|^2, \quad \forall \, \vect \xi \in \R^{3\times 3}_\text{sym}, \quad \forall y \in Y. 
        \end{equation}
        
        \item The tensor $\A$ satisfies the following material symmetries:
        \begin{equation}\label{assump:symmetric}
            \A_{jl}^{ik} = \A_{il}^{jk} = \A_{lj}^{ki}, \quad  i,j,k,l\in\{1,2,3\}.
        \end{equation}
        \item The coefficients of $\A$ satisfy $\A_{jl}^{ik} \in L^{\infty}(Y)$, where $i,j,k,l\in\{1,2,3\}$.
    \end{itemize}
\end{assumption}

We shall write $\A_\eps = \A(\frac{\cdot}{\eps})$. Then
\begin{definition}\label{defn:main_operator_fullspace}
    The operator $\mathcal{A}_\eps \equiv (\simgrad)^* \mathbb{A}_\eps \left( \simgrad \right) $ is the operator on $L^2(\R^3;\C^3)$ defined through its corresponding sesquilinear form $a_\eps$ with form domain $H^1(\R^3;\C^3)$ and action
    \begin{align}\label{eqn:ahom_complexHS}
        a_\eps (\vect u, \vect v) = \int_{\R^3} \mathbb{A} \left( \frac{x}{\eps} \right) \simgrad \vect u(x) : \overline{\simgrad \vect v (x)} ~ dx, \quad \vect u,\vect v \in \mathcal{D}(a_\eps) = H^1(\R^3;\C^3).
    \end{align}
\end{definition}

Next, we define the operator $\mathcal{A}_\chi$:
\begin{definition}\label{defn:main_operator_fiberspace}
    The operator $\mathcal{A}_\chi \equiv (\simgrad + i X_\chi)^* \mathbb{A} (\simgrad + i X_\chi) $ is the operator on $L^2(Y;\C^3)$ defined through its corresponding sesquilinear form $a_\chi$ with form domain $H_{\#}^1(Y;\C^3)$ and action
    \begin{align}\label{form_eqn}
        a_\chi (\vect u, \vect v) = \int_Y \mathbb{A}(y) (\simgrad + i X_\chi) \vect u (y) : \overline{(\simgrad + i X_\chi) \vect v (y)} ~ dy, \quad \vect u,\vect v \in H^1_{\#}(Y;\C^3).
    \end{align}
\end{definition}

\begin{remark}
    We write $\mathcal{D}[\mathcal{A}] = \mathcal{D}(a)$ for the form domain of an operator $\mathcal{A}$ corresponding to the form $a$, and $\mathcal{D}(\mathcal{A})$ for the operator domain. $\sigma(\mathcal{A})$ denotes the spectrum of $\mathcal{A}$, and $\rho(\mathcal{A})$ denotes the resolvent set. We shall omit the differential ``$dy$" where it is understood.
\end{remark}

We recall from \cite[Prop 2.9]{simplified_method}, the relation between $\mathcal{A}_\eps$ and $\mathcal{A}_\chi$:
\begin{proposition}[Passing to the unit cell for $\mathcal{A}_\eps$] \label{prop:pass_to_unitcell} 
    We have the following relation between the forms
    \begin{equation}
        a_\varepsilon(\vect u,\vect v) = \int_{Y'}\frac{1}{\varepsilon^2}a_{\chi}(\mathcal{G}_\varepsilon\vect u,\mathcal{G}_\varepsilon\vect v) d\chi,  \quad \vect u,\vect v \in H^1(\R^3;\C^3),
    \end{equation}
    and the following relations between the operators
    \begin{align}
        \mathcal{A}_\varepsilon &= \mathcal{G}_\varepsilon^* \left( \int_{Y'}^\oplus \frac{1}{\varepsilon^2}\mathcal{A}_{\chi} d\chi \right) \mathcal{G}_\varepsilon, \label{eqn:vonneumannformula_ops} \\
        \left(\mathcal{A}_\varepsilon - zI \right)^{-1} &= \mathcal{G}_\varepsilon^* \left( \int_{Y'}^\oplus \left(\frac{1}{\varepsilon^2}\mathcal{A}_{\chi} - zI \right)^{-1}  d\chi \right) \mathcal{G}_\varepsilon, \qquad \text{for } z \in \rho(\mathcal{A}_\varepsilon).\label{eqn:vonneumannformula_resolvent}
    \end{align}
\end{proposition}

In other words, $\frac{1}{\eps^2}\mathcal{A}_\chi$ are the fibers of an operator that is unitarily equivalent to $\mathcal{A}_\eps$, under the rescaled Gelfand transform $\mathcal{G}_\eps$.

\paragraph*{Homogenized operator.} We shall now define the homogenized operator $\mathcal{A}^{\hom}$. Since its resolvent at $z=-1$, $(\mathcal{A}^{\hom} + I)^{-1}$, is the leading-order term in our expansion, we introduce it now for the reader's convenience. The definition of higher-order ``corrector operators" are postponed to Section \ref{sect:fiberwise_results}.

To begin, we define the following form, in which we shall extract the homogenized tensor $\A^{\hom}$:
\begin{definition}
    Let $a^{\text{hom}}$ be the bilinear form on $\R^{3 \times 3}_\text{sym}$ defined by
    \begin{equation}
        a^{\rm hom}(\vect \xi, \vect \zeta) = \int_{Y} \A \left( \vect \xi + \simgrad \vect u^{\vect \xi} \right ) : \vect \zeta \, dy, \quad \vect \xi, \vect \zeta \in \R^{3 \times 3}_\text{sym},
    \end{equation}
    where the \textit{corrector term} $\vect u^{\vect \xi} \in H_{\#}^1(Y;\R^3)$ is the unique solution of the cell-problem:
    \begin{equation}\label{correctordefinition}
        \begin{cases}
            \int_{ Y} \A \left( \vect \xi + \simgrad \vect u^{\vect \xi} \right ) : \simgrad \vect v \, dy = 0, \quad \forall \vect v \in  H_{\#}^1(Y;\R^3), \\
            \int_Y \vect u^{ \vect \xi } = 0.
        \end{cases}
    \end{equation}
    By \cite[Prop 2.13]{simplified_method}, $a^{\hom}$ is a positive symmetric bilinear form on $\R^{3\times 3}_\text{sym}$, and thus can be uniquely represented by a (constant) tensor $\A^{\hom} \in \R^{3\times 3\times 3\times 3}$ satisfying the symmetries \eqref{assump:symmetric} as $\A(y)$, 
    \begin{align}
        a^{\hom} (\vect \xi, \vect \zeta) = \mathbb{A}^{\hom} \vect \xi : \vect \zeta,
        \qquad \forall \, \vect \xi, \vect \zeta \in \R^{3\times 3}_\text{sym}. \label{eqn:ahom_tensor_form_repr}
    \end{align}
    Moreover, there exist a constant $\nu_{\hom} = \nu_{\hom}( \nu ) > 0$ such that
    \begin{equation}\label{eqn:Ahomcoercivity_bdd}
        \nu_{\text{hom}} |\vect \xi|^2 \leq \A^{\rm hom}\vect \xi: \vect \xi \leq \frac{1}{\nu_{\text{hom}}} |\vect \xi|^2, \quad \forall \, \vect \xi \, \in \R^{3 \times 3}_\text{sym}. 
    \end{equation}
\end{definition}

The tensor $\mathbb{A}^{\rm hom}$ is used to define the homogenized operator $\mathcal{A}^{\rm hom}$ as follows:

\begin{definition}[Homogenized operator] \label{defn:hom_operator_fullspace}
    Let $\mathcal{A}^{\rm hom}$ be the operator on $L^2(\R^3;\C^3)$ given by the following differential expression 
    \begin{equation}
        \mathcal{A}^{\rm hom}  \equiv (\simgrad)^* \mathbb{A}^{\rm hom} \left(\simgrad \right),
    \end{equation}
    with domain $\mathcal{D}\left(\mathcal{A}^{\rm hom}\right) = H^2(\R^3;\C^3)$. Its corresponding form is given by
    \begin{equation}
        \left\langle \mathcal{A}^{\rm hom} \vect u, \vect v \right\rangle_{L^2(\R^3;\C^3)} = \int_{\R^3} \mathbb{A}^{\rm hom} \simgrad \vect u : \overline{\simgrad \vect v} \, dy, \quad \vect u \in \mathcal{D}(\mathcal{A}^{\rm hom}), \vect v \in \mathcal{D}[\mathcal{A}^{\rm hom}] := H^1(\R^3,\C^3).
    \end{equation}
\end{definition}

We shall also need the corresponding operator for $\mathcal{A}^{\hom}$ in the fiber space, and the analogue of Proposition \ref{prop:pass_to_unitcell}. We begin with the definition of $\mathcal{A}_\chi^{\hom}$.




\begin{definition}\label{defn:hom_matrix}
    For each $\chi \in Y'$, set $\mathcal{A}^{\rm hom}_\chi \in \C^{3 \times 3}$ to be the constant matrix satisfying
    \begin{equation}\label{eqn:hom_chi_matrix}
        \left\langle \mathcal{A}^{\rm hom}_\chi \vect c,\vect d \right\rangle_{\C^3} = \int_{Y} \A\left(  \simgrad   \vect u_{\vect c}  +  iX_\chi   \vect c \right) : \overline{  iX_\chi  \vect d }, \qquad \forall \vect c,\vect d \in \C^3,
    \end{equation}
    where the \textit{corrector term} $\vect u_{\vect c} \in H_\#^1(Y;\C^3)$ is the unique solution of the ($\chi$ dependent) cell-problem
    \begin{equation}\label{eqn:chi_dependent_cell_problem}
        \begin{cases}
		\int_{Y} \A\left(  \simgrad \vect u_{\vect c}  +  iX_\chi  \vect c \right) : \overline{ \simgrad   \vect v} = 0, \quad \forall\vect v \in H^1_{\#}(Y;\C^3),\\
        \int_Y \vect u_{\vect c} = 0.
        \end{cases}
    \end{equation}
\end{definition}

Next, we summarize the key properties of $\mathcal{A}_\chi^{\hom}$.

\begin{proposition}
    $\mathcal{A}_\chi^{\hom}$ is a hermitian matrix, and can be written as $\mathcal{A}_\chi^{\rm hom} = \left(iX_\chi \right)^*\A^{\rm hom} (i X_\chi)$. In particular, $\mathcal{A}_\chi^{\rm hom}$ is quadratic in $\chi$ in the following sense: there exist a constant $\nu_1 >0$ that depends only on $\nu_{\text{hom}}$ (from \eqref{eqn:Ahomcoercivity_bdd}), such that 
    \begin{equation}\label{eqn:achihom_coer_bdd}
        \nu_1 |\chi|^2|\vect c|^2 \leq \langle \mathcal{A}_\chi^{\rm hom} \vect c, \vect c \rangle_{\C^3} \leq \frac{1}{\nu_1}|\chi|^2 |\vect c|^2, \quad \forall \vect c \in \C^3.
    \end{equation} 
\end{proposition}
\begin{proof}
    \cite[Lemma 5.3, Proposition 5.5]{simplified_method}.
\end{proof}

Thus $\mathcal{A}_\chi^{\text{hom}}$ is diagonalizable, and we may introduce the notation:

\begin{definition}
    Write $\lambda_1^{\text{hom},\chi}$, $\lambda_2^{\text{hom},\chi}$, and $\lambda_3^{\text{hom},\chi}$ for the eigenvalues of $\mathcal{A}_\chi^{\text{hom}}$, arranged in non-decreasing order.
\end{definition}

Finally, we state the analogous result to Proposition \ref{prop:pass_to_unitcell}, for the homogenized operator $\mathcal{A}^{\hom}$.
\begin{proposition}[\texorpdfstring{\cite[Prop. 5.6]{simplified_method}}{} Passing to the unit cell for $\mathcal{A}^{\text{hom}}$] \label{prop:pass_to_unitcell2} 
    We have the following identities: 
    \begin{align}
        \mathcal{A}^{\rm hom} \Xi_\varepsilon 
        &= \mathcal{G}_\eps^* \left(\int_{Y'}^{\oplus} \frac{1}{\eps^2} P_0 \mathcal{A}_\chi^{\rm hom} P_0 \, d\chi \right) \mathcal{G}_\eps. \label{eqn:vn_ahom_ops} \\
        \left( \mathcal{A}^{\rm{hom}} - zI \right)^{-1} \Xi_\eps 
        &= \mathcal{G}_\eps^* \left( \int_{Y'}^{\oplus} \left( \frac{1}{\eps^{2}} \mathcal{A}_{\chi}^{\rm{hom}} - zI_{\C^3} \right)^{-1} P_0 d\chi \right) \mathcal{G}_\eps, \qquad \text{for } z \in \rho(\mathcal{A}^{\rm{hom}}), \label{eqn:vn_ahom_resolvents}
    \end{align}
    where $P_0$ is the projection of $L^2(Y;\C^3)$ onto the space of constant functions $\C^3$, and $\Xi_\eps$ is the smoothing operator $\mathcal{G}_\eps^\ast \left( \int_{Y'}^{\oplus} P_0 \, d\chi \right) \mathcal{G}_\eps$.
\end{proposition}

We refer the reader to Section \ref{sect:recap_spectral_analysis} for a further discussion on the connection between $P_0$, $\Xi_\eps$, and the spectrum of $\mathcal{A}_\chi$.

\newpage
\subsection{Main results}\label{sect:main_results}
The main result of this paper is as follows.
\begin{theorem}\label{thm:main_result_intro}
    Let $\vect f \in L^2(\R^3;\C^3)$ be such that its Fourier transform $\mathcal{F}(\vect f)$ is supported in a compact set $K$. Then there exist constants $\mu = \mu(\nu , C_\text{Fourier}) > 0$ (see \eqref{eqn:constant_mu_smallfreq} and Remark \ref{rmk:mu_rho0_dependencies}) and $\eps_0 = \eps_0(\mu, K)>0$ such that whenever $0<\eps<\eps_0$, the following estimate hold for each $n \in \N_0$:
    \begin{align}
        \left\| 
        \left( \mathcal{A}_\eps + I \right)^{-1} \widehat{\Xi}_{\eps,\mu} \vect f - \sum_{k=0}^n \left( \mathcal{R}_{0,\eps}^{(k)} + \mathcal{R}_{1,\eps}^{(k)} + \mathcal{R}_{2,\eps}^{(k)} \right) \widehat{\Xi}_{\eps,\mu} \vect f
        \right\|_{L^2(\R^3;\C^3)}
        \leq C^{n+1} \eps^{n+1} \| \vect f \|_{L^2(\R^3;\C^3)}, \label{eqn:l2tol2_intro} \\
        \left\| 
        \left( \mathcal{A}_\eps + I \right)^{-1} \widehat{\Xi}_{\eps,\mu} \vect f - \sum_{k=0}^n \left( \mathcal{R}_{0,\eps}^{(k)} + \mathcal{R}_{1,\eps}^{(k)} + \mathcal{R}_{2,\eps}^{(k)} \right) \widehat{\Xi}_{\eps,\mu} \vect f
        \right\|_{H^1(\R^3;\C^3)}
        \leq C^{n+1} \eps^{n} \| \vect f \|_{L^2(\R^3;\C^3)}, \label{eqn:l2toh1_intro}
    \end{align}
    for some constant $C = C(K, \nu, C_\text{Korn}, C_\text{Fourier}) > 0$. 
    The \textit{full-space corrector operators} $\mathcal{R}_{j,\eps}^{(k)}$ are defined in Section \ref{sect:rescaled_and_fullspace_ops}, the \textit{Bloch approximation operator} $\widehat{\Xi}_{\eps,\mu}$ is given by Definition \ref{defn:bloch_approximation}, and the constants $C_\text{Korn}$ and $C_\text{Fourier}$ are introduced in Appendix \ref{appendix:useful_estimates}.
\end{theorem}


Theorem \ref{thm:main_result_intro} is proven in Section \ref{sect:fullspace_estimates} as Theorem \ref{thm:l2tol2} for \eqref{eqn:l2tol2_intro}, and Theorem \ref{thm:l2toh1} for \eqref{eqn:l2toh1_intro}.

\begin{remark}
    \begin{itemize}
        \item Since the constant $C$ in \eqref{eqn:l2tol2} depends on $\vect f$ only through its Fourier support, we may further conclude that the estimate \eqref{eqn:l2tol2} is uniform over the collection of all $\vect f \in L^2(\R^3;\C^3)$ whose Fourier support lies in a \textit{common} compact set $K \subset \R^3$.

        \item It is known (e.g. \cite{birman_suslina_2004,birman_suslina_2006_l2l2higherorder}) and also shown using the operator asymptotic method (in \cite[Sect. 6]{simplified_method}) that the assumption of compact Fourier support and the use of Bloch approximation $\widehat{\Xi}_{\eps,\mu}$ may be omitted, for $L^2$ estimates \eqref{eqn:l2tol2_intro} of order $\eps^2$, and $H^1$ estimates \eqref{eqn:l2toh1_intro} of order $\eps$. Extra preparation of the data is only necessary if one wishes to go beyond $\mathcal{O}(\eps^2)$ error in $L^2$ or $\mathcal{O}(\eps)$ error in $H^1$.

        \item While it is necessary to impose additional restrictions on $\vect f \in L^2$ in order to obtain higher-order estimates, this paper does not aim for the optimal set of conditions for $\vect f$. Rather, we focus on one that is most natural from the spectral perspective. In particular, our choice, the Bloch approximation $\widehat{\Xi}_{\eps,\mu}$, projects the data $\mathcal{G}_\eps \vect f$ onto the eigenspace corresponding to the eigenvalues at the bottom of the spectrum $\sigma(\mathcal{A}_\chi)$. Details can be found in Section \ref{sect:well_prepared_data}. 
        We shall leave the extension of the method to larger classes of functions, e.g. compact Fourier support or Schwartz class, to a future work.

        \item It should also be noted that our choice $\widehat{\Xi}_{\eps,\mu}$ is a natural extension of \cite[Definition 1.5]{lamacz_yousept_2021} and \cite[Section 1.2]{conca_orive_vanninathan2002} to vector-valued functions (compare the definitions therein with Remark \ref{rmk:more_alt_formulae_blochapprox}). For this reason, we use the name ``Bloch approximation" for $\widehat{\Xi}_{\eps,\mu}$.

        \item One may recover the classical homogenization result from Theorem \ref{thm:main_result_intro} as follows: As noted above, we may drop $\widehat{\Xi}_{\eps,\mu}$. Then combine Proposition \ref{prop:pass_to_unitcell2}, Remark \ref{rmk:corrector_ops_vs_classical_formula_fibrespace}, Definitions \ref{defn:corrector_operators_chi_rescaled}-\ref{defn:corrector_operators_fullspace} to obtain
        \begin{align}
            \mathcal{R}_{0,\eps}^{(0)} = \left( \mathcal{A}^{\hom} + I \right)^{-1} \Xi_\eps.
        \end{align}
        Now use \eqref{eqn:fewer_terms_1}, Theorem \ref{thm:l2tol2}, and $\| (\mathcal{A}^{\hom} + I)^{-1} (I - \Xi_\eps) \|_{L^2\rightarrow L^2} \leq C \eps^2$ \cite[Theorem 6.17]{simplified_method} to get
        \begin{align}
            \| \left( \mathcal{A}_\eps + I \right)^{-1} - \left( \mathcal{A}^{\hom} + I \right)^{-1} \|_{L^2\rightarrow L^2} \leq C \eps.
        \end{align}
        This has been shown using various approaches, including the operator asymptotic method as \cite[Theorem 2.15]{simplified_method}. A similar remark holds for the $L^2\rightarrow H^1$ $\mathcal{O}(\eps)$ case, see Remark \ref{eqn:fewer_terms_1}.
        
        \item To keep the formulae clean, we have used $\sum_{k=0}^n \left( \mathcal{R}_{0,\eps}^{(k)} + \mathcal{R}_{1,\eps}^{(k)} + \mathcal{R}_{2,\eps}^{(k)} \right)$ in \eqref{eqn:l2tol2_intro}-\eqref{eqn:l2toh1_intro}. Some terms may be dropped without changing the order of the error. See Remarks \ref{rmk:fewer_terms_l2l2} and \ref{rmk:fewer_terms_l2h1}. \qedhere
    \end{itemize}
\end{remark}

\newpage

\section{A recap on the spectral analysis of \texorpdfstring{$\mathcal{A}_\chi$}{Achi} and consequences}\label{sect:recap_spectral_analysis}
This section summarizes and extends the relevant material from \cite[Section 4, 5.1]{simplified_method}.

First, by the compact embedding of $H^1_{\#}(Y;\C^3)$ into $L^2(Y;\C^3)$ \cite[Corollary 6.11]{david_borthwick}, the spectrum of $\mathcal{A}_\chi$ is discrete. We label its eigenvalues $\lambda_n^\chi$ in non-decreasing order:
\begin{align}
    0 \leq \lambda_1^\chi \leq \lambda_2^\chi \leq \lambda_3^\chi \leq \lambda_4^\chi \leq \cdots \rightarrow \infty
\end{align}

\begin{definition}\label{defn:proj_achi_eigenspace}
    Let $P_\chi: L^2(Y;\C^3) \rightarrow L^2(Y;\C^3)$ be the projection onto the eigenspace corresponding to the first three eigenvalues $\lambda_1^\chi$, $\lambda_2^\chi$, $\lambda_3^\chi$ of $\mathcal{A}_\chi$.
\end{definition}

Next, recall that the Rayleigh quotient associated with to form $a_\chi$ (Definition \ref{defn:main_operator_fiberspace}) is given by 
\begin{equation}
    \mathcal{R}_\chi(\vect u) = \frac{a_\chi( \vect u, \vect u)}{\Vert \vect u\Vert _{L^2(Y;\C^3)}^2}, \quad \vect u \in H^1_{\#}(Y;\C^3) \setminus \{ 0 \}.
\end{equation}

Then,
\begin{proposition}\label{prop:Rayleighestim}
    There exist constants $C_{\text{rayleigh}} > c_{\text{rayleigh}} > 0$, depending only on $C_\text{Fourier}$ and $\nu$, such that
    \begin{alignat}{2}
        c_{\text{rayleigh}}{|\chi|^2} &\leq \mathcal{R}_\chi(\vect u)    &&\forall \vect u \in H^1_{\#}(Y;\C^3)\setminus \{ 0 \}, \label{Rayleighestim_1} \\
        0 &\leq\mathcal{R}_\chi(\vect u) \leq C_{\text{rayleigh}}{|\chi|^2} \qquad   &&\forall \vect u \in \C^3 \setminus \{ 0 \}, \label{Rayleighestim_2}\\
        c_{\text{rayleigh}} &\leq \mathcal{R}_\chi(\vect u)    &&\forall \vect u \in (\C^3)^\perp \cap  H_{\#}^1(Y;\C^3) \setminus \{ 0 \}, \label{Rayleighestim_3}
    \end{alignat}
    where $\C^3$ is viewed as a subspace of $L^2(Y;\C^3)$ by identifying $\C^3$ with the space of constant functions.
\end{proposition}
\begin{proof}
    \cite[Proposition 4.2]{simplified_method}. The constant $C_\text{Fourier}$ comes from Proposition \ref{prop:coercive_est}.
\end{proof}

\begin{remark}\label{rmk:proj_0_is_integral_in_y}
    By Proposition \ref{prop:Rayleighestim}, the first three eigenvalues of $\mathcal{A}_{\chi=0}$ are $\lambda^0_1 = \lambda^0_2 = \lambda^0_3 = 0$. Moreover, it is clear that $\C^3$ is contained in the eigenspace $\eigenspace{\left( 0 ; \mathcal{A}_0 \right)} = \ker{(\mathcal{A}_0)}$. Since the latter is three-dimensional, we must have $\C^3 = \ker{(\mathcal{A}_0)}$. Thus, $P_0$ is the projection onto $\C^3$, i.e. 
    \begin{equation*}
        P_0:L^2(Y;\C^3) \rightarrow \C^3 \hookrightarrow L^2(Y;\C^3),
        \qquad P_0 \vect u = \int_Y \vect u. \qedhere
    \end{equation*}
\end{remark}

\subsection{Consequence 1: Separation of spectrum \texorpdfstring{$\sigma(\mathcal{A}_\chi)$}{spec(Achi)}, for small \texorpdfstring{$\chi$}{chi}.}

By combining Proposition \ref{prop:Rayleighestim} with the min-max principle \cite[Theorem 5.15]{david_borthwick}, we obtain
\begin{corollary}\label{cor:eigenvalue_sizes}
    The spectrum $\sigma(\mathcal{A}_\chi)$ contains three eigenvalues of order $\mathcal{O}({|\chi|^2})$ as $|\chi| \downarrow 0$, while the remaining eigenvalues are uniformly bounded away from zero.
\end{corollary}

Moreover, we recall the following from \cite[Lemma 6.1]{simplified_method}:
\begin{lemma}\label{lem:contour_existence}
    There exists a closed contour  $\Gamma \subset \left\{z\in \C, \Re(z)>0\right\}$, oriented anticlockwise, where:
    
    \begin{itemize}
        \item \textbf{(Separation of spectrum)} There exist some $\mu>0$, such that for each $\chi \in [-\mu,\mu]^3 \setminus \left\{0\right\}$, $\Gamma$ encloses the three smallest eigenvalues of the operators $\frac{1}{|\chi|^2}\mathcal{A}_\chi$ and $\frac{1}{|\chi|^2}\mathcal{A}_{\chi}^{\rm hom}$. That is, the points
        \begin{align}\label{eqn:contour_separation}
            \frac{1}{|\chi|^2} \lambda_i^\chi, \quad \frac{1}{|\chi|^2} \lambda_i^{\text{hom},\chi}, \quad i=1,2,3.
        \end{align}
        Furthermore, $\Gamma$ does not enclose any other eigenvalues of $\frac{1}{|\chi|^2}\mathcal{A}_\chi$ (and $\frac{1}{|\chi|^2}\mathcal{A}_{\chi}^{\rm hom}$).

        \item \textbf{(Buffer between contour and spectra)} There exist some $\rho_0 > 0$ such that 
        \begin{align}\label{eqn:contour_buffer}
            \inf_{\substack{ z\in \Gamma, \\ \chi \in [-\mu,\mu]^3\setminus \{ 0 \} \\ i \in \{1,2,3,4\}}} \left|z - \frac{1}{|\chi|^2}\lambda_i^{\chi} \right| \geq \rho_0
            \qquad \text{ and } \quad
            \inf_{\substack{ z\in \Gamma, \\ \chi \in [-\mu,\mu]^3\setminus \{ 0 \} \\ i \in \{1,2,3\}} } \left|z - \frac{1}{|\chi|^2}\lambda_i^{\text{hom},\chi} \right| \geq \rho_0.
        \end{align}
        
    \end{itemize}
\end{lemma}

See Figure \ref{fig:contour} for a picture of Lemma \ref{lem:contour_existence}.
\begin{figure}[t]
  \centering
  \includegraphics[page=1, clip, trim=4cm 4cm 5cm 5cm, width=0.90\textwidth]{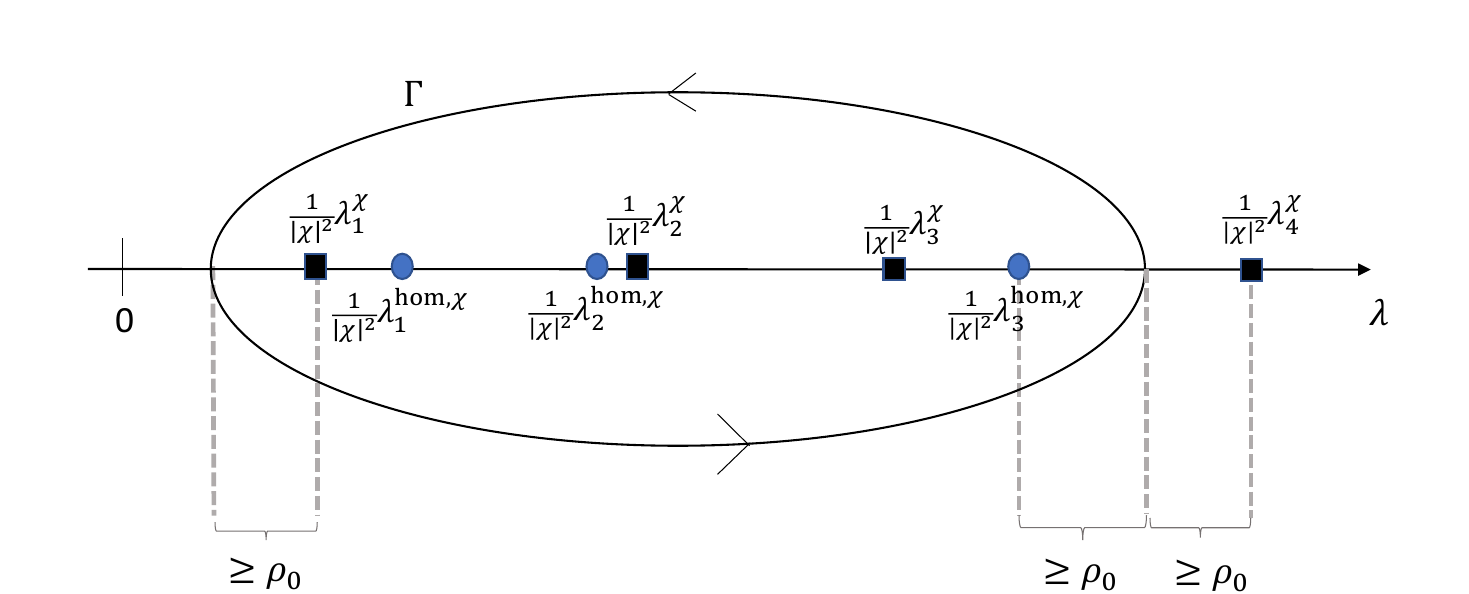}
  \caption{The contour $\Gamma$, for $\chi \in [-\mu,\mu]^3\setminus \{ 0 \}$. Taken from \cite[Figure 1]{simplified_method}.} \label{fig:contour}
\end{figure}
\begin{align}\label{eqn:constant_mu_smallfreq}
    \begin{split}
        &\textbf{We shall henceforth take the constants $\mu$, $\rho_0 >0$ and contour $\Gamma$ as provided by Lemma \ref{lem:contour_existence}.}
    \end{split}
\end{align}

\begin{remark}[Dependencies of $\mu$ and $\rho_0$] \label{rmk:mu_rho0_dependencies}
    Note that $\mu$ and $\rho_0$ depend only on the spectral properties of $\mathcal{A}_\chi$ (Proposition \ref{prop:Rayleighestim}) and $\mathcal{A}_\chi^{\hom}$ (see \eqref{eqn:Ahomcoercivity_bdd}). In particular, $\mu$ and $\rho_0$ depend only on $\nu$ and $C_\text{Fourier}$.
\end{remark}

In Section \ref{sect:fibrewise_estimates_working}, it will be convenient to introduce the following constants:
\begin{definition}[Constants depending on $z$] \label{defn:dhomchi_dchi}
    \begin{equation}\label{eqn:dist_to_spec}
        D_\chi^{\text{hom}}(z) := \text{dist} (z, \sigma(\tfrac{1}{|\chi|^2} \mathcal{A}_\chi^{\text{hom}} ) ), \qquad
        D_\chi(z) := \text{dist} (z, \sigma(\tfrac{1}{|\chi|^2} \mathcal{A}_\chi ) ).
    \end{equation}
\end{definition}

\begin{remark}\label{rmk:z_on_contour_nice}
    If $z \in \Gamma$, then $D_\chi^{\text{hom}}(z)$ and $D_\chi(z)$ are bounded away from zero, uniformly in $\chi$ and $z$.
\end{remark}



\subsection{Consequence 2: A candidate for ``well-prepared" data.}\label{sect:well_prepared_data}

As remarked in Section \ref{sect:introduction}, to go beyond the basic homogenization results \cite[Theorem 2.15]{simplified_method}, namely $L^2 \rightarrow L^2$ estimates with $\mathcal{O}(\eps)$ and $\mathcal{O}(\eps^2)$ error, and $L^2 \rightarrow H^1$ estimate of $\mathcal{O}(\eps)$ error, one is required to impose smoothness conditions on the data $\vect f$\footnote{or on the coefficient tensor $\A(y)$, which we shall avoid.}. In this paper, we do not aim for optimal conditions that give an $\mathcal{O}(\eps^n)$ error. Rather, we give conditions that are most natural from the spectral perspective. This shall be done in two steps. 

\paragraph*{Step 1.} First, we introduce the smoothing operator $\Xi_\eps$.

\begin{definition}\label{defn:smoothing_operator}
    For $\eps>0$, the smoothing operator $\Xi_\eps : L^2(\R^3;\C^3) \rightarrow L^2(\R^3;\C^3)$ is defined as follows:
    \begin{align}
        \Xi_\eps \vect f = \mathcal{G}_\eps^\ast \left( \int_{Y'}^{\oplus} P_0 \, d\chi \right) \mathcal{G}_\eps \vect f,
    \end{align}
    In other words, $\mathcal{G}_\eps \Xi_\eps \vect f = \int_Y (\mathcal{G}\vect f)(y,\cdot) dy$, which is a function in $y \in Y$ and $\chi \in Y'$, that is constant in $y$. 
\end{definition}

Since $\Xi_\eps$ is constructed from spectral projections of $\mathcal{A}_\chi$, it is a natural object from the spectral perspective. However, it should be noted that $\Xi_\eps$ is also closely connected to the Fourier transform.
To be precise, we note from \cite[Section 4.2]{simplified_method} the following identity.
\begin{proposition}
    [Smoothing operator is a Fourier cutoff at a large cube] \label{prop:smoothing_vs_fourier}
    For $\vect f \in L^2(\R^3;\C^3)$, we have
    \begin{align}
        \Xi_\eps \vect f (x)
        = \int_{(2\pi\eps)^{-1} Y'} \mathcal{F} (\vect f) (\widetilde{\theta}) e^{i2\pi\widetilde{\theta} \cdot x} d\widetilde{\theta}
        = \left( \mathcal{F}^{-1}(\mathbf{1}_{(2\pi\eps)^{-1}Y'}) \ast \vect f  \right)(x).
    \end{align}
    In particular, $\vect f \in L^2$ has compact Fourier support if and only if $\Xi_\eps \vect f = \vect f$, whenever $\eps>0$ is small enough.
\end{proposition}

It is thus clear from Proposition \ref{prop:smoothing_vs_fourier} and Definition \ref{defn:smoothing_operator}, that if $\vect f$ has compact Fourier support, then $(\mathcal{G}_\eps \vect f)(y,\chi)$ is constant in $y$ for $\eps$ small enough. But for our purposes, we need the following refinement:

\begin{proposition}\label{prop:compact_fourier_support_properties}
    Suppose that $\vect f \in L^2(\R^3;\C^3)$ have compact Fourier support $K = \supp{(\mathcal{F}(\vect f))}$, then there exist a constant $c_{\text{supp}}(K) > 0$, and for any $0<\mu\leq \pi$, another constant $\eps_0 (\mu, K) >0$ such that whenever $0<\eps<\eps_0$,
    
    \begin{enumerate}[label=(\roman*)]
        \item its rescaled Gelfand transform $\mathcal{G}_\eps \vect f$ is constant in $y$, and
        
        
        \item the $\chi-$support of $\mathcal{G}_\eps \vect f$ is contained in $c_{\text{supp}}[-\eps,\eps]^3$, which is contained in $[-\mu,\mu]^3$.
    \end{enumerate}
\end{proposition}

We refer the reader to Figure \ref{fig:frequency_regions} for a picture of Proposition \ref{prop:compact_fourier_support_properties}(ii). The key observation is as follows: if $\supp{(\mathcal{F}(\vect f))}$ is compact, then upon rescaling, it is supported in a region of size $O(\eps)$. Together with the scaling properties of the spectrum of $\mathcal{A}_\eps$ near the bottom (Corollary \ref{cor:eigenvalue_sizes}), we arrive at the heuristic $|\chi| \sim \eps$ connecting the two length scales.

\begin{figure}[t]
  \centering
  \includegraphics[page=1, clip, trim=3.5cm 4cm 6.5cm 2.8cm, width=0.80\textwidth]{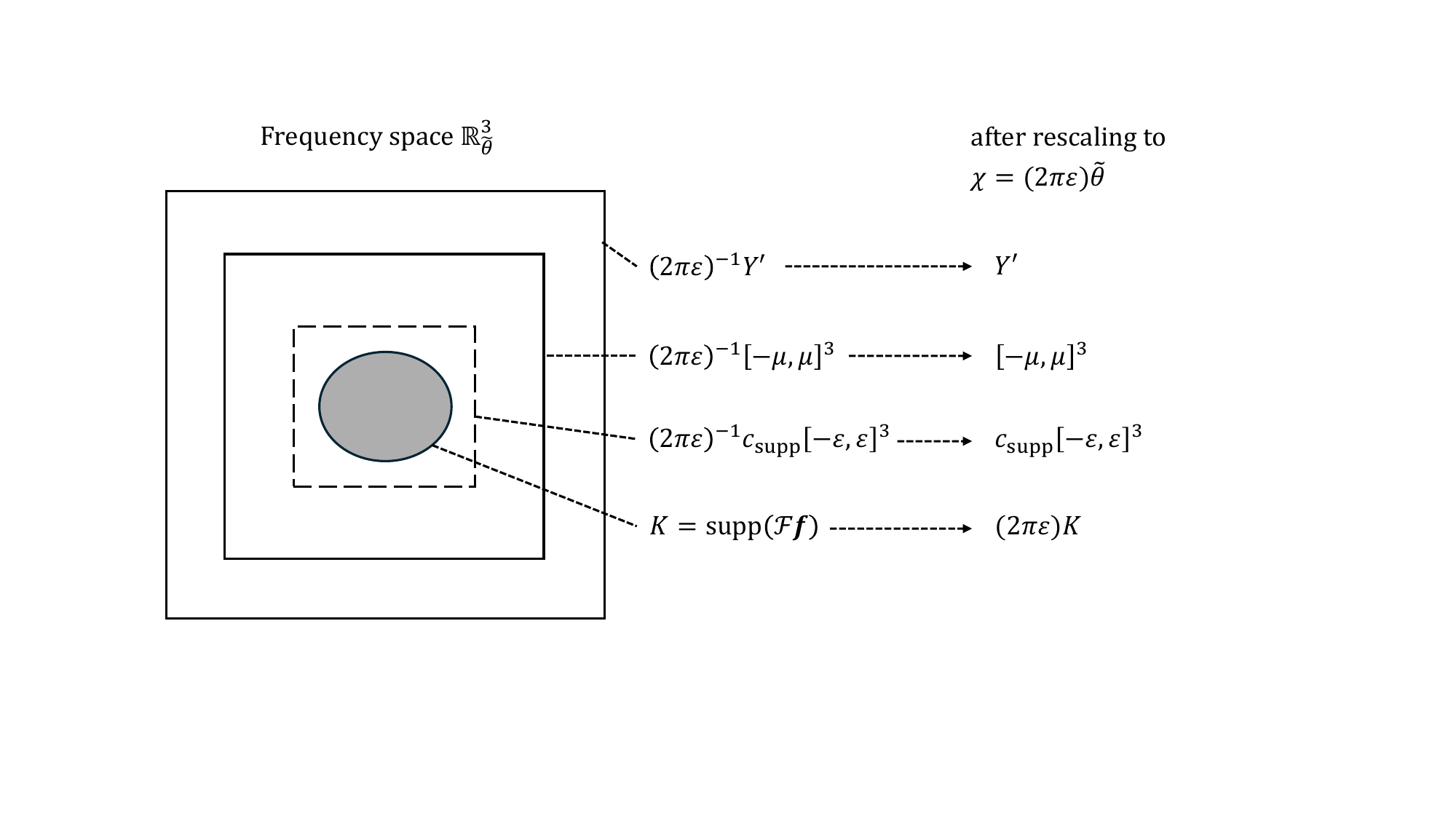} 
  \caption{Key regions in the frequency space, before and after rescaling.} \label{fig:frequency_regions}
\end{figure}

\begin{remark}
    If $\supp{(\mathcal{F}(\vect f))}$ is not compact, then heuristic $|\chi| \sim \eps$ may still be justified. However, the error estimates would be capped at $\mathcal{O}(\eps^2)$. See for instance, the proof of \cite[Theorem 6.6]{simplified_method}.
\end{remark}

We included the $\mathcal{O}(1)$ region $[-\mu,\mu]^3$ because $\mu$ shall be chosen in a way such that these are the range of $\chi$'s where the spectrum of $\mathcal{A}_\chi$ behaves ``nicely". (See Lemma \ref{lem:contour_existence} and Figure \ref{fig:contour}.)

\begin{proof}[Proof of Proposition \ref{prop:compact_fourier_support_properties}]
    Let $\vect f \in L^2(\R^3;\C^3)$. By assumption, $K = \supp(\mathcal{F} (\vect f))$ is a compact subset of $\R^3$ (in the frequency space). Hence, there exist a constant $c_{\text{supp}}(K)>0$ such that
    \begin{align}\label{eqn:f_support_containment1}
        K \subset (2\pi \eps)^{-1} c_{\text{supp}}[-\eps,\eps]^3 = (2\pi)^{-1} c_{\text{supp}}[-1,1]^3.
    \end{align}
    
    Moreover, for any $0<\mu \leq \pi$, there exist $\eps_0(\mu,c_\text{supp}) > 0$ such that
    \begin{align}
        K 
        \stackrel{\eqref{eqn:f_support_containment1}}{\subset} (2\pi \eps)^{-1} c_{\text{supp}}[-\eps,\eps]^3
        \subset (2\pi \eps)^{-1} [-\mu,\mu]^3 \subset (2\pi \eps)^{-1} Y', 
        \qquad \text{whenever $0<\eps<\eps_0$.} \label{eqn:f_support_containment2}
    \end{align}
    
    Henceforth, we shall take $0 < \eps < \eps_0$. Then,
    \begin{alignat}{3}
        \vect f(x)
        &= \int_{(2\pi\eps)^{-1} Y'} \mathbf{1}_K (\widetilde{\theta}) \mathcal{F}(\vect f) (\widetilde{\theta}) e^{2\pi i \widetilde{\theta} \cdot x} \, d\widetilde{\theta}
        &&\text{Fourier inversion.} \label{eqn:fourier_inversion_compact_support}\\
        &=
        \left( \frac{1}{2\pi \eps} \right)^3 \int_{Y'} \mathbf{1}_K \left( \frac{\eps^{-1}\chi}{2\pi} \right) \mathcal{F}(\vect f) \left( \frac{\eps^{-1}\chi}{2\pi} \right) e^{i \frac{\chi}{\eps} \cdot x} \, d\chi \qquad 
        &&\text{Set $\widetilde{\theta} = \tfrac{\chi}{2\pi \eps}$.}\\
        &=\left( \frac{1}{2\pi \eps} \right)^{3/2} \int_{Y'} \left( \frac{1}{2\pi \eps} \right)^{3/2} \mathbf{1}_K \left( \frac{\eps^{-1}\chi}{2\pi} \right) \mathcal{F}(\vect f) \left( \frac{\eps^{-1}\chi}{2\pi} \right) e^{i \chi \cdot \frac{x}{\eps}} \, d\chi \quad \\
        &= \mathcal{G}_\eps^\ast \left( \left( \frac{1}{2\pi\eps} \right)^{3/2} \mathbf{1}_K \left( \frac{\theta}{2\pi} \right) \mathcal{F}(\vect f) \left( \frac{\theta}{2\pi} \right) \right) (x)
        &&\text{Set $\theta = \eps^{-1} \chi$.} \label{eqn:smooth_f_fourier_cutoff}
    \end{alignat}

    Next, we recall from \cite[Lemma C.1]{simplified_method}, that for $\vect u \in L^2(\R^3;\C^3)$, we have the identity 
    \begin{align}\label{eqn:simplified_mthd_fourier_gelfand_identity}
        \left( \frac{1}{2\pi\eps} \right)^{3/2} \mathcal{F}(\vect u) \left( \frac{\theta}{2\pi} \right)
        = \int_Y (\mathcal{G}_\eps \vect u) (y, \eps\theta) \, dy, \qquad
        \text{for almost every $\theta \in \eps^{-1}Y'$.}
    \end{align}
    We view \eqref{eqn:simplified_mthd_fourier_gelfand_identity} as equality of functions in $y \in Y$ and $\theta \in \eps^{-1}Y'$, where the functions are constant in $y$. 

    Combining \eqref{eqn:smooth_f_fourier_cutoff} and \eqref{eqn:simplified_mthd_fourier_gelfand_identity}, we observe that
    \begin{align}
        \begin{split}
            (\mathcal{G}_\eps \vect f) (y,\chi)
            &\stackrel{\eqref{eqn:smooth_f_fourier_cutoff}}{=} \left( \frac{1}{2\pi\eps} \right)^{3/2} \mathbf{1}_K \left( \frac{\theta}{2\pi} \right) \mathcal{F}(\vect f) \left( \frac{\theta}{2\pi} \right) \\
            &\stackrel{\eqref{eqn:simplified_mthd_fourier_gelfand_identity}}{=} \mathbf{1}_K\left( \frac{\theta}{2\pi} \right) \int_Y (\mathcal{G}_\eps \vect f) (y,\eps\theta) \, dy
            \stackrel{\chi = \eps\theta}{=} \mathbf{1}_{(2\pi \eps) K} \left( \chi \right) \int_Y (\mathcal{G}_\eps \vect f) (y,\chi) \, dy. \label{eqn:smooth_f_fourier_cutoff2}
        \end{split}
    \end{align}
    Note that we have $\widetilde{\theta} \in K$ (see \eqref{eqn:fourier_inversion_compact_support}). Hence by \eqref{eqn:f_support_containment2} and $\widetilde{\theta} = \frac{\chi}{2\pi\eps}$, we have
    \begin{align}
        \chi \in (2\pi\eps) K 
        \subseteq c_\text{supp}[-\eps,\eps]^3 
        \subseteq [-\mu,\mu]^3
        \subseteq Y'.
    \end{align}
    In particular, we see from \eqref{eqn:smooth_f_fourier_cutoff2} that $\mathcal{G}_\eps \vect f$ is constant in $y$, and vanishes outside $\chi \in c_\text{supp} [-\eps,\eps]^3$.
\end{proof}

\paragraph*{Step 2.} We introduce the Bloch approximation operator $\widehat{\Xi}_{\eps,\mu}$. To prepare for this, recall the constant $\mu$ (\ref{eqn:constant_mu_smallfreq}) and the projection $P_\chi$ onto the lowest three eigenspaces of $\mathcal{A}_\chi$ (Definition \ref{defn:proj_achi_eigenspace}). We have the following:

\begin{lemma}\label{lem:measurability}
    The mapping $[-\mu,\mu]^3 \ni \chi \mapsto P_\chi$ is measurable.
\end{lemma}

\begin{proof}
    For $\chi \in [-\mu,\mu]^3$, the first three rescaled eigenvalues $\tfrac{1}{|\chi|^2} \lambda_1^\chi$, $\tfrac{1}{|\chi|^2} \lambda_2^\chi$, $\tfrac{1}{|\chi|^2} \lambda_3^\chi$ of $\tfrac{1}{|\chi|^2} \mathcal{A}_\chi$ are enclosed by the contour $\Gamma$. The remaining eigenvalues $\tfrac{1}{|\chi|^2} \lambda_i^\chi$, $i\geq 4$, lie outside the contour (Lemma \ref{lem:contour_existence}). So there exists a constant $C > 0$ such that
    \begin{equation}
        P_\chi = \mathbf{1}_{[0, C]} \left( \mathcal{A}_\chi \right), \qquad \text{for $\chi \in [-\mu,\mu]^3$}.
    \end{equation}
    The conclusion now follows from the measurability of $\mathcal{A}_\chi$ in $\chi$ and \cite[Proposition V.1.2]{carmona_lacroix_book}.
\end{proof}

With Lemma \ref{lem:measurability} in mind, we can now make the following definition:

\begin{definition}\label{defn:bloch_approximation}
    Given $\eps>0$, the \textbf{Bloch approximation operator} $\widehat{\Xi}_{\eps,\mu} : L^2(\R^3;\C^3) \to L^2(\R^3;\C^3)$ is
    \begin{align}\label{eqn:bloch_approx}
        \widehat{\Xi}_{\eps,\mu} \vect f := \mathcal{G}_\eps^\ast \left( \int_{Y'}^\oplus \mathbf{1}_{[-\mu,\mu]^3} (\chi) P_\chi P_0 \, d\chi \right) \mathcal{G}_\eps \vect f.
    \end{align}
    We shall refer to the function $\vect f_\eps = \widehat{\Xi}_{\eps,\mu} \vect f$ as the \textbf{Bloch approximation/Bloch adapted function} (of $\vect f$), and $\range{(\widehat{\Xi}_{\eps,\mu})}$ as the \textbf{set of Bloch adapted functions}.
\end{definition}

The notion of \textbf{``well-prepared data"} that we shall adopt in this paper is as follows:

\begin{assumption}[Well-prepared data] \label{assump:rhs}
    $\vect f_\eps = \widehat{\Xi}_{\eps,\mu}\vect f$ is a Bloch adapted function of some $\vect f$ which has \newline compact Fourier support.
\end{assumption}

\begin{remark}[Alternate formulae] \label{rmk:rhs_alt_formula_blochapprox}
    Under Assumption \ref{assump:rhs}, the Bloch approximation $\widehat{\Xi}_{\eps,\mu} \vect f$ of $\vect f$ has several alternate expressions:
    \begin{alignat}{3}
        \widehat{\Xi}_{\eps,\mu} \vect f
        &= \mathcal{G}_\eps^\ast \left( \int_{c_\text{supp}[-\eps,\eps]^3}^\oplus P_\chi P_0 \, d\chi \right) \mathcal{G}_\eps \vect f
        \qquad\qquad
        &&\text{By Proposition \ref{prop:compact_fourier_support_properties}(ii), for $\eps$ small enough.} \label{eqn:bloch_approx_alt1}\\
        &= \mathcal{G}_\eps^\ast \left( \int_{[-\mu,\mu]^3}^\oplus P_\chi \, d\chi \right) \mathcal{G}_\eps \Xi_\eps \vect f
        &&\text{By Definition \ref{defn:smoothing_operator}.} \label{eqn:bloch_approx_alt2}
    \end{alignat}
    Since, $\supp(\mathcal{F}\vect f)$ is assumed to be compact, $P_0$ and $\Xi_\eps$ may be omitted from \eqref{eqn:bloch_approx_alt1} and \eqref{eqn:bloch_approx_alt2} respectively. However, we feel that it is more instructive to write \eqref{eqn:bloch_approx}-\eqref{eqn:bloch_approx_alt2}, as it emphasizes that we have taken an arbitrary $L^2$ function through two stages of ``smoothing": First by cutting off its support in the frequency space. Second by removing its contribution in the ``bulk" of the spectrum $\sigma(\mathcal{A}_\chi)$.     
\end{remark}

\begin{remark}[More alternate formulae] \label{rmk:more_alt_formulae_blochapprox}
    Moreover, for a general $\vect f \in L^2(\R^3;\C^3)$, we also have
    \begin{alignat}{3}
        \widehat{\Xi}_{\eps,\mu} \vect f (x)
        &= \sum_{i=1}^3 \int_{(2\pi\eps)^{-1}Y'} \left[ \mathcal{F} \vect f (\widetilde{\theta})
        \cdot \overline{\vect m_i \left( (2\pi\eps)\widetilde{\theta} \right)} \, \right]
        \vect \omega_i \left( \frac{x}{\eps}, (2\pi\eps) \widetilde{\theta} \right)
        e^{i2\pi \widetilde{\theta} \cdot x}  
        \, d\widetilde{\theta}.
        \label{eqn:bloch_approx_alt3}
    \end{alignat}
    where $\vect \omega_i(\cdot,\chi)$, denotes the eigenfunction corresponding to the $i$-th eigenvalue of $\mathcal{A}_\chi$, and the vector $\vect m_i(\chi)$ denotes the mean $\int_Y \vect \omega_i(y,\chi) \,dy$ (taken component-wise).
    
    To see this, fix $\chi = \eps\theta \in Y'$. Then we compute:
    \begin{alignat}{3}
        (P_\chi P_0 \mathcal{G}_\eps \vect f) (y,\eps\theta)
        &= \sum_{i=1}^3 \left( \int_Y (P_0\mathcal{G}_\eps \vect f)(y,\eps\theta) \cdot \overline{\vect \omega_i(y,\chi)} \, dy \right) \vect \omega_i (y,\chi) \\
        &= \sum_{i=1}^3 \frac{1}{(2\pi\eps)^{3/2}} \mathcal{F} \vect f \left( \frac{\theta}{2\pi} \right) \cdot \underbrace{\left( \int_Y \overline{\vect \omega_i(y,\eps\theta)} \, dy \right)}_{=\overline{\vect m_i(\eps\theta)}} \vect \omega_i (y,\eps \theta). \qquad
        &&\text{By \eqref{eqn:simplified_mthd_fourier_gelfand_identity}.} \label{eqn:compute_bloch_approx_step1}
    \end{alignat}
    Let us refer to the function in \eqref{eqn:compute_bloch_approx_step1} as $\vect g (y,\chi)$. Then,
    \begin{alignat}{3}
        (\mathcal{G}_\eps^\ast \vect g) (x)
        &= \sum_{i=1}^3 \frac{1}{(2\pi\eps)^{3/2}} \int_{Y'} \frac{1}{(2\pi\eps)^{3/2}} \mathcal{F} \vect f \left(\frac{\chi/\eps}{2\pi} \right) \cdot \overline{\vect m_i(\chi)} \vect \omega_i \left(\frac{x}{\eps},\chi \right) e^{i\chi \cdot \frac{x}{\eps}} \, d\chi \qquad \\
        &= \sum_{i=1}^3 \frac{1}{(2\pi\eps)^3} \int_{Y'} \mathcal{F} \vect f \left( \frac{\chi/\eps}{2\pi} \right) \cdot \overline{\vect m_i(\chi)} \vect \omega_i \left( \frac{x}{\eps}, \chi \right) e^{i \frac{\chi}{\eps} \cdot x}  \, d\chi \\
        &= \sum_{i=1}^3 \int_{(2\pi\eps)^{-1}Y'} \mathcal{F} \vect f (\widetilde{\theta}) \cdot \overline{\vect m_i\left( (2\pi\eps)\widetilde{\theta} \right)}
        \vect \omega_i \left( \frac{x}{\eps}, (2\pi\eps) \widetilde{\theta} \right) e^{i2\pi \widetilde{\theta} \cdot x} 
        \, d\widetilde{\theta},
        &&\text{Set $\widetilde{\theta} = \frac{\chi}{2\pi\eps}$.}
    \end{alignat}
    as required.
\end{remark}

\newpage
\section{Fibre-wise asymptotics} \label{sect:fibrewise_asmptotics}
Fix $\chi \in Y'\setminus \{ 0 \}$, $z \in \rho(\tfrac{1}{|\chi|^2}\mathcal{A}_\chi) \cap \rho(\tfrac{1}{|\chi|^2}\mathcal{A}_\chi^{\hom} )$, and $\vect f \in L^2(Y;\C^3)$. The goal of this section is to describe the algorithm behind the iterative construction of the higher-order approximations of the unique solution $\vect u \in H_{\#}^1$ to the resolvent problem for $\frac{1}{|\chi|^2} \mathcal{A}_\chi$:
\begin{equation}\label{eqn:fiberwiseresolventproblem}
    \frac{1}{|\chi|^2} \int_{Y} \A \left( \simgrad + iX_\chi \right) \vect u  : \overline{\left( \simgrad + iX_\chi \right)  \vect v} -z \int_Y \vect u \cdot \overline{\vect v} = \int_Y \vect f \cdot \overline{\vect v} ,\quad \forall \, \vect v\in H_\#^1(Y;\C^3). 
\end{equation}
After providing the algorithm (Section \ref{sect:definition_of_algorithm}), we inductively prove the formulae for the terms in the asymptotic expansion together with the formulae for the error term. (Section \ref{sect:algebraic_step}). Next, we provide the estimates for the terms, as well as the  error estimates (Section \ref{sect:fibrewise_estimates_working}), thus justifying the heuristic guiding the Algorithm \ref{defn:algorithm}. We refer to the iterations of the algorithm as ``cycles". In each cycle, three new terms are constructed, improving the order of approximation in  $\chi$. We shall keep track of the dependence of the error estimates both on the spectral parameter $z$ and $\chi$.

\subsection{Heuristics of the inductive procedure}\label{sect:definition_of_algorithm}

The following is the definition of the algorithm used to construct the approximations. 



\begin{algorithm}[Asymptotic procedure] \label{defn:algorithm} 
    At \textbf{cycle $n \geq 0$},
    \begin{enumerate}[leftmargin=*, label=\textbf{Step \arabic*}.]
        \item Consider an expansion of $\vect u$ as follows:
        \begin{align}\label{eqn:u_expansion_cyclei}
            \vect u = \sum_{k=0}^n (\vect u_0^{(k)} + \vect u_1^{(k)} + \vect u_2^{(k)}) + \vect u_\text{error}^{(n)}.
        \end{align}

        \item Substitute \eqref{eqn:u_expansion_cyclei} into the resolvent equation \eqref{eqn:fiberwiseresolventproblem}  to get
        \begin{align}\label{eqn:resolvent_equation_alg}
            \begin{split}
                &\sum_{k=0}^n \sum_{j=0}^2 \frac{1}{|\chi|^2} \int_Y \A (\simgrad + iX_\chi) \left( \vect u_j^{(k)} + \vect u_\text{error}^{(n)} \right) : 
                \overline{(\simgrad + iX_\chi) \vect v}
                - z \int_Y \left( \vect u_j^{(k)} + \vect u_\text{error}^{(n)} \right) \cdot \overline{\vect v} \\
                &\qquad\qquad= \int_Y \vect f \cdot \overline{\vect v},
                \qquad \forall \vect v \in H_{\#}^1(Y;\C^3).
            \end{split}
        \end{align}

        \item \textbf{(For $n \geq 1$)} Using the equations for $\vect u_0^{(k)}$, $\vect u_1^{(k)}$, and $\vect u_2^{(k)}$ as defined in \textbf{Step 4} of cycles $k=0,\dots, n-1$, cancel out the terms in \eqref{eqn:resolvent_equation_alg} as appropriate.

        \item With the remaining terms in \eqref{eqn:resolvent_equation_alg}, collect like powers of $|\chi|$ according to the heuristic:
        \begin{quote}\label{power_counting_heuristic}
            $\vect u_j^{(k)} = O(|\chi|^{j+k})$. An instance of $iX_\chi$ counts for one order of $|\chi|$. Ignore test functions.
        \end{quote}
        Thus, for instance, $\int_Y \A \simgrad \vect u_1^{(1)} : \overline{iX_\chi \vect v}$ is treated as a term of order $|\chi|^3$.
        In particular, terms of order $|\chi|^{n-2}$, $|\chi|^{n-1}$, and $|\chi|^{n}$ give equations for $\vect u_0^{(n)}$, $\vect u_1^{(n)}$, and $\vect u_2^{(n)}$ respectively.

        \item Collect the leftover terms in \eqref{eqn:resolvent_equation_alg} to obtain the equation for $\vect u_\text{error}^{(n)}$.

        \item Proceed to the \textbf{next cycle, $n+1$}, beginning with \textbf{Step 1}.
    \end{enumerate}

\end{algorithm}

\subsection{Formulae for terms in the \texorpdfstring{$n$-th}{n-th} cycle}\label{sect:algebraic_step}
By applying Algorithm \ref{defn:algorithm}, we claim that we arrive at the following well-posed problems for the terms $\vect u_j^{(n)}$, $j=0,1,2$, $n\geq 0$.

\subsubsection{Cycle \texorpdfstring{$0$}{0}}
For $z \in \rho(\mathcal{A}_\chi^{\rm hom})$, $\chi \in Y'$, the following well-posed problems make up the cycle $0$ of the approximation:

\begin{equation}\label{eqn:cycle_i0}
    \begin{split}
        \int_{Y} \A  \simgrad   \vect u_0^{(0)}  : \overline{ \simgrad \vect v}  = & 0, \quad \forall\vect v \in H^{1}_\#(Y;\C^3), \\
        \int_{Y} \A  \simgrad   \vect u_1^{(0)}  : \overline{ \simgrad   \vect v}  = & - \int_{Y} \A  iX_\chi   \vect u_0^{(0)}  : \overline{ \simgrad   \vect v}, \quad \forall\vect v \in H^{1}_\#(Y;\C^3), \quad \text{ and } \quad \quad \int_Y \vect u_1^{(0)} = 0, \\
            \int_{Y} \A  \simgrad   \vect u_2^{(0)}  : \overline{ \simgrad   \vect v} = 
    & - \int_{Y} \A  iX_\chi    \vect u_1^{(0)}  : \overline{ \simgrad  \vect v} 
    - \int_{Y} \A  \simgrad   \vect u_1^{(0)}  : \overline{ i X_\chi   \vect v} - \int_{Y} \A  iX_\chi    \vect u_0^{(0)}  : \overline{ i X_\chi   \vect v} \\
    &+z |\chi|^2\int_Y \vect u_0^{(0)} \cdot\overline{\vect v} + |\chi|^2\int_Y \vect f  \cdot \overline{\vect v}, 
    \quad \forall\vect v \in H^{1}_\#(Y;\C^3), \quad \text{ and } \quad \int_Y \vect u_2^{(0)} = 0, \\
         \frac{1}{|\chi|^2}\left\langle \mathcal{A}_\chi^{\rm hom} \vect u_0^{(0)},\vect v_0 \right\rangle_{\C^3} & -z \int_Y \vect u_0^{(0)} \cdot \overline{\vect v_0} =  \int_Y \vect f \cdot \overline{\vect v_0}, 
    \quad \forall\vect v_0 \in \C^3.
    \end{split}
\end{equation}

\subsubsection{Cycle \texorpdfstring{$1$}{1}}

For $z \in \rho(\mathcal{A}_\chi^{\rm hom})$, $\chi \in Y'$, the following well-posed problems make up the cycle $1$ of the approximation:

\begin{equation}\label{eqn:cycle_i1}
    \begin{split}
        \int_{Y} \A  \simgrad   \vect u_0^{(1)}  : \overline{ \simgrad \vect v}  = & 0, \quad \forall\vect v \in H^{1}_\#(Y;\C^3), \\
        \int_{Y} \A  \simgrad   \vect u_1^{(1)}  : \overline{ \simgrad   \vect v}  = & - \int_{Y} \A  iX_\chi   \vect u_0^{(1)}  : \overline{ \simgrad   \vect v}, \quad \forall\vect v \in H^{1}_\#(Y;\C^3), \quad \text{ and } \quad \quad \int_Y \vect u_1^{(1)} = 0, \\
            \int_{Y} \A  \simgrad   \vect u_2^{(1)}  : \overline{ \simgrad   \vect v} = 
    & - \int_{Y} \A  iX_\chi  \left(  \vect u_1^{(1)} + u_2^{(0)} \right) : \overline{ \simgrad  \vect v} 
    - \int_{Y} \A  \simgrad   \left(  \vect u_1^{(1)} + u_2^{(0)} \right)  : \overline{ i X_\chi   \vect v} \\ & - \int_{Y} \A  iX_\chi  \left(  \vect u_0^{(1)} + \vect u_1^{(0)} \right)  : \overline{ i X_\chi   \vect v} +z |\chi|^2\int_Y \left(  \vect u_0^{(1)} + \vect u_1^{(0)} \right)\cdot\overline{\vect v}, \\ & 
    \quad \forall\vect v \in H^{1}_\#(Y;\C^3), \quad \text{ and } \quad \int_Y \vect u_2^{(1)} = 0, \\
         \frac{1}{|\chi|^2}\left\langle \mathcal{A}_\chi^{\rm hom} \vect u_0^{(1)},\vect v_0 \right\rangle_{\C^3} & -z \int_Y \vect u_0^{(1)} \cdot \overline{\vect v_0} =  -\frac{1}{|\chi|^2}\int_Y \A \left(\simgrad \vect u_2^{(0)} + iX_\chi \vect u_1^{(0)} \right): \overline{iX_\chi\vect v_0}, 
    \quad \forall\vect v_0 \in \C^3.
    \end{split}
\end{equation}

\subsubsection{Inductive step, cycle \texorpdfstring{$n$}{n}, \texorpdfstring{$n\geq 2$}{n geq 2}}

For $z \in \rho(\mathcal{A}_\chi^{\rm hom})$, $\chi \in Y'$, $n \geq 2$ the following well-posed problems make up the inductive step of the approximation:

\begin{equation}
\label{inductivestep}
    \begin{split}
        \int_{Y} \A  \simgrad   \vect u_0^{(n)}  : \overline{ \simgrad \vect v}  = & 0, \quad \forall\vect v \in H^{1}_\#(Y;\C^3), \\
        \int_{Y} \A  \simgrad   \vect u_1^{(n)} : \overline{ \simgrad   \vect v}  = & - \int_{Y} \A  iX_\chi   \vect u_0^{(n)}  : \overline{ \simgrad   \vect v}, \quad \forall\vect v \in H^{1}_\#(Y;\C^3), \quad \text{ and } \quad \quad \int_Y \vect u_1^{(n)} = 0, \\
            \int_{Y} \A  \simgrad   \vect u_2^{(n)}  : \overline{ \simgrad   \vect v} = 
    & - \int_{Y} \A  iX_\chi  \left(  \vect u_1^{(n)} + u_2^{(n-1)} \right) : \overline{ \simgrad  \vect v} 
    - \int_{Y} \A  \simgrad   \left(  \vect u_1^{(n)} + u_2^{(n-1)} \right)  : \overline{ i X_\chi   \vect v} \\ & - \int_{Y} \A  iX_\chi  \left(  \vect u_0^{(n)} + \vect u_1^{(n-1)}+ \vect u_2^{(n-2)} \right)  : \overline{ i X_\chi   \vect v} \\ & +z |\chi|^2\int_Y \left(  \vect u_0^{(n)} + \vect u_1^{(n-1)}+ \vect u_2^{(n-2)} \right)  \cdot\overline{\vect v},  
    \quad \forall\vect v \in H^{1}_\#(Y;\C^3),  \ \int_Y \vect u_2^{(n)} = 0, \\
         \frac{1}{|\chi|^2}\left\langle \mathcal{A}_\chi^{\rm hom} \vect u_0^{(n)},\vect v_0 \right\rangle_{\C^3} & -z \int_Y \vect u_0^{(n)} \cdot \overline{\vect v_0} =  -\frac{1}{|\chi|^2}\int_Y \A \left(\simgrad \vect u_2^{(n-1)} + iX_\chi \left(\vect u_1^{(n-1)} +\vect u_2^{(n-2)}\right)\right): \overline{iX_\chi\vect v_0}, 
     \\ & \forall\vect v_0 \in \C^3.
    \end{split}
\end{equation}
We shall also show that with the terms $\vect u_j^{(k)}$ as defined in \eqref{inductivestep}, the error term at the $n$-th cycle,
    \begin{equation}\label{eqn:errortermdefinition}
        \vect u_{\rm error}^{(n)} := \vect u - \sum_{k=0}^n \left(\vect u_0^{(k)} + \vect u_1^{(k)} + \vect u_2^{(k)} \right)
    \end{equation}
satisfies the following equation 
\begin{equation}\label{eqn:inductivestep_remainder_1}
        \frac{1}{|\chi|^2} \int_Y \A \left( \simgrad + iX_\chi \right) \vect u_{\rm error}^{(n)}  : \overline{\left( \simgrad + iX_\chi \right) \vect v} 
        -z \int_Y \vect u_{\rm error}^{(n)} \cdot \overline{\vect v}
        = \frac{1}{|\chi|^2}\mathcal{R}_{\rm error}^{(n)}(\vect v) ,\quad \forall \vect v\in H^1(Y;\C^3), 
\end{equation}
where the functional $\mathcal{R}_{\rm error}^{(n)}$ is given by
\begin{equation}\label{eqn:inductivestep_remainder_2}
\begin{split}
    \mathcal{R}_{\rm error}^{(n)} (\vect v) := 
    &- \int_Y \A iX_\chi  \vect u_2^{(n-1)}  :  \overline{ iX_\chi  \vect v}
    - \int_Y \A iX_\chi  \vect u_1^{(n)}  :  \overline{ iX_\chi  \vect v}
    - \int_Y \A \simgrad  \vect u_2^{(n)}  :  \overline{ iX_\chi  \vect v} 
    - \int_Y \A iX_\chi  \vect u_2^{(n)}  :  \overline{ \simgrad  \vect v} \\
    &- \int_Y \A iX_\chi  \vect u_2^{(n)}  :  \overline{ iX_\chi  \vect v}
    +z |\chi|^2\int_Y \vect u_2^{(n-1)} \cdot \overline{\vect v} 
    +z |\chi|^2\int_Y \vect u_1^{(n)} \cdot \overline{\vect v} 
    +z |\chi|^2\int_Y \vect u_2^{(n)} \cdot \overline{\vect v}.
\end{split}
\end{equation}
%
%








\begin{theorem}\label{thm:ncycle_formulae}
    Consider the asymptotic procedure as described in Algorithm \ref{defn:algorithm}. Then $\vect u_0^{(0)}$, $\vect u_1^{(0)}$, $\vect u_2^{(0)}$ are uniquely determined by \eqref{eqn:cycle_i0}, $\vect u_0^{(1)}$, $\vect u_1^{(1)}$, $\vect u_2^{(1)}$ are uniquely determined by \eqref{eqn:cycle_i1}, and for $n\geq 2$, $\vect u_0^{(n)}$, $\vect u_1^{(n)}$, and $\vect u_2^{(n)}$ are defined by the well-posed equations \eqref{inductivestep}, and the error $\vect u_\text{error}^{(n)}$ satisfies \eqref{eqn:inductivestep_remainder_1}-\eqref{eqn:inductivestep_remainder_2}.
\end{theorem}

The remainder of this section is devoted to the proof of Theorem \ref{thm:ncycle_formulae}, for $n\geq 2$. This will be done in an inductive manner, starting with the base case $n=2$. The cases $n=0$ and $n=1$ are covered in \cite{simplified_method}. 

\subsubsection{Cycle \texorpdfstring{$2$}{2}}\label{sect:basecase_formulae}

The resolvent equation reads:
\begin{align}
    \begin{split}
        &\frac{1}{|\chi|^2} \int_Y \A (\simgrad + iX_\chi) \left( \sum_{k=0}^2 (\vect u_0^{(k)} + \vect u_1^{(k)} + \vect u_2^{(k)} ) + \vect u_\text{error}^{(2)} \right) : 
        \overline{(\simgrad + iX_\chi) \vect v} \\
        &\qquad\qquad\qquad\qquad
        - z \int_Y \left( \sum_{k=0}^2 (\vect u_0^{(k)} + \vect u_1^{(k)} + \vect u_2^{(k)} ) + \vect u_\text{error}^{(2)} \right) \cdot \overline{\vect v}
        = \int_Y \vect f \cdot \overline{\vect v},
        \qquad \forall \vect v \in H_{\#}^1(Y;\C^3).
    \end{split}
\end{align}
Rearrange the above to get
\begin{align}\label{eqn:resolvent_i2}
    \begin{split}
        &\sum_{k=0}^2 \sum_{j=0}^2 
        \frac{1}{|\chi|^2} \int_Y \A \simgrad \vect u_j^{(k)} : \overline{\simgrad \vect v}
        + \frac{1}{|\chi|^2}  \int_Y \A \simgrad \vect u_j^{(k)} : \overline{iX_\chi \vect v}\\
        &\qquad+ \frac{1}{|\chi|^2}  \int_Y \A iX_\chi \vect u_j^{(k)} : \overline{\simgrad \vect v}
        + \frac{1}{|\chi|^2}  \int_Y \A iX_\chi \vect u_j^{(k)} : \overline{iX_\chi \vect v}
        - z \int_Y \vect u_j^{(k)} \cdot \overline{\vect v} \\
        &\qquad+ \frac{1}{|\chi|^2} \int_Y \A (\simgrad + iX_\chi) \vect u_\text{error}^{(2)} \cdot \overline{(\simgrad + iX_\chi) \vect v}
        - z \int_Y \vect u_\text{error}^{(2)} \cdot \overline{\vect v}
        = \int_Y \vect f \cdot \overline{\vect v}
    \end{split}
\end{align}

\noindent
Here and henceforth we shall omit the quantifier ``$\text{for all } \vect v ~\text{in}~ H_{\#}^1(Y;\C^3)$" to keep the expressions compact.

By the equations \eqref{eqn:cycle_i0} defining $\vect u_0^{(0)}$, $\vect u_1^{(0)}$, and $\vect u_2^{(0)}$, the resolvent equation \eqref{eqn:resolvent_i2} becomes
\begin{equation}\label{eqn:resolvent_i2_v2}
    \begin{split}
    \begin{alignedat}{3}
        & &~~\frac{1}{|\chi|^2} \int_Y \A iX_\chi \vect u_1^{(0)} : \overline{iX_\chi \vect v}
        - z \int_Y \vect u_1^{(0)} \cdot \overline{\vect v} \\
        &+ \frac{1}{|\chi|^2} \int_Y \A \simgrad \vect u_2^{(0)} : \overline{iX_\chi \vect v}
        + \frac{1}{|\chi|^2} \int_Y \A iX_\chi \vect u_2^{(0)} : \overline{\simgrad \vect v}
        &~+ \frac{1}{|\chi|^2} \int_Y \A iX_\chi \vect u_2^{(0)} : \overline{iX_\chi \vect v}
        - z \int_Y \vect u_2^{(0)} \cdot \overline{\vect v} \\
        &+\sum_{k=1}^2 \sum_{j=0}^2 
        \frac{1}{|\chi|^2} \int_Y \A \simgrad \vect u_j^{(k)} : \overline{\simgrad \vect v} \\
        &+ \frac{1}{|\chi|^2}  \int_Y \A \simgrad \vect u_j^{(k)} : \overline{iX_\chi \vect v}
        + \frac{1}{|\chi|^2}  \int_Y \A iX_\chi \vect u_j^{(k)} : \overline{\simgrad \vect v}
        &~+ \frac{1}{|\chi|^2}  \int_Y \A iX_\chi \vect u_j^{(k)} : \overline{iX_\chi \vect v}
        - z \int_Y \vect u_j^{(k)} \cdot \overline{\vect v} \\
        &+ \frac{1}{|\chi|^2} \int_Y \A (\simgrad + iX_\chi) \vect u_\text{error}^{(2)} \cdot \overline{(\simgrad + iX_\chi) \vect v}
        - z \int_Y \vect u_\text{error}^{(2)} \cdot \overline{\vect v}
        = 0 \span
    \end{alignedat}
    \end{split}
\end{equation}
We have arranged the above equation as follows: The first two lines consists of terms from cycle $0$. The next two lines consists of terms from cycles $n=1,2$. The final line involves $\vect u_\text{error}^{(2)}$ and the RHS.

By the equations \eqref{eqn:cycle_i1} defining $\vect u_0^{(1)}$, $\vect u_1^{(1)}$, and $\vect u_2^{(1)}$, the resolvent equation \eqref{eqn:resolvent_i2_v2} becomes

\begin{equation}\label{eqn:resolvent_i2_v3}
    \begin{split}
    \begin{alignedat}{3}
        & &~~\frac{1}{|\chi|^2} \int_Y \A iX_\chi \vect u_2^{(0)} : \overline{iX_\chi \vect v}
        - z \int_Y \vect u_2^{(0)} \cdot \overline{\vect v} \\
        & &~+\frac{1}{|\chi|^2} \int_Y \A iX_\chi \vect u_1^{(1)} : \overline{iX_\chi \vect v}
        - z \int_Y \vect u_1^{(1)} \cdot \overline{\vect v} \\
        &+ \frac{1}{|\chi|^2} \int_Y \A \simgrad \vect u_2^{(1)} : \overline{iX_\chi \vect v}
        + \frac{1}{|\chi|^2} \int_Y \A iX_\chi \vect u_2^{(1)} : \overline{\simgrad \vect v}
        &~+ \frac{1}{|\chi|^2} \int_Y \A iX_\chi \vect u_2^{(1)} : \overline{iX_\chi \vect v}
        - z \int_Y \vect u_2^{(1)} \cdot \overline{\vect v} \\
        &+\sum_{j=0}^2 
        \frac{1}{|\chi|^2} \int_Y \A \simgrad \vect u_j^{(2)} : \overline{\simgrad \vect v} \\
        &+ \frac{1}{|\chi|^2}  \int_Y \A \simgrad \vect u_j^{(2)} : \overline{iX_\chi \vect v}
        + \frac{1}{|\chi|^2}  \int_Y \A iX_\chi \vect u_j^{(2)} : \overline{\simgrad \vect v}
        &~+ \frac{1}{|\chi|^2}  \int_Y \A iX_\chi \vect u_j^{(2)} : \overline{iX_\chi \vect v}
        - z \int_Y \vect u_j^{(2)} \cdot \overline{\vect v} \\
        &+ \frac{1}{|\chi|^2} \int_Y \A (\simgrad + iX_\chi) \vect u_\text{error}^{(2)} \cdot \overline{(\simgrad + iX_\chi) \vect v}
        - z \int_Y \vect u_\text{error}^{(2)} \cdot \overline{\vect v}
        = 0 \span
    \end{alignedat}
    \end{split}
\end{equation}
The first line consist of terms from cycle $0$. The next two lines consists of terms from cycle $1$. The next two lines consists of terms from cycle $2$. The final line involves $\vect u_\text{error}^{(2)}$ and the RHS.

We now proceed to collect $O(1)$ and $O(|\chi|)$ terms, which only appears in the fourth and fifth lines.

Collecting the $O(1)$ terms gives us an equation for $\vect u_0^{(2)}$:
\begin{align}\label{eqn:u02}
    \frac{1}{|\chi|^2} \int_Y \A \simgrad \vect u_0^{(2)} : \overline{\simgrad \vect v} = 0.
\end{align}

Collecting the $O(|\chi|)$ terms gives us an equation for $\vect u_1^{(2)}$:
\begin{align}\label{eqn:u12}
    \frac{1}{|\chi|^2} \int_Y \A (\simgrad \vect u_1^{(2)} + iX_\chi \vect u_0^{(2)} ) : \overline{\simgrad \vect v} = 0.
\end{align}
(Note that the $O(|\chi|)$ term $\int_Y \A \simgrad \vect u_0^{(2)} : \overline{iX_\chi \vect v}$ is zero, as $\vect u_0^{(2)}$ is a constant, by \eqref{eqn:u02}.) This has a unique solution in the space $\dot{H}_{\#}(Y;\C^3)$.

Collecting the $O(|\chi|^2)$ terms gives us an equation for $\vect u_2^{(2)}$:
\begin{align}\label{eqn:u22}
    \begin{split}
        &\frac{1}{|\chi|^2} \int_Y \A \simgrad \vect u_2^{(2)} : \overline{\simgrad \vect v} \\
        &\quad = - \frac{1}{|\chi|^2} \int_Y \A \simgrad (\vect u_1^{(2)} + \vect u_2^{(1)}) : \overline{iX_\chi \vect v}
        - \frac{1}{|\chi^2} \int_Y \A iX_\chi (\vect u_1^{(2)} + \vect u_2^{(1)}) : \overline{\simgrad \vect v} \\
        &\qquad + \int_Y \A iX_\chi (\vect u_0^{(2)} + \vect u_1^{(1)} + \vect u_2^{(0)}) : \overline{iX_\chi \vect v}
        + z \int_Y (\vect u_0^{(2)} + \vect u_1^{(1)} + \vect u_2^{(0)}) \cdot \overline{\vect v}.
    \end{split}
\end{align}
\eqref{eqn:u22} has a unique solution in $\dot{H}_{\#}$ if and only if its RHS vanishes for all $\vect v \in \C^3$. By the definition of $\mathcal{A}_\chi^{\hom}$, this criterion can be written as
\begin{align}\label{eqn:u02_homogenized}
    \left\langle \left(\frac{1}{|\chi|^2} \mathcal{A}_\chi^{\hom} - zI_{\C^3} \right) \vect u_0^{(2)}, \vect v_0 \right\rangle_{\C^3} =
    -\frac{1}{|\chi|^2} \int_Y \A \left( \simgrad \vect u_2^{(1)} + iX_\chi (\vect u_1^{(1)} + \vect u_2^{(0)}) \right)  : \overline{iX_\chi \vect v_0}, \qquad
    \forall \vect v_0 \in \C^3.
\end{align}
Since $z \in \rho(\tfrac{1}{|\chi|^2} \mathcal{A}_\chi^{\hom} )$, the constant vector $\vect u_0^{(2)}$ is uniquely chosen by \eqref{eqn:u02_homogenized}. Consequently, $\vect u_1^{(2)}$ and $\vect u_2^{(2)}$ are uniquely determined by \eqref{eqn:u12} and \eqref{eqn:u22}.

Finally, we collect the leftover terms to get an equation for $\vect u_\text{error}^{(2)}$:
\begin{equation}\label{eqn:uerr2}
    \begin{split}
    \begin{alignedat}{3}
        &\frac{1}{|\chi|^2} \int_Y \A (\simgrad + iX_\chi) \vect u_\text{error}^{(2)} : \overline{(\simgrad + iX_\chi) \vect v }
        - z \int_Y \vect u_\text{error}^{(2)} \cdot \overline{\vect v}
        \span \\
        &= -\frac{1}{|\chi|^2} \int_Y \A iX_\chi \vect u_2^{(1)} : \overline{iX_\chi \vect v}
        + z \int_Y \vect u_2^{(1)} \cdot \overline{\vect v}
        & \\
        &- \frac{1}{|\chi|^2} \int_Y \A iX_\chi \vect u_1^{(2)} : \overline{iX_\chi \vect v}
        + z \int_Y \vect u_1^{(2)} \cdot \overline{\vect v} & \\
        &- \frac{1}{|\chi|^2} \int_Y \A iX_\chi \vect u_2^{(2)} : \overline{iX_\chi \vect v}
        + z \int_Y \vect u_2^{(2)} \cdot \overline{\vect v}
        &- \frac{1}{|\chi|^2} \int_Y \A \simgrad \vect u_2^{(2)} : \overline{iX_\chi \vect v}
        - \frac{1}{|\chi|^2} \int_Y \A iX_\chi \vect u_2^{(2)} : \overline{\simgrad \vect v} \\
        &=: \frac{1}{|\chi|^2} \mathcal{R}_\text{error}^{(2)}(\vect v).
    \end{alignedat}
    \end{split}
\end{equation}
We have arranged \eqref{eqn:uerr2} such that the first line of the RHS contains terms from cycle $1$, and the next two lines contains terms from cycle $2$. There are no contribution from cycle $0$.

We have thus obtained the problems defining $\vect u_0^{(2)}$, $\vect u_1^{(2)}$, $\vect u_2^{(2)}$, and $\vect u_\text{error}^{(2)}$, namely \eqref{eqn:u02}-\eqref{eqn:uerr2}. Cycle $2$ is complete.

\subsubsection{Cycle \texorpdfstring{$\geq 2$}{>= 2}}\label{sect:inductivecase_formulae}

We shall now proceed with the inductive step. Assume that \eqref{inductivestep}, \eqref{eqn:inductivestep_remainder_1}, and \eqref{eqn:inductivestep_remainder_2} holds for $k=0,\cdots,n-1$, where $n\geq 3$. Consider the expansion of $\vect u$ at cycle $n$:
\begin{align}
   \vect u = \sum_{k=0}^n (\vect u_0^{(k)} + \vect u_1^{(k)} + \vect u_2^{(k)}) + \vect u_\text{error}^{(n)}.
\end{align}
Substitute this expansion into the resolvent equation and rearrange, 
\begin{align}\label{eqn:resolvent_i}
    \begin{split}
        &\sum_{k=0}^n \sum_{j=0}^2 
        \frac{1}{|\chi|^2} \int_Y \A \simgrad \vect u_j^{(k)} : \overline{\simgrad \vect v}
        + \frac{1}{|\chi|^2}  \int_Y \A \simgrad \vect u_j^{(k)} : \overline{iX_\chi \vect v}\\
        &\qquad+ \frac{1}{|\chi|^2}  \int_Y \A iX_\chi \vect u_j^{(k)} : \overline{\simgrad \vect v}
        + \frac{1}{|\chi|^2}  \int_Y \A iX_\chi \vect u_j^{(k)} : \overline{iX_\chi \vect v}
        - z \int_Y \vect u_j^{(k)} \cdot \overline{\vect v} \\
        &\qquad+ \frac{1}{|\chi|^2} \int_Y \A (\simgrad + iX_\chi) \vect u_\text{error}^{(n)} \cdot \overline{(\simgrad + iX_\chi) \vect v}
        - z \int_Y \vect u_\text{error}^{(n)} \cdot \overline{\vect v}
        = \int_Y \vect f \cdot \overline{\vect v}.
    \end{split}
\end{align}
Again, we shall omit the quantifier ``$\text{for all } \vect v ~\text{in}~ H_{\#}^1(Y;\C^3)$" to keep the expressions compact.

By the equations for $\vect u_0^{(k)}$, $\vect u_1^{(k)}$, $\vect u_2^{(k)}$, $\vect u_\text{error}^{(k)}$, for $k=0,\cdots,n-1$, namely \eqref{inductivestep}, \eqref{eqn:inductivestep_remainder_1}, and \eqref{eqn:inductivestep_remainder_2} (the inductive hypothesis), we claim that we are left with
\begin{equation}\label{eqn:resolvent_i_v2}
    \begin{split}
    \begin{alignedat}{3}
        & &~~\frac{1}{|\chi|^2} \int_Y \A iX_\chi \vect u_2^{(n-2)} : \overline{iX_\chi \vect v}
        - z \int_Y \vect u_2^{(n-2)} \cdot \overline{\vect v} \\
        & &~+\frac{1}{|\chi|^2} \int_Y \A iX_\chi \vect u_1^{(n-1)} : \overline{iX_\chi \vect v}
        - z \int_Y \vect u_1^{(n-1)} \cdot \overline{\vect v} \\
        &+ \frac{1}{|\chi|^2} \int_Y \A \simgrad \vect u_2^{(n-1)} : \overline{iX_\chi \vect v}
        + \frac{1}{|\chi|^2} \int_Y \A iX_\chi \vect u_2^{(n-1)} : \overline{\simgrad \vect v}
        &~+ \frac{1}{|\chi|^2} \int_Y \A iX_\chi \vect u_2^{(n-1)} : \overline{iX_\chi \vect v}
        - z \int_Y \vect u_2^{(n-1)} \cdot \overline{\vect v} \\
        &+\sum_{j=0}^2 
        \frac{1}{|\chi|^2} \int_Y \A \simgrad \vect u_j^{(n)} : \overline{\simgrad \vect v} \\
        &+ \frac{1}{|\chi|^2}  \int_Y \A \simgrad \vect u_j^{(n)} : \overline{iX_\chi \vect v}
        + \frac{1}{|\chi|^2}  \int_Y \A iX_\chi \vect u_j^{(n)} : \overline{\simgrad \vect v}
        + \frac{1}{|\chi|^2}  \int_Y \A iX_\chi \vect u_j^{(n)} : \overline{iX_\chi \vect v}
        - z \int_Y \vect u_j^{(n)} \cdot \overline{\vect v} \span \\
        &+ \frac{1}{|\chi|^2} \int_Y \A (\simgrad + iX_\chi) \vect u_\text{error}^{(n)} \cdot \overline{(\simgrad + iX_\chi) \vect v}
        - z \int_Y \vect u_\text{error}^{(n)} \cdot \overline{\vect v}
        = 0 \span
    \end{alignedat}
    \end{split}
\end{equation}
To see that this is indeed what remains, recall the equation for $\vect u_\text{error}^{(n-1)}$, namely
\begin{align}
    \begin{split}
        &\frac{1}{|\chi|^2} \int_Y \A (\simgrad + iX_\chi) \vect u_\text{error}^{(n-1)} : 
        \overline{(\simgrad + iX_\chi) \vect v}
        - z \int_Y \vect u_\text{error}^{(n-1)} \cdot \overline{\vect v}
        = \frac{1}{|\chi|^2} \mathcal{R}_\text{error}^{(n-1)}(\vect v),
    \end{split}
\end{align}
which we have assumed to hold by the inductive hypothesis. More precisely, the first three lines of \eqref{eqn:resolvent_i_v2} are the leftover terms from cycles $\leq n-1$, which by definition is precisely $- \frac{1}{|\chi|^2} \mathcal{R}_\text{error}^{(n-1)}(\vect v)$.

There are no contribution from cycles $< n-2$.

\begin{remark}\label{rmk:error_i}
     The first six terms of $- \frac{1}{|\chi|^2} \mathcal{R}_\text{error}^{(n-1)}(\vect v)$ are of order $|\chi|^n$, and the remaining two terms are of order $|\chi|^{n+1}$. In particular, $\frac{1}{|\chi|^2} \mathcal{R}_\text{error}^{(n-1)}(\vect v) = O(|\chi|^n)$, and hence its terms do not play a role in the equations defining $\vect u_0^{(n)}$ and $\vect u_1^{(n)}$.
\end{remark}

We now proceed to collect $O(|\chi|^{n-2})$ terms to get an equation for $\vect u_0^{(n)}$:
\begin{align}\label{eqn:u0I}
    \frac{1}{|\chi|^2} \int_Y \A \simgrad \vect u_0^{(n)} : \overline{\simgrad \vect v} = 0.
\end{align}

Collecting the $O(|\chi|^{n-1})$ terms gives us an equation for $\vect u_1^{(n)}$:
\begin{align}\label{eqn:u1I}
    \frac{1}{|\chi|^2} \int_Y \A (\simgrad \vect u_1^{(n)} + iX_\chi \vect u_0^{(n)} ) : \overline{\simgrad \vect v} = 0.
\end{align}
(Note that the $O(|\chi|^{n-1})$ term $\int_Y \A \simgrad \vect u_0^{(n)} : \overline{iX_\chi \vect v}$ is zero, as $\vect u_0^{(n)}$ is a constant, by \eqref{eqn:u0I}.) This has a unique solution in the space $\dot{H}_{\#}(Y;\C^3)$.

We claim that by collecting the $O(|\chi|^{n})$ terms, we obtain the following equation for $\vect u_2^{(n)}$:
\begin{align}\label{eqn:u2I}
    \begin{split}
        &\frac{1}{|\chi|^2} \int_Y \A \simgrad \vect u_2^{(n)} : \overline{\simgrad \vect v} \\
        &\quad = - \frac{1}{|\chi|^2} \int_Y \A \simgrad (\vect u_1^{(n)} + \vect u_2^{(n-1)}) : \overline{iX_\chi \vect v}
        - \frac{1}{|\chi|^2} \int_Y \A iX_\chi (\vect u_1^{(n)} + \vect u_2^{(n-1)}) : \overline{\simgrad \vect v} \\
        &\qquad + \frac{1}{|\chi|^2} \int_Y \A iX_\chi (\vect u_0^{(n)} + \vect u_1^{(n-1)} + \vect u_2^{(n-2)}) : \overline{iX_\chi \vect v}
        + z \int_Y (\vect u_0^{(n)} + \vect u_1^{(n-1)} + \vect u_2^{(n-2)}) \cdot \overline{\vect v}.
    \end{split}
\end{align}
To see this, observe that the RHS of \eqref{eqn:u2I} consists of the six $O(|\chi|^n)$ terms from cycles $\leq n-1$ (in the first three lines of \eqref{eqn:resolvent_i_v2}, see also Remark \ref{rmk:error_i}), and four terms from cycle $n$ (in the fourth and fifth line of \eqref{eqn:resolvent_i_v2}).

\eqref{eqn:u2I} has a unique solution in $\dot{H}_{\#}$ if and only if its RHS vanishes for all $\vect v \in \C^3$. By the definition of $\mathcal{A}_\chi^{\hom}$, this criterion can be written as
\begin{align}\label{eqn:u0I_homogenized}
    \left\langle \left(\frac{1}{|\chi|^2} \mathcal{A}_\chi^{\hom} - zI_{\C^3} \right) \vect u_0^{(n)}, \vect v_0 \right\rangle_{\C^3} =
    -\frac{1}{|\chi|^2} \int_Y \A \left( \simgrad \vect u_2^{(n-1)} + iX_\chi (\vect u_1^{(n-2)} + \vect u_2^{(n-2)}) \right)  : \overline{iX_\chi \vect v_0}, \quad
    \forall \vect v_0 \in \C^3.
\end{align}
Since $z \in \rho(\tfrac{1}{|\chi|^2} \mathcal{A}_\chi^{\hom} )$, the constant vector $\vect u_0^{(n)}$ is uniquely chosen by \eqref{eqn:u0I_homogenized}. Consequently, $\vect u_1^{(n)}$ and $\vect u_2^{(n)}$ are uniquely determined by \eqref{eqn:u1I} and \eqref{eqn:u2I}.

Finally, we claim that collecting the leftover terms gives us the following equation for $\vect u_\text{error}^{(n)}$:
\begin{equation}\label{eqn:uerrI}
    \begin{split}
    \begin{alignedat}{3}
        &\frac{1}{|\chi|^2} \int_Y \A (\simgrad + iX_\chi) \vect u_\text{error}^{(n)} : \overline{(\simgrad + iX_\chi) \vect v }
        - z \int_Y \vect u_\text{error}^{(n)} \cdot \overline{\vect v}
        \span \\
        &= -\frac{1}{|\chi|^2} \int_Y \A iX_\chi \vect u_2^{(n-1)} : \overline{iX_\chi \vect v}
        + z \int_Y \vect u_2^{(n-1)} \cdot \overline{\vect v}
        & \\
        &- \frac{1}{|\chi|^2} \int_Y \A iX_\chi \vect u_1^{(n)} : \overline{iX_\chi \vect v}
        + z \int_Y \vect u_1^{(n)} \cdot \overline{\vect v} & \\
        &- \frac{1}{|\chi|^2} \int_Y \A iX_\chi \vect u_2^{(n)} : \overline{iX_\chi \vect v}
        + z \int_Y \vect u_2^{(n)} \cdot \overline{\vect v}
        \span~- \frac{1}{|\chi|^2} \int_Y \A \simgrad \vect u_2^{(n)} : \overline{iX_\chi \vect v}
        - \frac{1}{|\chi|^2} \int_Y \A iX_\chi \vect u_2^{(n)} : \overline{\simgrad \vect v} \\
        &=: \frac{1}{|\chi|^2} \mathcal{R}_\text{error}^{(n)}(\vect v).
    \end{alignedat}
    \end{split}
\end{equation}
To see this, observe that the RHS of \eqref{eqn:uerrI} contains the two $O(|\chi|^{n+1})$ terms from $\frac{1}{|\chi|^2} \mathcal{R}_\text{error}^{(n-1)}(\vect v)$ (second line of \eqref{eqn:uerrI}), and six terms from cycle $n$ (third and fourth line of \eqref{eqn:uerrI}). There are no contribution from cycles $\leq n-2$.

We have thus obtained the problems defining $\vect u_0^{(n)}$, $\vect u_1^{(n)}$, $\vect u_2^{(n)}$, and $\vect u_\text{error}^{(n)}$, namely \eqref{eqn:u0I}-\eqref{eqn:uerrI}. These are precisely the equations \eqref{inductivestep}, \eqref{eqn:inductivestep_remainder_1}, and \eqref{eqn:inductivestep_remainder_2}. The proof is complete.

\subsection{Estimates}\label{sect:fibrewise_estimates_working}

In the following proposition we provide the estimates in $H^1$ norm for the terms constructed using the Algorithm \ref{defn:algorithm}. For the definition of $D_\chi^{\rm hom}(z)$, $D_\chi(z)$, recall the Definition \ref{defn:dhomchi_dchi}.

\begin{proposition}\label{prop:fibrewise_terms_size}
     Fix $n \in \N_0$, $\chi \in Y'\setminus \{ 0 \}$, and $z \in  \rho(\frac{1}{|\chi|^2} \mathcal{A}_\chi^{\text{hom}} )$. There exist a constant $C>0$, independent of $n$, $\chi$, and $z$, such that the following estimates hold:
    \begin{align}
        &\lVert \vect u_0^{(n)}\rVert_{H^1(Y;\C^3)} \leq C^{n+1} \left[\sum_{k=0}^{n+1} \sum_{l=0}^n \frac{|z|^l}{\left(D_\chi^{\text{hom}}(z)\right)^{k}} \right]     |\chi|^{n}\lVert\vect f\rVert_{L^2(Y;\C^3)} \label{eqn:fibrewise_estimates1}\\
        &\lVert \vect u_1^{(n)}\rVert_{H^1(Y;\C^3)} \leq C^{n+1} \left[\sum_{k=0}^{n+1} \sum_{l=0}^n \frac{|z|^l}{\left(D_\chi^{\text{hom}}(z)\right)^{k}} \right]    |\chi|^{n+1}\lVert\vect f\rVert_{L^2(Y;\C^3)} \label{eqn:fibrewise_estimates2}\\
        &\lVert \vect u_2^{(n)}\rVert_{H^1(Y;\C^3)} \leq C^{n+1} \left[\sum_{k=0}^{n+1} \sum_{l=0}^{n+1} \frac{|z|^l}{\left(D_\chi^{\text{hom}}(z)\right)^{k}} \right]     |\chi|^{n+2}\lVert\vect f\rVert_{L^2(Y;\C^3)} \label{eqn:fibrewise_estimates3},
    \end{align}
    where $\vect u_0^{(n)}$, $\vect u_1^{(n)}$, $\vect u_2^{(n)}$ are provided in \eqref{eqn:inductivestep_remainder_1}, $\vect f \in L^2(Y;\C^3)$.
    Furthermore, the constant $C$ has the following explicit form:     
    \begin{equation}
    \begin{split}
           &C:= \max\{C_0, C_1, C_2\}, \quad  C_0 := \max\left\{ \frac{3}{\nu}, 1 \right\}, \quad  C_1 := \max\left\{\frac{3C_{\rm Korn}}{\nu^3}, \frac{C_{\rm Korn}}{\nu^2},1\right\},\\ & C_2 := \max\left\{\frac{4C_{\rm Korn}^2}{\nu^2} C_1 +  \frac{6C_{\rm Korn}^2}{\nu}\max\left\{1, \frac{1}{\nu}\right\}    C_0, 1\right\}.
    \end{split}
    \end{equation}
\end{proposition}
\begin{proof}
     The proof goes by mathematical induction. First, we recall from \cite[Sect.\,5.3]{simplified_method}, that the estimates \eqref{eqn:fibrewise_estimates1}-\eqref{eqn:fibrewise_estimates3} for cycle $n = 0, 1$ have been established. We shall now show that \eqref{eqn:fibrewise_estimates1}-\eqref{eqn:fibrewise_estimates3} holds for $n\geq 2$, assuming that \eqref{eqn:fibrewise_estimates1}-\eqref{eqn:fibrewise_estimates3} is true for cycles $0,\cdots, n-1$.
     
     
     The fourth equation in \eqref{inductivestep}, together with \eqref{assump:elliptic}, \eqref{eqn:Xchi_est}, yields:


\begin{equation}
\label{intermediateestimate0}
    \begin{split}
         \lVert \vect u_0^{(n)}\rVert_{H^1} & =  \lVert \vect u_0^{(n)}\rVert_{L^2} \leq  \frac{1}{D_\chi^{\text{hom}}(z)} \frac{1}{|\chi|^2} \left\lVert (iX_\chi)^* \A \left(\simgrad \vect u_2^{(n-1)} + iX_\chi \left(\vect u_1^{(n-1)} + \vect u_2^{(n-2)}  \right) \right)\right\rVert_{L^2} \\
         & \leq  \frac{1}{D_\chi^{\text{hom}}(z)} \frac{1}{|\chi|^2} \frac{1}{\nu}\left(  |\chi| \left\lVert \simgrad \vect u_2^{(n-1)}\right\rVert_{L^2} +  |\chi|^2 \left( \left\lVert \vect u_1^{(n-1)} \right\rVert_{L^2} + \left\lVert \vect u_2^{(n-2)}\right\rVert_{L^2} \right) \right) \\
         & \leq \frac{1}{D_\chi^{\text{hom}}(z)} \frac{1}{|\chi|} \frac{1 }{\nu}\left(  \left\lVert  \vect u_2^{(n-1)}\right\rVert_{H^1} + |\chi| \left( \left\lVert \vect u_1^{(n-1)} \right\rVert_{H^1} + \left\lVert \vect u_2^{(n-2)}\right\rVert_{H^1} \right) \right).
    \end{split}
\end{equation}
Furthermore, from the second equation in \eqref{inductivestep} and \eqref{assump:elliptic}, \eqref{eqn:Xchi_est}, we have:

\begin{equation}
    \begin{split}
        \nu \lVert \simgrad \vect u_1^{(n)}\rVert_{L^2}^2 & \leq \int_Y \A \simgrad \vect u_1^{(n)} : \overline{\simgrad \vect u_1^{(n)}} \leq \left|\int_Y \A iX_\chi \vect u_0^{(n)}: \overline{\simgrad \vect u_1^{(n)}} \right| \\
        & \leq \frac{1}{\nu} |\chi| \lVert \vect u_0^{(n)} \rVert_{L^2} \lVert \simgrad \vect u_1^{(n)}\rVert_{L^2} \leq \frac{1}{\nu} |\chi| \lVert \vect u_0^{(n)} \rVert_{H^1} \lVert \simgrad \vect u_1^{(n)}\rVert_{L^2}.
    \end{split}
\end{equation}
Thus, by Korn's inequality \eqref{korninequality33}:
\begin{equation}
\label{intermediateestimate1}
    \lVert  \vect u_1^{(n)}\rVert_{H^1} \leq \frac{C_{\rm Korn} }{\nu^2}|\chi|\lVert \vect u_0^{(n)} \rVert_{H^1}
\end{equation}
Finally, from the third equation in \eqref{inductivestep} and \eqref{assump:elliptic}, \eqref{eqn:Xchi_est} it follows: 
\begin{equation}
    \begin{split}
        &   \nu \lVert \simgrad \vect u_2^{(n)}\rVert_{L^2}^2  \leq \int_Y \A \simgrad \vect u_2^{(n)} : \overline{\simgrad \vect u_2^{(n)}}  \\
         & \leq \left|\int_Y \A iX_\chi \left(\vect u_1^{(n)} + \vect u_2^{(n-1)} \right): \overline{\simgrad \vect u_2^{(n)}} \right| + \left|\int_Y \A \simgrad \left(\vect u_1^{(n)} + \vect u_2^{(n-1)} \right): \overline{iX_\chi \vect u_2^{(n)}} \right| \\
         & + \left|\int_Y \A iX_\chi \left(\vect u_0^{(n)} + \vect u_1^{(n-1)} + \vect u_2^{(n-2)} \right): \overline{i X_\chi \vect u_2^{(n)}} \right| + |z||\chi|^2\left|\int_Y  \left(\vect u_0^{(n)} + \vect u_1^{(n-1)} + \vect u_2^{(n-2)} \right)\cdot \overline{ \vect u_2^{(n)}} \right| \\
         & \leq \frac{1}{\nu} |\chi| \left(\lVert \vect u_1^{(n)}\rVert_{L^2} + \lVert \vect u_2^{(n-1)}\rVert_{L^2} \right)\lVert \simgrad \vect u_2^{(n)}\rVert_{L^2} + \frac{1}{\nu} |\chi| \left(\lVert \simgrad \vect u_1^{(n)}\rVert_{L^2} + \lVert \simgrad \vect  u_2^{(n-1)}\rVert_{L^2} \right)\lVert \vect u_2^{(n)}\rVert_{L^2} \\
         & + \frac{1}{\nu} |\chi|^2 \left(\lVert \vect u_0^{(n)}\rVert_{L^2} + \lVert \vect u_1^{(n-1)}\rVert_{L^2} + \lVert \vect u_2^{(n-2)}\rVert_{L^2} \right)\lVert \vect u_2^{(n)}\rVert_{L^2} + |z| |\chi|^2 \left(\lVert \vect u_0^{(n)}\rVert_{L^2} + \lVert \vect u_1^{(n-1)}\rVert_{L^2} + \lVert \vect u_2^{(n-2)}\rVert_{L^2}  \right)\lVert \vect u_2^{(n)}\rVert_{L^2} \\
         & \leq \left(\frac{2}{\nu}|\chi| \left(\lVert \vect u_1^{(n)}\rVert_{H^1} + \lVert \vect u_2^{(n-1)}\rVert_{H^1} \right) + \left(\frac{1}{\nu} + |z| \right)|\chi|^2 \left( \lVert \vect u_0^{(n)}\rVert_{H^1} + \lVert \vect u_1^{(n-1)}\rVert_{H^1} + \lVert \vect u_2^{(n-2)}\rVert_{H^1}\right) \right)\lVert \vect u_2^{(n)}\rVert_{H^1} \\
         & \leq \left(\frac{2}{\nu}|\chi| \left(\lVert \vect u_1^{(n)}\rVert_{H^1} + \lVert \vect u_2^{(n-1)}\rVert_{H^1} \right) + \left(\frac{1}{\nu} + |z| \right)|\chi|^2 \left( \lVert \vect u_0^{(n)}\rVert_{H^1} + \lVert \vect u_1^{(n-1)}\rVert_{H^1} + \lVert \vect u_2^{(n-2)}\rVert_{H^1}\right) \right)C_{\rm Korn}\lVert  \simgrad \vect u_2^{(n)}\rVert_{L^2}.
    \end{split}
\end{equation}
Therefore, we have: 
\begin{equation}
\label{intermediateestimate2}
    \begin{split}
        & \lVert  \vect u_2^{(n)}\rVert_{H^1} \leq C_{\rm Korn} \lVert  \simgrad \vect u_2^{(n)}\rVert_{L^2} \\
        & \leq \frac{2C_{\rm Korn}^2}{\nu^2}|\chi| \left(\lVert \vect u_1^{(n)}\rVert_{H^1} + \lVert \vect u_2^{(n-1)}\rVert_{H^1} \right) + \frac{C_{\rm Korn}^2}{\nu}\left(\frac{1}{\nu} + |z| \right)|\chi|^2 \left( \lVert \vect u_0^{(n)}\rVert_{H^1} + \lVert \vect u_1^{(n-1)}\rVert_{H^1} + \lVert \vect u_2^{(n-2)}\rVert_{H^1}\right)  \\
        & \leq \frac{2C_{\rm Korn}^2}{\nu^2}|\chi| \left(\lVert \vect u_1^{(n)}\rVert_{H^1} + \lVert \vect u_2^{(n-1)}\rVert_{H^1} \right) + \frac{C_{\rm Korn}^2}{\nu}C_\nu(1+|z|)  |\chi|^2 \left( \lVert \vect u_0^{(n)}\rVert_{H^1} + \lVert \vect u_1^{(n-1)}\rVert_{H^1} + \lVert \vect u_2^{(n-2)}\rVert_{H^1}\right), 
    \end{split}
\end{equation}
where we use notation: 
\begin{equation}
    C_\nu:= \max\left\{1, \frac{1}{\nu}\right\}
\end{equation}
Next we use the induction hypothesis: 
    \begin{equation}
        \lVert \vect u_2^{(n-1)}\rVert_{H^1} \leq C^{n} \left[\sum_{k=0}^{n} \sum_{l=0}^n \frac{|z|^l}{\left(D_\chi^{\text{hom}}(z)\right)^{k}} \right]     |\chi|^{n+1}\lVert\vect f\rVert_{L^2}
    \end{equation}
    \begin{equation}
        \lVert \vect u_1^{(n-1)}\rVert_{H^1} \leq C^{n} \left[\sum_{k=0}^{n} \sum_{l=0}^{n-1} \frac{|z|^l}{\left(D_\chi^{\text{hom}}(z)\right)^{k}} \right]    |\chi|^{n}\lVert\vect f\rVert_{L^2}
    \end{equation}
    \begin{equation}
        \lVert \vect u_2^{(n-2)}\rVert_{H^1} \leq C^{n-1} \left[\sum_{k=0}^{n-1} \sum_{l=0}^{n-1} \frac{|z|^l}{\left(D_\chi^{\text{hom}}(z)\right)^{k}} \right]     |\chi|^{n}\lVert\vect f\rVert_{L^2}
    \end{equation}
Combining this with \eqref{intermediateestimate0}, we are able to obtain the following estimate: 
\begin{equation}
    \begin{split}
         & \lVert \vect u_0^{(n)}\rVert_{H^1} \\
        &  \leq \frac{1}{D_\chi^{\text{hom}}(z)} \frac{1 }{\nu} C^{n-1} \left(    C\sum_{k=0}^{n} \sum_{l=0}^{n} \frac{|z|^l}{\left(D_\chi^{\text{hom}}(z)\right)^{k}}  +  C\sum_{k=0}^{n} \sum_{l=0}^{n-1} \frac{|z|^l}{\left(D_\chi^{\text{hom}}(z)\right)^{k}} + \sum_{k=0}^{n-1} \sum_{l=0}^{n-1} \frac{|z|^l}{\left(D_\chi^{\text{hom}}(z)\right)^{k}} \right)|\chi|^{n}\lVert\vect f\rVert_{L^2} \\
         &\leq  \frac{3 }{\nu} C^{n}  \left(   \sum_{k=0}^{n+1} \sum_{l=0}^{n} \frac{|z|^l}{\left(D_\chi^{\text{hom}}(z)\right)^{k}} \right)|\chi|^{n}\lVert\vect f\rVert_{L^2},
    \end{split}
\end{equation}
where we have used the fact that $C>1$.
We introduce the following auxiliary constant: 
\begin{equation}
    C_0 := \max\left\{\frac{3}{\nu}, 1 \right\} 
\end{equation}
Now we have: 
\begin{equation}
\label{almostestimate0}
    \lVert \vect u_0^{(n)}\rVert_{H^1} \leq C_0 C^{n}  \sum_{k=0}^{n+1} \sum_{l=0}^{n} \frac{|z|^l}{\left(D_\chi^{\text{hom}}(z)\right)^{k}}|\chi|^{n}\lVert\vect f\rVert_{L^2}
\end{equation}
Next, combining this with \eqref{intermediateestimate1} yields:
\begin{equation}
    \lVert  \vect u_1^{(n)}\rVert_{H^1} \leq \frac{C_{\rm Korn} }{\nu^2} C_0 C^{n} \sum_{k=0}^{n+1} \sum_{l=0}^{n} \frac{|z|^l}{\left(D_\chi^{\text{hom}}(z)\right)^{k}}|\chi|^{n+1}\lVert\vect f\rVert_{L^2}
\end{equation}
By introducing the auxiliary constant 
\begin{equation}
    C_1 := \max\left\{\frac{C_{\rm Korn} }{\nu^2} C_0,1\right\},
\end{equation}
we have: 
\begin{equation}
\label{almostestimate1}
    \lVert  \vect u_1^{(n)}\rVert_{H^1} \leq C_1 C^{n} \sum_{k=0}^{n+1} \sum_{l=0}^{n} \frac{|z|^l}{\left(D_\chi^{\text{hom}}(z)\right)^{k}}|\chi|^{n+1}\lVert\vect f\rVert_{L^2}
\end{equation}
Lastly, from \eqref{intermediateestimate2}, \eqref{almostestimate0}, \eqref{almostestimate1}, induction hypothesis and the fact that $C\geq 1$, we have
\begin{equation}
    \begin{split}
        & \lVert  \vect u_2^{(n)}\rVert_{H^1}  \\
        & \leq \frac{2C_{\rm Korn}^2}{\nu^2}|\chi| \left(\lVert \vect u_1^{(n)}\rVert_{H^1} + \lVert \vect u_2^{(n-1)}\rVert_{H^1} \right) + \frac{C_{\rm Korn}^2}{\nu}C_\nu(1+|z|)  |\chi|^2 \left( \lVert \vect u_0^{(n)}\rVert_{H^1} + \lVert \vect u_1^{(n-1)}\rVert_{H^1} + \lVert \vect u_2^{(n-2)}\rVert_{H^1}\right) \\
        &\leq 
        \frac{4C_{\rm Korn}^2}{\nu^2} C^{n}C_1 \sum_{k=0}^{n+1} \sum_{l=0}^{n} \frac{|z|^l}{\left(D_\chi^{\text{hom}}(z)\right)^{k}} |\chi|^{n+2}\lVert\vect f\rVert_{L^2}   + \frac{3C_{\rm Korn}^2}{\nu}C_\nu  C^{n}  C_0(1+|z|) \sum_{k=0}^{n+1} \sum_{l=0}^{n} \frac{|z|^l}{\left(D_\chi^{\text{hom}}(z)\right)^{k}}|\chi|^{n+2}\lVert\vect f\rVert_{L^2} \\
        & \leq \frac{4C_{\rm Korn}^2}{\nu^2} C^{n}C_1 \sum_{k=0}^{n+1} \sum_{l=0}^{n} \frac{|z|^l}{\left(D_\chi^{\text{hom}}(z)\right)^{k}} |\chi|^{n+2}\lVert\vect f\rVert_{L^2} + \frac{6C_{\rm Korn}^2}{\nu}C_\nu  C^{n}  C_0\sum_{k=0}^{n+1} \sum_{l=0}^{n+1} \frac{|z|^l}{\left(D_\chi^{\text{hom}}(z)\right)^{k}}|\chi|^{n+2}\lVert\vect f\rVert_{L^2}. 
    \end{split}
\end{equation}
Therefore: 
\begin{equation}\label{almostestimate2}
    \lVert  \vect u_2^{(n)}\rVert_{H^1} \leq C_2 C^n \sum_{k=0}^{n+1} \sum_{l=0}^{n+1} \frac{|z|^l}{\left(D_\chi^{\text{hom}}(z)\right)^{k}}|\chi|^{n+2}\lVert\vect f\rVert_{L^2}.
\end{equation}
where 
\begin{equation}
    C_2 := \max\left\{\frac{4C_{\rm Korn}^2}{\nu^2} C_1 +  \frac{6C_{\rm Korn}^2}{\nu}C_\nu    C_0, 1\right\}.
\end{equation}
Finally, by setting:
\begin{equation}
    C:= \max\{C_0, C_1, C_2\},
\end{equation}
we obtain: 
\begin{equation}
\label{estimate0}
    \lVert \vect u_0^{(n)}\rVert_{H^1} \leq C^{n+1}  \sum_{k=0}^{n+1} \sum_{l=0}^{n} \frac{|z|^l}{\left(D_\chi^{\text{hom}}(z)\right)^{k}}|\chi|^{n}\lVert\vect f\rVert_{L^2}
\end{equation}
\begin{equation}
\label{estimate1}
    \lVert  \vect u_1^{(n)}\rVert_{H^1} \leq C^{n+1} \sum_{k=0}^{n+1} \sum_{l=0}^{n} \frac{|z|^l}{\left(D_\chi^{\text{hom}}(z)\right)^{k}}|\chi|^{n+1}\lVert\vect f\rVert_{L^2}
\end{equation}
\begin{equation}\label{estimate2}
    \lVert  \vect u_2^{(n)}\rVert_{H^1} \leq C^{n+1} \sum_{k=0}^{n+1} \sum_{l=0}^{n+1} \frac{|z|^l}{\left(D_\chi^{\text{hom}}(z)\right)^{k}}|\chi|^{n+2}\lVert\vect f\rVert_{L^2}.
\end{equation}
\end{proof}

\begin{proposition}\label{estimateuerrorprop}
     Fix $n \in \N_0$, $\chi \in Y'\setminus \{ 0 \}$, and $z \in \rho(\frac{1}{|\chi|^2} \mathcal{A}_\chi ) \cap \rho(\frac{1}{|\chi|^2} \mathcal{A}_\chi^{\text{hom}} )$. There exist constants $C_{\rm error}>0$, $C>0$, independent of $n$, $\chi$, and $z$, such that the following estimates hold:
    \begin{equation}
        \lVert  \vect u_{\text{error}}^{(n)} \rVert_{H^1(Y;\C^3)} \leq C_{\rm error}C^{n+1} \max\left\{1,\frac{|z+1|}{D_\chi(z)} \right\} \sum_{k=0}^{n+1} \sum_{l=0}^{n+2} \frac{|z|^l}{\left(D_\chi^{\text{hom}}(z)\right)^{k}} |\chi|^{n+1}\lVert\vect f\rVert_{L^2(Y;\C^3)},
    \end{equation}
    where $\vect u_{\rm error}$ is the solution to the equation \eqref{eqn:inductivestep_remainder_1}, $\vect f \in L^2(Y;\C^3)$.
    Furthermore, the constants have the following explicit form: 
    %
    %
    %
    \begin{equation}
    \begin{split}
           &C:= \max\{C_0, C_1, C_2\}, \quad  C_0 := \max\left\{ \frac{3}{\nu}, 1 \right\}, \quad  C_1 := \max\left\{\frac{3C_{\rm Korn}}{\nu^3}, \frac{C_{\rm Korn}}{\nu^2},1\right\},\\ & C_2 := \max\left\{\frac{4C_{\rm Korn}^2}{\nu^2} C_1 +  \frac{6C_{\rm Korn}^2}{\nu}\max\left\{1, \frac{1}{\nu}\right\}    C_0, 1\right\},\quad  C_{\rm error}:= 4\max\left\{6 + \frac{2}{\nu}, \frac{8}{\nu}, 1\right\}.
    \end{split}
    \end{equation}
\end{proposition}

\begin{proof}
In order to estimate the residual $\mathcal{R}_{\rm error}^{(n)}$, we recall that for $\vect v \in H^1(Y;\C^3)$:
\begin{equation}
    \begin{split}
       & \frac{1}{|\chi|^2} \mathcal{R}_\text{error}^{(n)}(\vect v) =  -\frac{1}{|\chi|^2} \int_Y \A iX_\chi \vect u_2^{(n-1)} : \overline{iX_\chi \vect v} - \frac{1}{|\chi|^2} \int_Y \A iX_\chi \vect u_1^{(n)} : \overline{iX_\chi \vect v} - \frac{1}{|\chi|^2} \int_Y \A iX_\chi \vect u_2^{(n)} : \overline{iX_\chi \vect v}
        \\
        &- \frac{1}{|\chi|^2} \int_Y \A \simgrad \vect u_2^{(n)} : \overline{iX_\chi \vect v} - \frac{1}{|\chi|^2} \int_Y \A iX_\chi \vect u_2^{(n)} : \overline{ \simgrad\vect v}
        + z \int_Y \vect u_2^{(n-1)} \cdot \overline{\vect v}
        + z \int_Y \vect u_1^{(n)} \cdot \overline{\vect v} + z \int_Y \vect u_2^{(n)} \cdot \overline{\vect v}.
    \end{split}
\end{equation}
Now: 
\begin{equation}
    \begin{split}
        \left|\mathcal{R}_{\rm error}^{(n)}(\vect v) \right| &\leq \left(\left(\frac{1}{\nu} + |z|\right)|\chi|^2 \left(\left\lVert \vect u_2^{(n-1)}\right\rVert_{H^1} + \left\lVert \vect u_1^{(n)}\right\rVert_{H^1} + \left\lVert \vect u_2^{(n)}\right\rVert_{H^1} \right) + \frac{2}{\nu}|\chi|\left\lVert \vect u_2^{(n)}\right\rVert_{H^1} \right) \lVert \vect v \rVert_{H^1}\\
        & \leq \left( C_\nu (1 + |z|)|\chi|^2\left(\left\lVert \vect u_2^{(n-1)}\right\rVert_{H^1} + \left\lVert \vect u_1^{(n)}\right\rVert_{H^1} + \left\lVert \vect u_2^{(n)}\right\rVert_{H^1} \right) + \frac{2}{\nu}|\chi|\left\lVert \vect u_2^{(n)}\right\rVert_{H^1} \right) \lVert \vect v \rVert_{H^1}.
    \end{split}
\end{equation}
Therefore 
\begin{equation}
    \lVert \mathcal{R}_{\rm error}^{(n)} \rVert_{(H^1)^*} \leq C_\nu (1 + |z|)|\chi|^2\left(\left\lVert \vect u_2^{(n-1)}\right\rVert_{H^1} + \left\lVert \vect u_1^{(n)}\right\rVert_{H^1} + \left\lVert \vect u_2^{(n)}\right\rVert_{H^1} \right) + \frac{2}{\nu}|\chi|\left\lVert \vect u_2^{(n)}\right\rVert_{H^1}.
\end{equation}
Now we use the estimates \eqref{eqn:fibrewise_estimates1}, \eqref{eqn:fibrewise_estimates2}, \eqref{eqn:fibrewise_estimates3}, and the fact that $|\chi| < 1$ to see that: 
\begin{equation}
    \begin{split}
        \lVert \mathcal{R}_{\rm error}^{(n)} \rVert_{(H^1)^*} & \leq C^{n+1}\left( 3 C_\nu  (1 + |z|)\sum_{k=0}^{n+1} \sum_{l=0}^{n+1} \frac{|z|^l}{\left(D_\chi^{\text{hom}}(z)\right)^{k}} + \frac{2}{\nu}\sum_{k=0}^{n+1} \sum_{l=0}^{n+1} \frac{|z|^l}{\left(D_\chi^{\text{hom}}(z)\right)^{k}} \right)|\chi|^{n+3}\lVert\vect f\rVert_{L^2} \\
        &\leq C_3 C^{n+1} \sum_{k=0}^{n+1} \sum_{l=0}^{n+2} \frac{|z|^l}{\left(D_\chi^{\text{hom}}(z)\right)^{k}} |\chi|^{n+3}\lVert\vect f\rVert_{L^2},
    \end{split}
\end{equation}
where 
\begin{equation}
    C_3 := \max\left\{6C_\nu + \frac{2}{\nu} , 1\right\}
 \end{equation}
 Finally, by applying \eqref{eqn:absract_ineq} to the solution $u_\text{error}^{(n)}$ to the problem
 \begin{align}
    \begin{split}
        &\frac{1}{|\chi|^2} \int_Y \A (\simgrad + iX_\chi) \vect u_\text{error}^{(n)} : 
        \overline{(\simgrad + iX_\chi) \vect v}
        - z \int_Y \vect u_\text{error}^{(n)} \cdot \overline{\vect v}
        = \frac{1}{|\chi|^2} \mathcal{R}_\text{error}^{(n)}(\vect v),
    \end{split}
\end{align}
 we obtain: 
\begin{equation}
        \lVert  \vect u_{\text{error}}^{(n)} \rVert_{H^1(Y)} \leq C_{\rm error}C^{n+1} \max\left\{1,\frac{|z+1|}{D_\chi(z)} \right\} \sum_{k=0}^{n+1} \sum_{l=0}^{n+2} \frac{|z|^l}{\left(D_\chi^{\text{hom}}(z)\right)^{k}} |\chi|^{n+1}\lVert\vect f\rVert_{L^2},
    \end{equation}
where we set
\begin{equation*}
    C_{\rm error} := 4C_3 . \qedhere
\end{equation*}
\end{proof}

\subsection{Fiberwise results}\label{sect:fiberwise_results}

Rephrasing Proposition \ref{estimateuerrorprop} using \eqref{eqn:errortermdefinition}, we have the following:

\begin{theorem}\label{thm:fibrewise_estimates_functionform}
    Fix $n \in \N_0$, $\chi \in Y'\setminus \{ 0 \}$, and $z \in \rho(\frac{1}{|\chi|^2} \mathcal{A}_\chi ) \cap \rho(\frac{1}{|\chi|^2} \mathcal{A}_\chi^{\text{hom}} )$. There exist constants $C_{\rm error}>0$, $C>0$, independent of $n$, $\chi$, and $z$, such that the following estimates hold:
    \begin{equation}
    \begin{split}
         &\left\lVert  \vect u - \sum_{k = 0}^n\left(\vect u_0^{(k)} + \vect u_1^{(k)}+ \vect u_2^{(k)} \right) \right\rVert_{H^1}
         \leq C_{\rm error}C^{n+1} \max\left\{1,\frac{|z+1|}{D_\chi(z)} \right\} \sum_{k=0}^{n+1} \sum_{l=0}^{n+2} \frac{|z|^l}{\left(D_\chi^{\text{hom}}(z)\right)^{k}} |\chi|^{n+1}\lVert\vect f\rVert_{L^2}.
    \end{split}
    \end{equation}
    Here $\vect u \in H^1(Y;\C^3)$ is the solution to the resolvent problem \eqref{eqn:fiberwiseresolventproblem} for the operator $\mathcal{A}_\chi$, $\vect u_j^{(k)} \in H^1(Y;\C^3)$ are given by the inductive procedure \eqref{eqn:cycle_i0}, \eqref{eqn:cycle_i1}, \eqref{inductivestep}, and $\vect f \in L^2(Y;\C^3)$. The constants have the explicit form    
    \begin{equation}\label{eqn:constants_explicit}
    \begin{split}
           &C:= \max\{C_0, C_1, C_2\}, \quad  C_0 := \max\left\{ \frac{3}{\nu}, 1 \right\}, \quad  C_1 := \max\left\{\frac{3C_{\rm Korn}}{\nu^3}, \frac{C_{\rm Korn}}{\nu^2},1\right\},\\ & C_2 := \max\left\{\frac{4C_{\rm Korn}^2}{\nu^2} C_1 +  \frac{6C_{\rm Korn}^2}{\nu}\max\left\{1, \frac{1}{\nu}\right\}    C_0, 1\right\},\quad   C_{\rm error}:= 4\max\left\{6 + \frac{2}{\nu}, \frac{8}{\nu}, 1\right\}.
    \end{split}
    \end{equation}
\end{theorem}

\noindent
\renewcommand{\arraystretch}{1.3}
\begin{table}[t]
    \centering
    \begin{tabular}{|l||*{10}{c|}}\hline
        \backslashbox{\textbf{Cycle}\kern-4em}{\textbf{Norm}\kern-0.3em}
        &\makebox[1.3em]{$1$}
        &\makebox[1.3em]{$|\chi|$}
        &\makebox[1.3em]{$|\chi|^2$}
        &\makebox[1.3em]{$|\chi|^3$}
        &\makebox[1.3em]{$|\chi|^4$}
        &$\dots$
        &\makebox[1.3em]{$|\chi|^{n-1}$}
        &\makebox[1.3em]{$|\chi|^{n}$}
        &\makebox[1.3em]{$|\chi|^{n+1}$}
        &\makebox[1.3em]{$|\chi|^{n+2}$} \\ \hline\hline
        $0$ & $\vect u_0^{(0)}$ &$\vect u_1^{(0)}$&$\vect u_2^{(0)}$&&&&&&&\\\hline
        $1$ &&$\vect u_0^{(1)}$&$\vect u_1^{(1)}$&$\vect u_2^{(1)}$&&&&&&\\\hline
        $2$ &&&$\vect u_0^{(2)}$&$\vect u_1^{(2)}$&$\vect u_2^{(2)}$&&&&&\\\hline
        $3$ &&&&$\ddots$&$\ddots$&$\ddots$&&&&\\\hline
        $\vdots$ &&&&&$\ddots$&$\ddots$&$\ddots$&&&\\\hline
        $n-2$ &&&&&&$\vect u_0^{(n-2)}$&$\vect u_1^{(n-2)}$&$\vect u_2^{(n-2)}$&&\\\hline
        $n-1$ &&&&&&&$\vect u_0^{(n-1)}$&$\vect u_1^{(n-1)}$&$\vect u_2^{(n-1)}$&\\\hline
        $n$ &&&&&&&&$\vect u_0^{(n)}$&$\vect u_1^{(n)}$ &$\vect u_2^{(n)}$\\\hline
        \hline
        \textbf{Error} &$|\chi|$&$|\chi|^2$&$|\chi|^3$&$|\chi|^4$&$|\chi|^5$&$\dots$&$|\chi|^n$&$|\chi|^{n+1}$&$|\chi|^{n+2}$&$|\chi|^{n+3}$ \\\hline
    \end{tabular}
    \caption{The terms $\vect u_j^{(i)}$ arranged according to cycle and size of its $H^1$ norm. The final row denotes the size of the error (in $H^1$) after taking all $\vect u_j^{(i)}$'s up to that column. For example, $|\chi|^2$ refers to the size of $\| \vect u - (\vect u_0^{(0)} + \vect u_1^{(0)} + \vect u_0^{(1)}) \|_{H^1}$.}
    \label{tab:fiberwise_summary}
\end{table}

We refer the reader to Table \ref{tab:fiberwise_summary} for a summary of Section \ref{sect:fibrewise_estimates_working}, where terms $\vect u_j^{(k)}$ are arranged according to cycle and the size of its $H^1$ norm (Proposition \ref{prop:fibrewise_terms_size}). The ``Error" row refers to the $H^1$ norm of the difference between $\vect u$ and the sum of $\vect u_j^{(k)}$'s collected column-wise, e.g. $\| \vect u - (\vect u_0^{(0)} + \vect u_1^{(0)} + \vect u_0^{(1)}) \|_{H^1} = \mathcal{O}(|\chi|^2)$, as $|\chi| \downarrow 0$. This is a consequence of $\| \vect u - (\vect u_0^{(0)} + \vect u_1^{(0)} + \vect u_2^{(0)}) - (\vect u_0^{(1)} + \vect u_1^{(1)} + \vect u_2^{(1)}) \|_{H^1} = \| \vect u_\text{error}^{(1)} \|_{H^1} = \mathcal{O}(|\chi|^{2})$ (Theorem \ref{thm:fibrewise_estimates_functionform}) and $\vect u_2^{(0)}, \vect u_1^{(1)}, \vect u_2^{(1)} = \mathcal{O}(|\chi|^2)$ (Proposition \ref{prop:fibrewise_terms_size}).

Since the functions $\vect u_j^{(k)}$ depend linearly on $\vect f \in L^2(Y;\C^3)$, it is natural to state the results of the asymptotic procedure in terms of (linear) operators. To this end, we define

\begin{definition}[Corrector operators]\label{defn:corrector_operators_chi}
    For each $\chi \in Y'\setminus \{ 0 \}$, $z \in \rho (\frac{1}{|\chi|^2}\mathcal{A}_\chi ) \cap \rho (\frac{1}{|\chi|^2}\mathcal{A}_\chi^{\rm hom} )$, $k\in \N_0$, and $j\in\{0,1,2\}$, define the mapping
    \begin{equation}\label{eqn:corrector_operators}
        \mathcal{R}_{j,\chi}^{(k)} (z) : L^2(Y;\C^3) \to L^2(Y;\C^3), \quad 
        \mathcal{R}_{j,\chi}^{(k)} (z) \vect f := \vect u_j^{(k)}, 
    \end{equation}
    where $\vect u_j^{(k)}$ are given by \eqref{eqn:cycle_i0}, \eqref{eqn:cycle_i1}, and \eqref{inductivestep}.
\end{definition}

\begin{remark}
    $\mathcal{R}_{j,\chi}^{(k)} (z)$ are bounded linear maps, with operator norm of order $\mathcal{O}(|\chi|^{j+k})$.
\end{remark}

\begin{remark}\label{rmk:corrector_ops_vs_classical_formula_fibrespace}
    The first two terms in the zeroth cycle $\vect u_0^{(0)}$ and $\vect u_1^{(0)}$ can be expressed in terms of $\mathcal{A}_\chi^{\hom}$ and the classical first-order corrector $\mathbb{N}(y)$. More precisely, we have 
    \begin{align}
        &\mathcal{R}_{0,\chi}^{(0)} (z) = \left(\frac{1}{|\chi|^2} \mathcal{A}_\chi^{\rm hom} -z I_{\C^3} \right)^{-1} P_0, \\
        &\mathcal{R}_{1,\chi}^{(0)} (z) = \N(y) : (iX_\chi)\left(\frac{1}{|\chi|^2} \mathcal{A}_\chi^{\rm hom} -z I_{\C^3} \right)^{-1} P_0,
    \end{align}
    See \cite[(6.82)]{simplified_method} for a proof of this fact.
\end{remark}

Rephrasing our result in terms of the corrector operators, we have:
\begin{theorem}\label{thm:fibrewise_estimates}
    Fix $n \in \N_0$, $\chi \in Y'\setminus \{ 0 \}$, $z \in \rho(\frac{1}{|\chi|^2} \mathcal{A}_\chi ) \cap \rho(\frac{1}{|\chi|^2} \mathcal{A}_\chi^{\text{hom}} )$, and $\vect f \in L^2(Y;\C^3)$. There exist constants $C_{\rm error}>0$, $C>0$, independent of $n$, $\chi$, and $z$, such that the following estimates hold:
    \begin{align}\label{eqn:fibrewise_estimates_josip}
    \begin{split}
        &\left\| \left( \frac{1}{|
        \chi|^2} \mathcal{A}_\chi - zI \right)^{-1} \vect f - \sum_{k=0}^n \left( \mathcal{R}_{0,\chi}^{(k)}(z) + \mathcal{R}_{1,\chi}^{(k)}(z) + \mathcal{R}_{2,\chi}^{(k)}(z) \right) \vect f \right\|_{H^1(Y;\C^3)} \\
        &\qquad\qquad\qquad\qquad \leq C_{\rm error}C^{n+1} \max\left\{1,\frac{|z+1|}{D_\chi(z)} \right\} \sum_{k=0}^{n+1} \sum_{l=0}^{n+2} \frac{|z|^l}{\left(D_\chi^{\text{hom}}(z)\right)^{k}} |\chi|^{n+1}\lVert\vect f\rVert_{L^2(Y;\C^3)}.
    \end{split}
    \end{align}
    The constants have the explicit form \eqref{eqn:constants_explicit}.
\end{theorem}

\begin{remark}\label{rmk:z_on_contour_nice_part2}
    Continuing from Remark \ref{rmk:z_on_contour_nice}, if $z \in \Gamma$, then we may further bound
    \begin{align}
        \max\left\{1,\frac{|z+1|}{D_\chi(z)} \right\} \sum_{k=0}^{n+1} \sum_{l=0}^{n+2} \frac{|z|^l}{\left(D_\chi^{\text{hom}}(z)\right)^{k}}
        \leq \widetilde{C}^{n+3},
    \end{align}
    where the constant $\widetilde{{C}}>0$ depends only on $\nu$ and $C_\text{Fourier}$. In the proofs of the main theorems \ref{thm:l2tol2} and \ref{thm:l2toh1} later, it will be important to note that the overall constant in \eqref{eqn:fibrewise_estimates_josip} is of the form $C^n$.
\end{remark}

\section{Full space estimate}\label{sect:fullspace_estimates}

We shall now show how the fibre-wise results of Section \ref{sect:fibrewise_asmptotics} implies our main result (Theorem \ref{thm:main_result_intro}). We shall first introduce various auxiliary objects that will help us to bring the discussion back to the full space $L^2(\R^3;\C^3)$ (Section \ref{sect:rescaled_and_fullspace_ops}). Then, we prove the error estimates in the $L^2$ sense (Section \ref{sect:proof_l2tol2}), and in the $H^1$ sense (Section \ref{sect:proof_l2toh1}), thus concluding the paper.

\subsection{Rescaled and full-space correctors}\label{sect:rescaled_and_fullspace_ops}

For this section, assume that $z \in \rho( \tfrac{1}{|\chi|^2} \mathcal{A}_\chi^{\hom} ) \cap \rho( \tfrac{1}{|\chi|^2} \mathcal{A}_\chi )$ and $\chi \in Y' \setminus \{ 0 \}$.

\paragraph*{The rescaling function.}

First, let us recall from \cite[Sect. 6]{simplified_method}, the auxiliary function $g_{\eps,\chi}$.

\begin{definition}\label{defn:function_g}
    For each $\eps>0$, $\chi \neq 0$, define the function $g_{\eps,\chi}: \{z \in \C : \Re(z) > 0 \} \rightarrow \C$ 
    \begin{equation}\label{eqn:function_g}
       g_{\eps,\chi}(z) := \left(\frac{|\chi|^2}{\eps^{ 2}}z + 1 \right)^{-1}. 
    \end{equation} 
\end{definition}

The key property of $g_{\eps,\chi}$ that we need in the sequel is the following estimate:

\begin{lemma}\label{lem:function_g_bound}
    For every fixed $\eta > 0$, function  $g_{\varepsilon,\chi}$ is bounded on the half plane $\{z \in \C, \Re(z) \geq \eta\}$, with
    \begin{equation}
        |g_{\varepsilon,\chi}(z)| \leq C_\eta \left(\max\left\{\frac{|\chi|^2}{\varepsilon^{ 2}}, 1\right\}\right)^{-1}.
    \end{equation}
\end{lemma}
\begin{proof}
    \cite[Lemma 6.5]{simplified_method}
\end{proof}

\paragraph*{The rescaled corrector operators.}

Second, we shall connect the $\tfrac{1}{|\chi|^2}$ rescaling in Section \ref{sect:fibrewise_asmptotics} back to $\tfrac{1}{\eps^2}$ by introducing the rescaled version of the the corrector operators. Recall the corrector operator $\mathcal{R}_{j,\chi}^{(k)}(z)$ (Definition \ref{defn:corrector_operators_chi}), the contour $\Gamma$ (and constant $\mu>0$, Lemma \ref{lem:contour_existence}), and the function $g_{\eps,\chi}$ \eqref{defn:function_g}.

\begin{definition}[Rescaled corrector operators]\label{defn:corrector_operators_chi_rescaled}
    For $i \in \N_0$, $j \in \{0,1,2\}$, and $\eps > 0$, the rescaled corrector operator $\mathcal{R}_{j,\chi,\eps}^{(k)} : L^2(Y;\C^3) \to L^2(Y;\C^3)$ is given by
    \begin{equation}\label{eqn:corrector_operators_chi_rescaled}
        \mathcal{R}_{j,\chi, \eps}^{(k)} := 
        \begin{cases}
            - \frac{1}{2\pi i} \oint_\Gamma g_{\eps,\chi}(z) \mathcal{R}_{j,\chi}^{(k)} (z) \, dz \qquad 
            &\text{if $\chi \in [-\mu,\mu]^3 \setminus \{ 0 \}$}, \\
            0 \qquad &\text{otherwise.}
        \end{cases}
    \end{equation}
\end{definition}

\paragraph*{The full-space corrector operators.}

Finally, we shall connect the fiber space $L^2(Y\times Y' ; \C^3)$ back to the full space $L^2(\R^3;\C^3)$ by introducing the full-space version of the corrector operators. Recall the rescaled corrector operators $\mathcal{R}_{j, \chi, \eps}^{(i)}$ (Definition \ref{defn:corrector_operators_chi_rescaled}) and the rescaled Gelfand transform \eqref{eqn:gelfand_inversion}.

\begin{definition}[Full-space corrector operator]\label{defn:corrector_operators_fullspace}
    For $i\in\{0,1,2\}$, $j\in \N_0$, $\eps>0$, the full-space corrector operators $\mathcal{R}_{j,\eps}^{(k)} : \mathcal{D}(\mathcal{R}_{j,\eps}^{(k)}) \to L^2(\R^3;\C^3)$ is given by
    \begin{equation}\label{eqn:corrector_operators_fullspace}
        \mathcal{R}_{j, \eps}^{(k)} = \mathcal{G}_\eps^\ast \left( \int_{Y'}^{\oplus} \mathcal{R}_{j, \chi, \eps}^{(k)} \, d\chi \right) \mathcal{G}_\eps, \qquad
        \mathcal{D}(\mathcal{R}_{j,\eps}^{(k)}) = \left\{ \vect f \in L^2(\R^3;\C^3) \,:\, (\mathcal{R}_{j,\chi,\eps}^{(k)} \mathcal{G}_\eps \vect f) (y,\chi) \in L^2(Y\times Y';\C^3) \right\}
    \end{equation}
\end{definition}

In the same way that the smoothness of $\vect f$ is necessary for the definition of higher-order terms in the classical two-scale expansion, the same is true here for the definition of $\mathcal{R}_{j, \eps}^{(k)}$, in the form of an integrability-in-$\chi$ condition in $\mathcal{D}(\mathcal{R}_{j,\eps}^{(k)})$. Note that integrability-in-$\chi$ concerns were absent in the previous section, as the asymptotic procedure was carried out for fixed $\chi$. In any case, returning to Definition \ref{defn:corrector_operators_fullspace}, we point out that the full-space correctors $\mathcal{R}_{j,\eps}^{(k)}$ are densely defined:

\begin{lemma}
    If $\vect f \in L^2(\R^3;\C^3)$ has compact Fourier support, then $\vect f$ and $\widehat{\Xi}_{\eps,\mu} \vect f$ lie in $\mathcal{D}(\mathcal{R}_{j,\eps}^{(i)})$.
\end{lemma}

\begin{proof}\label{lem:domain_of_corrector_operators_fullspace}
    By assumption, $\vect f \in H^k(\R^3;\C^3)$ for all $k$, so $D^k \vect f \in L^2(\R^3;\C^3)^{3k}$ for all $k$, and thus $\mathcal{G}_\eps D^k \vect f$ belongs to $L^2(Y\times Y';\C^3)^{3k}$. By \eqref{eqn:gelfand_symgrad_formula}, this implies that $(iX_\chi)^{\otimes k}(\mathcal{G}_\eps \vect f) \in L^2(Y\times Y';\C^3)^k$ for all $k$. Since $\mathcal{R}_{j,\chi,\eps}^{(k)}$, is an operator on $L^2(Y;\C^3)$ with norm of size $\mathcal{O}(|\chi|^{j+k})$, we conclude that $\mathcal{R}_{j,\chi,\eps}^{(k)}\mathcal{G}_\eps \vect f \in L^2(Y\times Y';\C^3)$. The same is true for $\widehat{\Xi}_{\eps,\mu} \vect f$, because
    \begin{equation*}
        \| \mathcal{R}_{j,\chi,\eps}^{(k)} P_\chi P_0 \mathcal{G}_\eps \vect f \|_{L^2(Y\times Y';\C^3)} 
        \leq \| \mathcal{R}_{j,\chi,\eps}^{(k)} \| \cancel{\| P_\chi \| \| P_0 \| \|} \mathcal{G}_\eps \vect f \|
        \leq C \| (iX_\chi)^{\otimes (j+k)} \mathcal{G}_\eps \vect f \|_{L^2(Y\times Y';\C^3)^{j+k}} < \infty. \qedhere
    \end{equation*}
\end{proof}

\subsection{\texorpdfstring{$L^2 \rightarrow L^2$ estimates}{L2 to L2 estimates}}\label{sect:proof_l2tol2}

Recall the constant $\mu(\nu, C_\text{Fourier})$ \eqref{eqn:constant_mu_smallfreq} and the Bloch approximation operator $\widehat{\Xi}_{\eps,\mu}$ (Definition \ref{defn:bloch_approximation}). 
We shall now state and prove our first main result.

\begin{theorem}\label{thm:l2tol2}
    %
    If $\vect f \in L^2(\R^3;\C^3)$ is such that its Fourier transform $\mathcal{F}(\vect f)$ is supported in a compact set $K$. Then there exist $\eps_0 = \eps_0(\mu, K)>0$ such that whenever $0<\eps<\eps_0$, the following estimate holds for each $n \in \N_0$:
    \begin{align}\label{eqn:l2tol2}
        \left\| 
        \left( \mathcal{A}_\eps + I \right)^{-1} \widehat{\Xi}_{\eps,\mu} \vect f - \sum_{k=0}^n \left( \mathcal{R}_{0,\eps}^{(k)} + \mathcal{R}_{1,\eps}^{(k)} + \mathcal{R}_{2,\eps}^{(k)} \right) \widehat{\Xi}_{\eps,\mu} \vect f
        \right\|_{L^2(\R^3;\C^3)}
        \leq C^{n+1} \eps^{n+1} \| \vect f \|_{L^2(\R^3;\C^3)},
    \end{align}
    for some constant $C = C(K, \nu, C_\text{Korn}, C_\text{Fourier}) > 0$.
\end{theorem}

\begin{remark}
    Equivalently, we may reformulate Theorem \ref{thm:l2tol2} according to Assumption \ref{assump:rhs} as follows: If $\vect f_\eps = \widehat{\Xi}_{\eps,\mu}\vect f$ is the Bloch approximation of $\vect f \in L^2(\R^3;\C^3)$ where $\supp{(\mathcal{F}(\vect f))}$ is compact, then 
    \begin{align}
        \left\| 
        \left( \mathcal{A}_\eps + I \right)^{-1} \vect f_\eps - \sum_{k=0}^n \left( \mathcal{R}_{0,\eps}^{(k)} + \mathcal{R}_{1,\eps}^{(k)} + \mathcal{R}_{2,\eps}^{(k)} \right) \vect f_\eps
        \right\|_{L^2}
        \leq C^{n+1} \eps^{n+1} \| \vect f \|_{L^2},
    \end{align}
    for $\eps$ small enough, and some $C(K, \nu, C_\text{Korn}, C_\text{Fourier}) > 0$.
\end{remark}

\begin{proof}[Proof of Theorem \ref{thm:l2tol2}]
    Fix $\vect f \in L^2$ with $\supp{(\mathcal{F}(\vect f))} \subset K$, for some compact $K$. Then Proposition \ref{prop:compact_fourier_support_properties} applies with $\mu = \mu(\nu, C_\text{Fourier})$ as in \eqref{eqn:constant_mu_smallfreq}, to give the constants $c_\text{supp}(K)$, $\eps_0(\mu, K)>0$. In what follows, we shall assume that $\eps$ satisfies
    \begin{align}
        0 < \eps < \eps_0.
    \end{align}
    To keep the notation compact, we write $\mathcal{G}_\eps \widehat{\Xi}_{\eps,\mu} \vect f = \widetilde{\vect f}$ for the Gelfand transform of its Bloch approximation.
    
    We shall now proceed with a similar strategy to \cite[Theorem 6.6]{simplified_method}: We first collect our estimates on the fiber space (\textbf{Step 1}), then we pull back the estimates to the full space (\textbf{Step 2}).

    \textbf{Step 1: Estimates on $L^2(Y;\C^3)$.} For each $\chi \in Y'$, decompose the resolvent of $\frac{1}{\eps^2} \mathcal{A}_\chi$ as follows:
    \begin{align}
        \left( \frac{1}{\eps^2} \mathcal{A}_\chi + I \right)^{-1}
        = P_\chi \left( \frac{1}{\eps^2} \mathcal{A}_\chi + I \right)^{-1} P_\chi
        + (1-P_\chi) \left( \frac{1}{\eps^2} \mathcal{A}_\chi + I \right)^{-1} (1-P_\chi).
    \end{align}
    By assumption on $\vect f$, we have that
    \begin{align}\label{eqn:cauchy_integral_achieps}
        \left( \frac{1}{\eps^2} \mathcal{A}_\chi + I \right)^{-1} \widetilde{\vect f} = 
        \begin{cases}
            P_\chi \left( \frac{1}{\eps^2} \mathcal{A}_\chi + I \right)^{-1} P_\chi \widetilde{\vect f} \qquad
            &\text{for $\chi \stackrel{\text{Prop \ref{prop:compact_fourier_support_properties}}}{\in} c_\text{supp}[-\eps,\eps]^3 \stackrel{\text{Rmk \ref{rmk:rhs_alt_formula_blochapprox}}}{\subseteq} [-\mu,\mu]^3$,} \\
            0 &\text{otherwise.}
        \end{cases} 
    \end{align}
    Now fix $\chi \in [-\mu,\mu]^3 \setminus \{ 0 \}$. By the Cauchy integral formula with contour $\Gamma$ provided by Lemma \ref{lem:contour_existence},
    \begin{align}\label{eqn:cauchy_integral_achi}
        P_\chi \left( \frac{1}{\eps^2} \mathcal{A}_\chi + I \right)^{-1} P_\chi
        = g_{\eps,\chi}\left( \frac{1}{|\chi|^2} \mathcal{A}_\chi \right) P_{\Gamma, \frac{1}{|\chi|^2} \mathcal{A}_\chi} 
        = - \frac{1}{2\pi i} \oint_\Gamma g_{\eps,\chi}(z) \left( \frac{1}{|
        \chi|^2} \mathcal{A}_\chi - zI \right)^{-1} \, dz, 
    \end{align}
    where the notation $P_{\Gamma,\mathcal{A}}$ denotes the projection onto the eigenspace corresponding to the eigenvalues enclosed by the contour $\Gamma$. Recall also for such $\chi$, Definition \ref{defn:corrector_operators_chi_rescaled} gives
    \begin{align}\label{eqn:corrector_integral_chi}
        \mathcal{R}_{j,\chi,\eps}^{(k)}
        = - \frac{1}{2\pi i} \oint_\Gamma g_{\eps,\chi}(z) \mathcal{R}_{j,\chi}^{(k)} (z) \, dz.
    \end{align}
    
    Thus, combining \eqref{eqn:cauchy_integral_achieps}, \eqref{eqn:cauchy_integral_achi}, and \eqref{eqn:corrector_integral_chi}, we have
    %
    %
    %
    %
    \begin{align}
        &\left\| \left( \frac{1}{\eps^2} \mathcal{A}_\chi + I \right)^{-1} \widetilde{\vect f} - \sum_{k=0}^n \left( \mathcal{R}_{0,\chi,\eps}^{(k)} + \mathcal{R}_{1,\chi,\eps}^{(k)} + \mathcal{R}_{2,\chi,\eps}^{(k)} \right) \widetilde{\vect f} \right\|_{L^2(Y;\C^3)} \\
        &\qquad\leq \frac{1}{2\pi} \oint_\Gamma |g_{\eps,\chi}(z)| \left\| \left( \frac{1}{|
        \chi|^2} \mathcal{A}_\chi - zI \right)^{-1} \widetilde{\vect f} - \sum_{k=0}^n \left( \mathcal{R}_{0,\chi}^{(k)}(z) + \mathcal{R}_{1,\chi}^{(k)}(z) + \mathcal{R}_{2,\chi}^{(k)}(z) \right) \widetilde{\vect f} \right\|_{L^2(Y;\C^3)} \, dz. \\
        &\qquad\leq C_1 \oint_\Gamma \left( \max\left\{ \frac{|\chi|^2}{\eps^2},1 \right\}, 1 \right)^{-1} \bigg\| \left( \frac{1}{|
        \chi|^2} \mathcal{A}_\chi - zI \right)^{-1} \widetilde{\vect f} \nonumber\\
        &\qquad\qquad\qquad\qquad\qquad\qquad\qquad\qquad - \sum_{k=0}^n \left( \mathcal{R}_{0,\chi}^{(k)}(z) + \mathcal{R}_{1,\chi}^{(k)}(z) + \mathcal{R}_{2,\chi}^{(k)}(z) \right) \widetilde{\vect f} \bigg\|_{L^2(Y;\C^3)} \, dz. \label{eqn:apply_estimate_l2l2_chi}
    \end{align}
    In the final inequality \eqref{eqn:apply_estimate_l2l2_chi}, we have applied Lemma \ref{lem:function_g_bound}, where the constant $C_1$ depends only on $\nu$ and $C_\text{Fourier}$ (see Remark \ref{rmk:mu_rho0_dependencies}). Now combine this with Theorem \ref{thm:fibrewise_estimates} to obtain
    \begin{align}
        \eqref{eqn:apply_estimate_l2l2_chi}
        &\leq C_2^{n+1} \left( \max\left\{ \frac{|\chi|^2}{\eps^2},1 \right\}, 1 \right)^{-1} |\chi|^{n+1} \| \widetilde{\vect f} \|_{L^2(Y;\C^3)}, \label{eqn:apply_estimate_l2l2_chi_1}
    \end{align}
    where the constant $C_2$ depends only on $\nu$, $C_\text{Korn}$, and $C_\text{Fourier}$. It is important to note here that $z$ here is taken to be on the contour $\Gamma$ (see Remark \ref{rmk:z_on_contour_nice_part2}).

    Continuing from \eqref{eqn:apply_estimate_l2l2_chi_1},
    \begin{alignat}{3}
        \eqref{eqn:apply_estimate_l2l2_chi_1}
        &= C_2^{n+1} \| \widetilde{\vect f} \|_{L^2(Y;\C^3)} \times 
        \begin{cases}
            \frac{|\chi|^{n+1}}{1} \qquad\qquad\qquad
            &\text{if $\max\left\{ \tfrac{|\chi|^2}{\eps^2},1 \right\} = 1$, i.e.\, $|\chi| \leq \eps$.}\\
            \frac{|\chi|^{n+1}\eps^2}{|\chi|^2} \qquad\qquad\qquad
            &\text{if $\max\left\{ \tfrac{|\chi|^2}{\eps^2},1 \right\} = \tfrac{|\chi|^2}{\eps^2}$, i.e.\, $\eps \leq |\chi|$.}
        \end{cases} \\
        &= C_2^{n+1} \| \widetilde{\vect f} \|_{L^2(Y;\C^3)} \times
        \begin{cases}
            |\chi|^{n+1} \qquad\quad ~~
            &\text{if $|\chi| \leq \eps$.}\\
            \eps^2 |\chi|^{n-1} \qquad\quad ~~
            &\text{if $\eps \leq |\chi|$.}
        \end{cases} \label{eqn:apply_estimate_l2l2_chi_2}\\
        &\leq C_3^{n+1} \eps^{n+1} \| \widetilde{\vect f} \|_{L^2(Y;\C^3)}, \label{eqn:l2l2_estimate_step1}
    \end{alignat}
    where in the second case of \eqref{eqn:apply_estimate_l2l2_chi_2}, we recall from \eqref{eqn:cauchy_integral_achieps} that $\chi \in c_\text{supp}[-\eps,\eps]^3$, hence $|\chi|^2 \leq 3(c_\text{supp} \eps)^2$. Thus the constant $C_3$ depends on $\nu$, $C_\text{Korn}$, $C_\text{Fourier}$, and $K$.

    \textbf{Step 2: Back to estimates on $L^2(\R^3;\C^3)$.} To bring the estimate \eqref{eqn:l2l2_estimate_step1} back to the full space, we recall from Proposition \ref{prop:pass_to_unitcell} and Definition \ref{eqn:corrector_operators_fullspace} that
    \begin{align}
        \left( \mathcal{A}_\eps + I \right)^{-1} \widehat{\Xi}_{\eps,\mu} \vect f 
        &= \mathcal{G}_\eps^\ast \left( \int_{Y'}^\oplus \left( \frac{1}{\eps^2} \mathcal{A}_\chi + I \right)^{-1} \, d\chi \right) \mathcal{G}_\eps \widehat{\Xi}_{\eps,\mu} \vect f, \label{eqn:l2l2_proof_passing_to_unit_cell_identity1}\\
        \mathcal{R}_{j,\eps}^{(k)} \widehat{\Xi}_{\eps,\mu} \vect f
        &= \mathcal{G}_\eps^\ast \left( \int_{Y'}^\oplus \mathcal{R}_{j,\chi,\eps}^{(k)} \, d\chi \right) \mathcal{G}_\eps \widehat{\Xi}_{\eps,\mu} \vect f. \label{eqn:l2l2_proof_passing_to_unit_cell_identity2}
    \end{align}
    Note that $\widehat{\Xi}_{\eps,\mu} \vect f \in \mathcal{D}( \mathcal{R}_{j,\eps}^{(k)} )$, by Lemma \ref{lem:domain_of_corrector_operators_fullspace}. Since the Gelfand transform $\mathcal{G}_\eps$ is a unitary operator from $L^2(\R^3;\C^3)$ to $L^2(Y';L^2(Y;\C^3))$, and $\| \widehat{\Xi}_{\eps,\mu} \vect f \|_{L^2} \leq \| \vect f \|_{L^2}$ (as it involves $\mathcal{G}_\eps$ and projections $P_\chi$), \eqref{eqn:l2l2_estimate_step1} implies the estimate \eqref{eqn:l2tol2}, completing the proof.
\end{proof}


\begin{remark}\label{rmk:fewer_terms_l2l2}
    For readability, we have opted to use the expansion $\sum_{k=0}^n \left( \mathcal{R}_{0,\eps}^{(k)} + \mathcal{R}_{1,\eps}^{(k)} + \mathcal{R}_{2,\eps}^{(k)} \right)$ in \eqref{eqn:l2tol2}. However, it should be noted that some terms may be removed from the $(n-1)$-th and $n$-th cycle while maintaining the same $\mathcal{O}(\eps^{n+1})$ estimate in $L^2$. In particular, we may replace $\sum_{k=0}^n \left( \mathcal{R}_{0,\eps}^{(k)} + \mathcal{R}_{1,\eps}^{(k)} + \mathcal{R}_{2,\eps}^{(k)} \right)$ by
    \begin{alignat}{3}
        &\mathcal{R}_{0,\eps}^{(0)},
        &&\text{for an } \mathcal{O}(\eps) ~ \text{estimate.} \label{eqn:fewer_terms_1}\\
        &\left( \mathcal{R}_{0,\eps}^{(0)} + \mathcal{R}_{1,\eps}^{(0)} \right) + \mathcal{R}_{0,\eps}^{(1)},
        &&\text{for an } \mathcal{O}(\eps^2) ~ \text{estimate.} \label{eqn:fewer_terms_n_geq_2_a}\\
        &\sum_{k=0}^{n-2} \left( \mathcal{R}_{0,\eps}^{(k)} + \mathcal{R}_{1,\eps}^{(k)} + \mathcal{R}_{2,\eps}^{(k)} \right)
        + \left( \mathcal{R}_{0,\eps}^{(n-1)} + \mathcal{R}_{1,\eps}^{(n-1)} \right) 
        + \mathcal{R}_{0,\eps}^{(n)}, \quad 
        &&\text{for an } \mathcal{O}(\eps^{n+1}) ~ \text{estimate, $n\geq 2$.} \label{eqn:fewer_terms_n_geq_2}
    \end{alignat}
    Correspondingly, the proof of Theorem \ref{thm:l2tol2} may be modified using a triangle inequality argument as follows: For instance, when $n \geq 2$, we write 
    \begin{align}
        \eqref{eqn:fewer_terms_n_geq_2} = \sum_{k=0}^n \left( \mathcal{R}_{0,\eps}^{(k)} + \mathcal{R}_{1,\eps}^{(k)} + \mathcal{R}_{2,\eps}^{(k)} \right) 
        - \mathcal{R}_{2,\eps}^{(n-1)} 
        - \left( \mathcal{R}_{1,\eps}^{(n)} + \mathcal{R}_{2,\eps}^{(n)} \right).
    \end{align}
    Then, \eqref{eqn:apply_estimate_l2l2_chi} may be modified accordingly to
    \begin{align}
        &C \oint_\Gamma |g_{\eps,\chi}(z)| \left\| \left( \frac{1}{|\chi|^2} \mathcal{A}_\chi - zI \right)^{-1} \widetilde{\vect f} - \sum_{k=0}^n \left( \mathcal{R}_{0,\chi}^{(k)}(z) + \mathcal{R}_{1,\chi}^{(k)}(z) + \mathcal{R}_{2,\chi}^{(k)}(z) \right) \widetilde{\vect f} \right\|_{L^2(Y;\C^3)} dz \nonumber\\
        &\qquad + C \oint_\Gamma |g_{\eps,\chi}(z)| \left\| \left( \mathcal{R}_{2,\chi}^{(n-1)}(z) + \mathcal{R}_{1,\chi}^{(n)}(z) + \mathcal{R}_{2,\chi}^{(n)}(z) \right) \widetilde{\vect f} \right\|_{L^2(Y;\C^3)} dz.
    \end{align}
    By Proposition \ref{prop:fibrewise_terms_size}, $\| \mathcal{R}_{j,\chi}^{(k)}(z) \widetilde{\vect f} \|_{L^2} = \| \vect u_j^{(k)} \|_{L^2} = \mathcal{O}(|\chi|^{j+k})$, so the norm in the second term is of size $\mathcal{O}(|\chi|^{n+1})$. The rest of the proof proceeds as before.
\end{remark}

\subsection{\texorpdfstring{$L^2 \rightarrow H^1$ estimates}{L2 to H1 estimates}}\label{sect:proof_l2toh1}

Building on Theorem \ref{thm:l2tol2}, we shall now explain how a higher-order homogenization result may be obtained in the $H^1$ sense.

\begin{theorem}\label{thm:l2toh1}
    If $\vect f \in L^2(\R^3;\C^3)$ is such that its Fourier transform $\mathcal{F}(\vect f)$ is supported in a compact set $K$. Then there exist $\eps_0 = \eps_0(\mu, \vect f)>0$ such that whenever $0<\eps<\eps_0$, the following estimate hold for each $n \in \N_0$:
    \begin{align}\label{eqn:l2toh1}
        \left\| 
        \left( \mathcal{A}_\eps + I \right)^{-1} \widehat{\Xi}_{\eps,\mu} \vect f - \sum_{k=0}^n \left( \mathcal{R}_{0,\eps}^{(k)} + \mathcal{R}_{1,\eps}^{(k)} + \mathcal{R}_{2,\eps}^{(k)} \right) \widehat{\Xi}_{\eps,\mu} \vect f
        \right\|_{H^1(\R^3;\C^3)}
        \leq C^{n+1} \eps^{n} \| \vect f \|_{L^2(\R^3;\C^3)},
    \end{align}
    for some constant $C = C(K, \nu, C_\text{Korn}, C_\text{Fourier}) > 0$.
\end{theorem}

\begin{proof}
    All the operators involved have their images contained in $H^1$ on their respective spaces. Thus it makes sense to measure the error in the $H^1$ norm. Fix $\vect f$ with $\supp{(\mathcal{F}\vect f)} \subset K$, for some compact $K$. As in the proof of Theorem \ref{thm:l2tol2}, we have constants $c_\text{supp}(K)$, $\eps_0(\mu,K)>0$, and we shall assume henceforth that $0<\eps<\eps_0$. Also, we shall write $\mathcal{G}_\eps \widehat{\Xi}_{\eps,\mu} \vect f = \widetilde{\vect f}$.
    
    Following \cite[Theorem 6.14]{simplified_method}, we shall split our discussion on the fiber space into two cases, namely, the $L^2$-norm and the $L^2$-norm of the gradients (\textbf{Step 1}), then we shall study how these two cases are pulled back to the full space (\textbf{Step 2}).

    \textbf{Step 1: Estimates on $L^2(Y;\C^3)$.} For the $L^2$-norm, recall from \eqref{eqn:l2l2_estimate_step1} that
    \begin{align}
        \left\| \left( \frac{1}{\eps^2} \mathcal{A}_\chi + I \right)^{-1} \widetilde{\vect f} - \sum_{k=0}^n \left( \mathcal{R}_{0,\chi,\eps}^{(k)} + \mathcal{R}_{1,\chi,\eps}^{(k)} + \mathcal{R}_{2,\chi,\eps}^{(k)} \right) \widetilde{\vect f} \right\|_{L^2(Y;\C^3)}
        \leq C^{n+1} \eps^{n+1} \| \widetilde{\vect f} \|_{L^2(Y;\C^3)}. \label{eqn:l2h1_proof_step1_l2norm}
    \end{align}
    We now turn to the $L^2$-norm of the gradients. Then,
    \begin{align}
        &\left\| 
        \nabla P_\chi \left( \frac{1}{\eps^2} \mathcal{A}_\chi + I \right)^{-1} P_\chi \widetilde{\vect f} 
        - \nabla \sum_{k=0}^n \left( \mathcal{R}_{0,\chi,\eps}^{(k)} + \mathcal{R}_{1,\chi,\eps}^{(k)} + \mathcal{R}_{2,\chi,\eps}^{(k)} \right) \widetilde{\vect f} 
        \right\|_{L^2(Y;\C^{3\times 3})} \\
        &\qquad = \frac{1}{2\pi} \left\| \nabla \oint_\Gamma g_{\eps,\chi}(z) \left[ \left( \frac{1}{|\chi|^2} \mathcal{A}_\chi - zI \right)^{-1} \widetilde{\vect f}
        - \sum_{k=0}^n \left( \mathcal{R}_{0,\chi}^{(k)} (z) + \mathcal{R}_{1,\chi}^{(k)} (z) + \mathcal{R}_{2,\chi}^{(k)} (z) \right) \widetilde{\vect f} \right] \right\|_{L^2} \, dz \\
        &\qquad\qquad\qquad\qquad \text{by \eqref{eqn:cauchy_integral_achi} and \eqref{eqn:corrector_integral_chi}} \nonumber\\
        &\qquad = \frac{1}{2\pi} \left\| \oint_\Gamma g_{\eps,\chi}(z) \left[ \nabla \left( \frac{1}{|\chi|^2} \mathcal{A}_\chi - zI \right)^{-1} \widetilde{\vect f} 
        - \nabla \sum_{k=0}^n \left( \mathcal{R}_{0,\chi}^{(k)} (z) + \mathcal{R}_{1,\chi}^{(k)} (z) + \mathcal{R}_{2,\chi}^{(k)} (z) \right) \widetilde{\vect f} \right] \right\|_{L^2} \, dz \\
        &\qquad\qquad\qquad\qquad \text{swap $\nabla$ with the integral (which is a finite sum).} \nonumber \\
        &\qquad\leq \frac{1}{2\pi} \oint_\Gamma |g_{\eps,\chi}(z)| \left\| \nabla \left( \frac{1}{|
        \chi|^2} \mathcal{A}_\chi - zI \right)^{-1} \widetilde{\vect f} - \nabla \sum_{k=0}^n \left( \mathcal{R}_{0,\chi}^{(k)}(z) + \mathcal{R}_{1,\chi}^{(k)}(z) + \mathcal{R}_{2,\chi}^{(k)}(z) \right) \widetilde{\vect f} \right\|_{L^2} \, dz \\
        &\qquad \leq C^{n+1} \left( \max\left\{ \frac{|\chi|^2}{\eps^2}, 1 \right\} \right)^{-1} |\chi|^{n+1} \| \widetilde{\vect f} \|_{L^2(Y;\C^3)} \\
        &\qquad\qquad\qquad\qquad \text{by Theorem \ref{thm:fibrewise_estimates} and Lemma \ref{lem:contour_existence}.} \nonumber \\
        &\qquad\leq C^{n+1} \eps^{n+1} \| \widetilde{\vect f} \|_{L^2}, \label{eqn:l2h1_proof_step1_l2normgrad}
    \end{align}
    where in the final step, we recall from Proposition \ref{eqn:fourier_inversion_compact_support} that $\widetilde{\vect f}$ is supported on $\chi \in c_\text{supp}[-\eps,\eps]^3$ (see also \eqref{eqn:l2l2_estimate_step1}). As in Theorem \ref{thm:l2tol2}, the constant $C$ depends on $\nu$, $C_\text{Korn}$, $C_\text{Fourier}$, and $K$.

    \textbf{Step 2: Back to estimates on $L^2(\R^3;\C^3)$.} For the $L^2$-norm, we have shown in \eqref{eqn:l2h1_proof_step1_l2norm} that the error is of order $\mathcal{O}(\eps^{n+1})$. As the Gelfand transform is unitary from $L^2(\R^3;\C^3)$ to $L^2(Y';L^2(Y;\C^3))$, the same $\mathcal{O}(\eps^{n+1})$ error holds on $L^2(\R^3;\C^3)$.

    Turning our attention to the $L^2$-norm of the gradients, we note that the $\mathcal{O}(\eps^{n+1})$ error on $L^2(Y;\C^3)$ becomes an $\mathcal{O}(\eps^{n})$ error on $L^2(\R^3;\C^3)$. That is,
    \begin{align}
        \bigg\| \nabla \left( \mathcal{A}_\eps + I \right)^{-1} \widehat{\Xi}_{\eps,\mu} \vect f - \nabla \sum_{k=0}^n \left( \mathcal{R}_{0,\eps}^{(k)} + \mathcal{R}_{1,\eps}^{(k)} + \mathcal{R}_{2,\eps}^{(k)} \right) \widehat{\Xi}_{\eps,\mu} \vect f \bigg\|_{L^2(\R^3;\C^{3\times 3})} \leq C^{n+1} \eps^n \| \vect f \|_{L^2(\R^3;\C^3)}.
    \end{align}
    The fact that one power of $\eps$ is lost is a consequence of the identity \eqref{eqn:gelfand_symgrad_formula}. Since the argument is exactly the same as in the proof of \cite[Theorem 6.14 (see Step 2, under ``$L^2$-norm of the gradient")]{simplified_method}, we omit the details. For the reader's convenience, we summarize the key points of this argument in Appendix \ref{appendix:h1_lose_eps}.

    \textbf{Conclusion.} We have shown that the $L^2$-norm contributes to an $\mathcal{O}(\eps^{n+1})$ estimate, and the $L^2$-norm of the gradients contributes to an $\mathcal{O}(\eps^{n})$ estimate. This concludes the proof.
\end{proof}

\begin{remark}\label{rmk:fewer_terms_l2h1}
    Once again, we note that some terms in the expansion $\sum_{k=0}^n \left( \mathcal{R}_{0,\eps}^{(k)} + \mathcal{R}_{1,\eps}^{(k)} + \mathcal{R}_{2,\eps}^{(k)} \right)$ in \eqref{eqn:l2toh1} may be removed from the $(n-1)$-th and $n$-th cycle while maintaining the $\mathcal{O}(\eps^n)$ estimate in the $H^1$ norm. In particular, we may take
    \begin{alignat}{3}
        &\left( \mathcal{R}_{0,\eps}^{(0)} + \mathcal{R}_{1,\eps}^{(0)} \right),
        &&\text{for an } \mathcal{O}(\eps^1) ~ \text{estimate.} \label{eqn:fewer_terms_h1_1}\\
        &\sum_{k=0}^{n-2} \left( \mathcal{R}_{0,\eps}^{(k)} + \mathcal{R}_{1,\eps}^{(k)} + \mathcal{R}_{2,\eps}^{(k)} \right)
        + \left( \mathcal{R}_{0,\eps}^{(n-1)} + \mathcal{R}_{1,\eps}^{(n-1)} \right) \qquad 
        &&\text{for an } \mathcal{O}(\eps^{n}) ~ \text{estimate, $n\geq 2$.}
    \end{alignat}
    That is, in the expansions \eqref{eqn:fewer_terms_n_geq_2_a}-\eqref{eqn:fewer_terms_n_geq_2}, we may further omit the final term $\mathcal{R}_{0,\eps}^{(n)}$. This is because $\mathcal{R}_{0,\chi}^{(n)}(z) \widetilde{\vect f} = \vect u_0^{(n)}$ is constant in $y$, and hence does not play a role in the discussion of the $L^2$-norm of gradients. See for instance, \cite[Theorem 6.14]{simplified_method}.

    Moreover, we note that the expansion \eqref{eqn:fewer_terms_h1_1} recovers the classical $L^2\rightarrow H^1$ homogenization result. Indeed, the $\mathcal{O}(\eps)$ error was shown as \cite[Theorem 6.14, 6.17]{simplified_method} under the operator asymptotic approach (and also using various other approaches, e.g. \cite{birman_suslina_2007_l2h1}). Moreover it was shown in \cite[Sect. 6.6]{simplified_method} that
    \begin{align}
        \mathcal{R}_{1,\eps}^{(0)} \vect f (x) = \eps \mathbb{N}\left( \frac{x}{\eps} \right) : \simgrad_x \vect u_0 (x), \qquad 
        \text{for a.e. $x \in \R^3$.}
    \end{align}
    where $\mathbb{N}(y) = \mathbf{N}_1 \vect e_1 + \cdots + \mathbf{N}_6 \vect e_6 \in H^1_\#(Y;\R^3) \otimes \R^{3\times 3}_{\text{sym}}$, $\{ \vect e_1,\cdots, \vect e_6 \}$ is an orthonormal basis for $\R_\text{sym}^{3\times 3}$, and $\mathbf{N}_i$ solves the usual cell-problem with $\vect e_i$. That is, the object $\mathcal{R}_{1,\eps}^{(0)}$ arising from the operator asymptotic approach coincides with the $O(\eps)$ oscillatory term of the classical two-scale expansion. 
\end{remark}

\section{Acknowledgements}
YSL is supported by NSF grants NSF DMS-2246031 and NSF DMS-2052572.
JZ is supported by the Croatian Science Foundation under the project number HRZZ-IP-2022-10-5181

\appendix
\renewcommand{\thesubsection}{\Alph{subsection}} 
\counterwithin{theorem}{section} 

\section{Some useful estimates}\label{appendix:useful_estimates}
In this appendix, we recall several useful estimates from \cite{simplified_method}.

\begin{proposition}
     There exists a constant $C_{\text{Korn}}>0$ such that we have the following estimate:
\begin{equation}
\label{korninequality33}
    \lVert \vect u \rVert_{H^1(Y;\C^3)} \leq C_{\text{Korn}} \left\lVert \simgrad \vect u \right\rVert_{L^2(Y;\C^{3\times 3})},
\end{equation}
for every $\vect u \in H^1_\#(Y;\C^3)$ satisfying $\int_Y \vect u = 0$.
\end{proposition}

\begin{lemma}\label{lem:symrk1}
    There exists $C>0$ such that for all $\vect a, \vect b \in \C^3$ we have:
    \begin{equation}
    \label{rankonesymformula}
         \lvert \vect a \otimes \vect b\rvert 
         \leq 2 \lvert \sym( \vect a \otimes \vect b )\rvert.
    \end{equation}
\end{lemma}
\begin{proof}
   Let $\vect a = (a_1,a_2,a_3), \vect b=(b_1,b_2,b_3) \in \C^3$. The following calculation proves the claim: 
   \begin{align*}
       \lvert \sym(\vect a \otimes \vect b )\rvert^2 
       &= \sum_{i,j}\left( \frac{a_ib_j+a_jb_i}{2}\right)^2 = \sum_{i}a_i^2 b_i^2 + \sum_{i \neq j} \left( \frac{a_ib_j+a_jb_i}{2}\right)^2 \\ 
       &= \sum_i a_i^2b_i^2 + \frac{1}{2} \sum_{i \neq j}a_i^2 b_j^2 + \sum_{i \neq j}a_ib_ja_jb_i \\
       &= \frac{1}{2}\sum_{i} a_i^2b_i^2 + \frac{1}{4}\sum_{i \neq j}a_i^2 b_j^2 + \frac{1}{8} \sum_{i \neq j}(a_ib_j + a_jb_i)^2 + \frac{1}{8}\sum_{i \neq j}(a_ib_i + a_jb_j)^2 \\
       &\geq \frac{1}{2}\sum_{i} a_i^2b_i^2 + \frac{1}{4}\sum_{i \neq j}a_i^2 b_j^2 \geq \frac{1}{4} \sum_{i,j}a_i^2 b_j^2 = \frac{1}{4} \lvert \vect a \otimes \vect b\rvert^2. \qedhere
   \end{align*}
\end{proof}

\begin{proposition}\label{prop:coercive_est}
    There exists a constant $C_{\text{Fourier}}>0$ such that for every $\vect u \in H_\#^1(Y;\C^3)$, $\chi \in Y'\setminus\{0\}$ we have the following estimates: 
    \begin{equation}
        \left\lVert \vect u \right\rVert_{L^2(Y;\C^3)} \leq \frac{C_{\text{Fourier}}}{|\chi|}\left\lVert \left(\simgrad + iX_{\chi}\right) \vect u \right\rVert_{L^2(Y;\C^{3 \times 3})}.
    \end{equation}
    \begin{equation}
        \left\lVert \nabla \vect u \right\rVert_{L^2(Y;\C^{3 \times 3})} \leq C_{\text{Fourier}}\left\lVert \left(\simgrad + iX_{\chi}\right) \vect u \right\rVert_{L^2(Y;\C^{3 \times 3})}.
    \end{equation}
    \begin{equation}
        \left\lVert \vect u - \fint_Y \vect u\right\rVert_{L^2(Y;\C^3)} \leq C_{\text{Fourier}}\left\lVert \left(\simgrad + iX_{\chi}\right) \vect u \right\rVert_{L^2(Y;\C^{3 \times 3})}.
    \end{equation}
\end{proposition}

\begin{proposition}
    \label{prop:abstract_ineq}
    Assume that $\lambda \in \rho(\mathcal{A})$, $\mathcal{R}$ is a bounded linear functional on $(X, \langle \cdot, \cdot \rangle_X)$, and $\vect u \in X$ solves the problem
    \begin{equation}
    \label{abstractweakresolventproblem}
        a(\vect u, \vect v) - \lambda \langle \vect u,\vect v \rangle_H = \mathcal{R}(\vect v), \quad \forall \vect v \in X.
    \end{equation}
    Then the following inequality holds:
    \begin{equation}
    \label{eqn:absract_ineq}
        \lVert \vect u\rVert_X \leq 4 \max \left\{1, \frac{|\lambda + 1|}{\rm{dist}(\lambda, \sigma(\mathcal{A}))} \right\} \| \mathcal{R} \|_{X^*}.
    \end{equation}
\end{proposition}

\section{Pulling back \texorpdfstring{$H^1$}{H1} estimates to the full space} \label{appendix:h1_lose_eps}
In this appendix, we summarize the key points in the proof of \cite[Theorem 6.14, Step 2 of proof]{simplified_method}. Namely, the claim that the estimate from Step 1,
\begin{align}\label{eqn:l2h1_proof_step1_l2normgrad_appendix}
    &\bigg\|
        \bigg( \nabla_y \underbrace{\left( \frac{1}{\eps^2} \mathcal{A}_\chi + I \right)^{-1}}_{=\mathbb{X}} \mathcal{G}_\eps \widehat{\Xi}_{\eps,\mu} \vect f \bigg)(\cdot, \chi) 
        - \bigg( \nabla_y \sum_{k=0}^n \underbrace{\left( \mathcal{R}_{0,\chi,\eps}^{(k)} + \mathcal{R}_{1,\chi,\eps}^{(k)} + \mathcal{R}_{2,\chi,\eps}^{(k)} \right)}_{=\mathbb{Y}} \mathcal{G}_\eps \widehat{\Xi}_{\eps,\mu} \vect f  \bigg)(\cdot,\chi)
    \bigg\|_{L^2(Y;\C^{3\times 3})} \nonumber\\
    &\quad \leq C^n \eps^{n+1} \| \mathcal{G}_\eps \widehat{\Xi}_{\eps,\mu} \vect f (\cdot,\chi) \|_{L^2(Y;\C^3)},
\end{align} 
implies the estimate of gradients on the full space
\begin{align}\label{eqn:l2h1_proof_step2_l2normgrad_appendix}
    \bigg\| \nabla \underbrace{\left( \mathcal{A}_\eps + I \right)^{-1}}_{= \mathcal{X}} \widehat{\Xi}_{\eps,\mu} \vect f - \nabla \underbrace{\sum_{k=0}^n \left( \mathcal{R}_{0,\eps}^{(k)} + \mathcal{R}_{1,\eps}^{(k)} + \mathcal{R}_{2,\eps}^{(k)} \right)}_{= \mathcal{Y}} \widehat{\Xi}_{\eps,\mu} \vect f \bigg\|_{L^2(\R^3;\C^{3\times 3})} \leq C^{n+1} \eps^n \| \vect f \|_{L^2(\R^3;\C^3)}.
\end{align}

To begin, we use the identity \eqref{eqn:gelfand_symgrad_formula} to obtain
\begin{align}
    &\| \mathcal{G}_\eps [\nabla (\mathcal{X} - \mathcal{Y}) \widehat{\Xi}_{\eps,\mu} \vect f ] \|_{L^2(Y\times Y'; \C^{3\times 3})} \nonumber\\
    &\quad \leq \frac{1}{\eps} \left\| \nabla_y (\mathcal{G}_\eps [\mathcal{X} - \mathcal{Y}] \widehat{\Xi}_{\eps,\mu} \vect f) \right\|_{L^2(Y\times Y'; \C^{3\times 3})} + \frac{1}{\eps} \left\| (\mathcal{G}_\eps [\mathcal{X} - \mathcal{Y}] \widehat{\Xi}_{\eps,\mu} \vect f) \otimes \chi \right\|_{L^2(Y\times Y'; \C^{3\times 3})}.\label{eqn:l2h1_proof_gelfandscaling}
\end{align}

Consider the first term in the RHS of \eqref{eqn:l2h1_proof_gelfandscaling}. We have
\begin{align}
    &\frac{1}{\eps^2} \| \nabla_y \mathcal{G}_\eps [\mathcal{X} - \mathcal{Y}] \widehat{\Xi}_{\eps,\mu} \vect f \|_{L^2(Y\times Y'; \C^{3\times 3})}^2 
    = \frac{1}{\eps^2} \int_{Y'} \| \left( \nabla_y \mathcal{G}_\eps [\mathcal{X} - \mathcal{Y}] \widehat{\Xi}_{\eps,\mu} \vect f \right) (\cdot,\chi) \|_{L^2(Y;\C^{3\times 3})}^2 d\chi \nonumber\\
    &\leq \frac{1}{\eps^2} C^{2n+2} \eps^{2n+2} \int_{Y'} \| \mathcal{G}_\eps \widehat{\Xi}_{\eps,\mu} \vect f (\cdot, \chi) \|_{L^2(Y;\C^3)}^2 d\chi
    = C^{2n+2} \eps^{2n} \| \mathcal{G}_\eps \widehat{\Xi}_{\eps,\mu} \vect f \|_{L^2(Y\times Y';\C^3)}^2
    \leq C^{2n+2} \eps^{2n} \| \vect f \|_{L^2(\R^3;\C^3)}^2, \label{eqn:l2h1_step2_final1}
\end{align}
where the first inequality is an application of \eqref{eqn:l2h1_proof_step1_l2normgrad_appendix} and the identities \eqref{eqn:l2l2_proof_passing_to_unit_cell_identity1}-\eqref{eqn:l2l2_proof_passing_to_unit_cell_identity2}. The second inequality follows by noting that $\mathcal{G}_\eps$ is unitary and $\widehat{\Xi}_{\eps,\mu}$ involves projections.

Next, consider the second term in the RHS of \eqref{eqn:l2h1_proof_gelfandscaling}. We have
\begin{align}
    &\left\| (\mathcal{G}_\eps [\mathcal{X} - \mathcal{Y}] \widehat{\Xi}_{\eps,\mu} \vect f) \otimes \tfrac{\chi}{\eps} \right\|_{L^2(Y\times Y'; \C^{3\times 3})} \nonumber\\
    &\qquad\leq C \left\| \int_{Y'}^\oplus (\mathbb{X} - \mathbb{Y}) d\chi \otimes \tfrac{\chi}{\eps} \right\|_{L^2(Y\times Y';\C^3)\rightarrow L^2(Y\times Y';\C^{3\times 3})} \| \mathcal{G}_\eps \widehat{\Xi}_{\eps,\mu} \vect f \|_{L^2(Y\times Y';\C^3)} \nonumber\\
    &\qquad \leq C^{n+1} \eps^{n+1} \| \mathcal{G}_\eps \widehat{\Xi}_{\eps,\mu} \vect f \|_{L^2(Y\times Y';\C^3)}
    \leq  C^{n+1} \eps^{n+1} \| \vect f \|_{L^2(\R^3;\C^3)}, \label{eqn:l2h1_step2_final2}
\end{align}
where the first inequality is due to \eqref{eqn:l2l2_proof_passing_to_unit_cell_identity1} and \eqref{eqn:l2l2_proof_passing_to_unit_cell_identity2}, and the second due to  \eqref{eqn:l2h1_proof_step1_l2normgrad_appendix}. The conclusion follows by combining \eqref{eqn:l2h1_step2_final1} and \eqref{eqn:l2h1_step2_final2}.

\renewcommand{\bibname}{References} 
\printbibliography
\addcontentsline{toc}{section}{\refname} 

\end{document}